\documentclass{birkjour}
\PassOptionsToPackage{linktocpage}{hyperref}

\usepackage{amssymb}
\usepackage{float}
\usepackage{mathabx}
\usepackage{mathscinet}
\usepackage{mathrsfs}
\usepackage{hyperref}
\usepackage{tikz}
\usetikzlibrary{decorations.markings,arrows}

\usepackage[hide]{ed}
\newtheorem{thm}{Theorem}[section]
\newtheorem{cor}[thm]{Corollary}
\newtheorem{lem}[thm]{Lemma}

\theoremstyle{definition}
\newtheorem{defn}[thm]{Definition}
\theoremstyle{remark}
\newtheorem{rem}[thm]{Remark}

\numberwithin{equation}{section}

\newenvironment{enui}{
    \begin{enumerate}
        \renewcommand{\theenumi}{\alph{enumi}}
        \renewcommand{\labelenumi}{\textnormal{(\theenumi)}}
    }{\end{enumerate}}

\newenvironment{enuiap}[1]{
    \begin{enumerate}\setcounter{enumi}{#1}
        \renewcommand{\theenumi}{\arabic{enumi}}
        \renewcommand{\labelenumi}{\textnormal{\theenumi)}}
    }{\end{enumerate}}
\newenvironment{Aeqi}[1]{
    \begin{enumerate}\setcounter{enumi}{#1}
        \renewcommand{\theenumi}{\Roman{enumi}}
        \renewcommand{\labelenumi}{\textnormal{(\theenumi)}}
    }{\end{enumerate}}
\newenvironment{Aeqiq}[1]{
    \begin{enumerate}\setcounter{enumi}{#1}
        \renewcommand{\theenumi}{\Roman{enumi}$^\kqk$}
        \renewcommand{\labelenumi}{\textnormal{(\theenumi)}}
    }{\end{enumerate}}
\newenvironment{Aeqie}[1]{
    \begin{enumerate}\setcounter{enumi}{#1}
        \renewcommand{\theenumi}{\Roman{enumi}$^\kek$}
        \renewcommand{\labelenumi}{\textnormal{(\theenumi)}}
    }{\end{enumerate}}
\newenvironment{enuiard}{
    \begin{enumerate}
        \renewcommand{\theenumi}{\arabic{enumi}}
        \renewcommand{\labelenumi}{\textnormal{\theenumi.}}
    }{\end{enumerate}}

\newcommand{\beql}[1]{\begin{equation}\label{#1}}
    \newcommand{\eeq}{\end{equation}}
\newcommand{\benui}{\begin{enui}}
    \newcommand{\eenui}{\end{enui}}
\newcommand{\bproof}{\begin{proof}}
    \newcommand{\eproof}{\end{proof}}
\newcommand{\bdefn}{\begin{defn}}
    \newcommand{\bdefnl}[1]{\bdefn\label{#1}}
    \newcommand{\bdefnnl}[2]{\bdefn[#1]\label{#2}}
    \newcommand{\edefn}{\end{defn}}
\newcommand{\bthm}{\begin{thm}}
    
    \newcommand{\bthml}[1]{\bthm\label{#1}}
    \newcommand{\bthmnl}[2]{\bthm[#1]\label{#2}}
    \newcommand{\ethm}{\end{thm}}
\newcommand{\blem}{\begin{lem}}
    
    \newcommand{\bleml}[1]{\blem\label{#1}}
    \newcommand{\blemnl}[2]{\blem[#1]\label{#2}}
    \newcommand{\elem}{\end{lem}}
\newcommand{\brem}{\begin{rem}}
    \newcommand{\breml}[1]{\brem\label{#1}}
    
    \newcommand{\erem}{\end{rem}}
\newcommand{\bcor}{\begin{cor}}
    \newcommand{\bcorl}[1]{\bcor\label{#1}}
    \newcommand{\bcornl}[2]{\bcor[#1]\label{#2}}
    \newcommand{\ecor}{\end{cor}}

\newcommand{\zitaa}[2]{\cite[#2]{#1}}

\newcommand{\capp}[1]{Appendix~#1}
\newcommand{\csec}[1]{Section~#1}
\newcommand{\csubsec}[1]{Subsection~#1}
\newcommand{\cch}[1]{Chapter~#1}
\newcommand{\cthm}[1]{Theorem~#1}

\newcommand{\rcor}[1]{Corollary~\ref{#1}}
\newcommand{\rlem}[1]{Lemma~\ref{#1}}
\newcommand{\rlemss}[2]{Lemmas~\ref{#1} and~\ref{#2}}
\newcommand{\rdefn}[1]{Definition~\ref{#1}}
\newcommand{\rrem}[1]{Remark~\ref{#1}}

\newcommand{\rthm}[1]{Theorem~\ref{#1}}
\newcommand{\rthmss}[2]{Theorems~\ref{#1} and ~\ref{#2}}

\newcommand{\rsubsec}[1]{Subsection~\ref{#1}}

\newcommand{\rsubsecsd}[2]{Subsections~\ref{#1}--\ref{#2}}
\newcommand{\rsec}[1]{Section~\ref{#1}}
\newcommand{\rsecss}[2]{Sections~\ref{#1} and~\ref{#2}}
\newcommand{\rch}[1]{Chapter~\ref{#1}}
\newcommand{\rfig}[1]{Figure~\ref{#1}}

\newcommand{\ab}{\ek{a,b}}
\newcommand{\albe}{\ek{\al,\be}}

\DeclareMathOperator*{\prodl}{\overset{\curvearrowleft}{\prod}}
\DeclareMathOperator*{\prodr}{\overset{\curvearrowright}{\prod}}
\DeclareMathOperator*{\intl}{\overset{\curvearrowleft}{\int}}
\DeclareMathOperator*{\intr}{\overset{\curvearrowright}{\int}}
\DeclareMathOperator{\im}{Im}
\DeclareMathOperator{\re}{Re}
\newcommand{\reA}{\re A}
\newcommand{\imA}{\im A}
\DeclareMathOperator{\range}{\mathcal{R}}
\newcommand{\ran}[1]{\range\rk{#1}}
\DeclareMathOperator{\sign}{sign}
\newcommand{\dif}{\mathrm{d}}
\newcommand{\rk}[1]{(#1)}
\newcommand{\ek}[1]{[#1]}

\newcommand{\rkb}[1]{\left(#1\right)}
\newcommand{\ekb}[1]{\left[#1\right]}
\newcommand{\gkb}[1]{\left\{#1\right\}}
\newcommand{\iu}{\mathrm{i}}
\newcommand{\bhh}{\mathbf{h}}

\newcommand{\smat}[1]{\rkb{\begin{smallmatrix}#1\end{smallmatrix}}}
\newcommand{\mat}[1]{(#1)}
\newcommand{\ko}[1]{\overline{#1}} 
\newcommand{\ec}{\mathrm{e}}

\newcommand{\bA}{\mathbf{A}}

\newcommand{\inv}{{-1}}
\newcommand{\ad}{\ast}
\newcommand{\invad}{{-\ad}}
\newcommand{\teg}{e.\,g.}
\newcommand{\tae}{a.\,e.}
\newcommand{\absb}[1]{\left\lvert #1\right\rvert}
\newcommand{\abs}[1]{\lvert #1\rvert}
\newcommand{\clo}[1]{\overline{#1}}
\newcommand{\oc}{\bot}
\newcommand{\red}{\mathrm{red}}
\newcommand{\pri}{\mathrm{pri}}

\newcommand{\kqk}{{\rk{\mathrm{Q}}}}
\newcommand{\kek}{{\rk{\mathrm{F}}}}
\newcommand{\ocol}[5]{\rk{#1,#2;#3,#4,#5}}
\newcommand{\hgapj}{\ocol{\cH}{\cG}{A}{\Phi}{J}}

\newcommand{\At}{{\widetilde{A}}}
\newcommand{\Phit}{{\widetilde{\Phi}}}
\newcommand{\tcof}{c.\,o.\,f.}
\newcommand{\tJ}[1]{$J$\nobreakdash-#1}
\newcommand{\tG}[1]{$\cG$\nobreakdash-#1}
\newcommand{\tie}{i.\,e.}
\newcommand{\tresp}{resp.}
\newcommand{\norm}[1]{\lVert #1\rVert}
\newcommand{\normb}[1]{\left\lVert #1\right\rVert}
\newcommand{\defeq}{:=}
\newcommand{\eqdef}{=:}
\newcommand{\Mat}[1]{\begin{pmatrix}#1\end{pmatrix}}
\newcommand{\ip}[2]{(#1,#2)}

\newcommand{\ipa}[2]{\langle#1,#2\rangle}
\newcommand{\ipab}[2]{\left\langle#1,#2\right\rangle}
\newcommand{\il}[1]{\item\label{#1}}
\DeclareMathOperator{\Rstr}{Rstr}
\DeclareMathOperator{\pr}{pr}
\newcommand{\rstr}[2]{\Rstr_{#2}#1}
\newcommand{\proj}[2]{\pr_{#2}#1}
\newcommand{\set}[1]{\{#1\}}
\newcommand{\setaca}[2]{\{#1\colon #2\}}

\newcommand{\setb}[1]{\left\{#1\right\}}

\newcommand{\vphi}{\varphi}
\newcommand{\tval}[1]{#1\nobreakdash-valued}
\newcommand{\EM}{I}
\newcommand{\Iu}[1]{\EM_{#1}}
\newcommand{\IG}{\Iu{\cG}}
\newcommand{\IH}{\Iu{\cH}}
\newcommand{\IC}{\Iu{\C}}
\newcommand{\hcH}{{\widehat{\mathfrak{H}}}}
\DeclareMathOperator{\tr}{Sp}

\newcommand{\R}{\mathbb{R}}
\newcommand{\C}{\mathbb{C}}
\newcommand{\cB}{{\mathfrak{B}}}
\newcommand{\cL}{\mathfrak{L}}

\newcommand{\cE}{{\mathfrak{E}}}
\newcommand{\cEo}{{\mathfrak{E}_0}}
\newcommand{\cF}{{\mathfrak{F}}}
\newcommand{\ck}{{\mathfrak{k}}}
\newcommand{\cM}{{\mathfrak{M}}}
\newcommand{\cm}{{\mathfrak{m}}}
\newcommand{\cH}{{\mathfrak{H}}}
\newcommand{\cHt}{{\widetilde{\cH}}}
\newcommand{\cG}{{\mathfrak{G}}}

\newcommand{\cO}{{\mathfrak{O}}}
\newcommand{\cI}{{\mathfrak{I}}}
\newcommand{\RR}{\mathbb{R}}
\newcommand{\CC}{\mathbb{C}}

\newcommand{\te}{\theta}
\newcommand{\la}{\lambda}

\newcommand{\Ga}{\Gamma}
\newcommand{\ga}{\gamma}

\newcommand{\ome}{\omega}
\newcommand{\al}{\alpha}
\newcommand{\ze}{\zeta}
\newcommand{\be}{\beta}
\newcommand{\si}{\sigma}
\newcommand{\Dl}{\Delta}
\newcommand{\Om}{\Omega}
\DeclareMathOperator{\rank}{rank}

\DeclareMathOperator{\col}{col}

\begin{document}

%
%
%
%
%
%
%
%
%

\title[Characteristic function of M.~S.~Liv\v{s}ic and triangular models]{Characteristic function of M.~S.~Liv\v{s}ic and triangular models of bounded linear operators}

\author[Dubovoy]{Vladimir K. Dubovoy}
\address{Department of Mathematics and Computer Science\\
    Karazin National University\\
    Kharkov\\
    Ukraine\\
    and\\
    Max Planck Institute for Human Cognitive and Brain Sciences\\
    Stephanstrasse~1A\\
    04103~Leipzig\\
    Germany\\
    and\\
    Max Planck Institute for Mathematics in the Sciences\\
    Inselstrasse~22\\
    04103~Leipzig\\
    Germany}
\email{dubovoy.v.k@gmail.com}


\author[Kirstein]{Bernd Kirstein}
\address{Universit\"at Leipzig\\
    Fakult\"at f\"ur Mathematik und Informatik\\
    PF~10~09~20\\
    D-04009~Leipzig\\
    Germany}
\email{kirstein@math.uni-leipzig.de}

\author[M\"adler]{Conrad M\"adler}
\address{Universit\"at Leipzig\\
    Zentrum f\"ur Lehrer:innenbildung und Schulforschung\\
    Prager Stra\ss{}e 38--40\\
    D-04317~Leipzig\\
    Germany}
\email{maedler@math.uni-leipzig.de} 

\author[M\"uller]{Karsten M\"uller}
\address{Max Planck Institute for Human Cognitive and Brain Sciences\\
    Stephanstrasse~1A\\
    04103~Leipzig\\
    Germany}
\email{karstenm@cbs.mpg.de}

\subjclass{47B28; 47A45; 47A40; 47A56}

\keywords{Completely non-selfadjoint operator, operator colligation, open system, characteristic operator function, $J$-expansive operator function, multiplicative integral, triangular model of a bounded linear operator}

\date{January 1, 2004}
\dedicatory{Dedicated to M.~S.~Liv\v{s}ic, M.~S.~Brodskii, and V.~P.~Potapov.}

\begin{abstract}
This paper is dedicated to the introduction in a circle of ideas and methods, which are connected with the notion of characteristic function of a non-selfadjoint operator.
We start with the consideration of \emph{closed and open systems} (\rsubsecsd{subsubsec2.1.1-1223}{subsubsec2.1.2-0715}).
In \rsubsecsd{subsubsec2.1.2-0715}{subsubsec2.1.3-0604} we introduce the notion of operator colligation and define the characteristic function of the operator colligation as transfer function of the corresponding open system.

In \rsec{sec3-1223} we state three basic properties of the \tcof{}.
First (\rsubsec{subsec3.1-0606}), we note that the \tcof{}\ is the full unitary invariant of the operator colligation.
Second (see \rthm{T3.7}), it turns out that the invariant subspaces of the corresponding operator are associated with left divisors of the \tcof{}.
Third, the $J$\nobreakdash-property of the \tcof{}\ (see \eqref{E3.11}--\eqref{E3.13}) is a basic property which determines the class of \tcof{}\ (see \rsec{S4}).
In \rch{S4} we describe the classes of characteristic functions which play an important role in our considerations.
In \rch{sec5-0725} we state necessary facts on multiplicative integral.
\rch{sec6-0917} is devoted to the factorization theorem (\rthm{T6.6A-0925}) for matrix-valued characteristic function.
In \rch{S1551} 
we construct a triangular Liv\v{s}ic model of bounded linear operator 
and as application we obtain some known results on dissipative operators.
\end{abstract}

\maketitle
\tableofcontents
\section{Preliminaries}
\subsection{Introduction}
This paper is dedicated to the introduction in a circle of ideas and methods, which are connected with the notion of characteristic function of a non-selfadjoint operator.

We start with the consideration of \emph{closed and open systems} (Subsections 2.1.1-2).
The main idea of M.~S.~Liv\v{s}ic, which is based on physical examples, is that to non-self-adjoint operators there correspond systems interacting with the external world. Such systems are usually called \emph{open} (see, \teg{}, \cite{Liv63,MR0182516,MR0347396,MR0837339,MR1109838,L54}).

``We will designate as open systems those physical systems connected with the external world by coupling channels (or, more shortly, channels when no confusion arises)'' (see \zitaa{MR0347396}{Introduction}).

Let $\cH$ be a Hilbert space and $A\in\ek{\cH}$ be a non-self-adjoint operator.
In this case (see \zitaa{MR0634096}{Preface}), ``it has become apparent that in many questions the natural object of investigation is not the operator itself but a more complicated object known as \emph{an operator colligation} or simply \emph{a colligation}.
The space on which an operator is defined is connected in an operator colligation with another space by means of a mapping whose role is to transfer a metric.
The deep connection between operator colligations and open physical systems lies in the fact that if the motion of the system involves a change in its energy, then there must exist a space playing the role of a window through which an interaction with the external world takes place.
The value of a colligation is to be found in the fact that it explicitly contains not only the operator defining the notion but also a corresponding `window'.''

In \rsubsecsd{subsubsec2.1.2-0715}{subsubsec2.1.3-0604} we introduce the notion of operator colligation and define the characteristic function of the operator colligation as transfer function of the corresponding open system.
In \rsubsec{subsubsec2.1.2-0715} will be verified that the coupling of open systems corresponds to the product of corresponding operator colligations and the product of its characteristic functions.
\rsubsec{subsec2.3-0604} contains basic facts of the theory of operator colligations.

In Section 3 we state three basic properties of the \tcof{} 
First (\rsubsec{subsec3.1-0606}), we note that the \tcof{}\ is the full unitary invariant of the operator colligation.
Second (see \rthm{T3.7}), it turns out that the invariant subspaces of the corresponding operator are associated with left divisors of the \tcof{} 
Third, the $J$\nobreakdash-property of the \tcof{}\ (see \eqref{E3.11}--\eqref{E3.13}) is a basic property which determines the class of \tcof{}\ (see \rsec{S4}).
The discovery of this property by M.~S.~Liv\v{s}ic served as stimulation for V.~P.~Potapov's \cite{MR0076882} construction of the multiplicative theory for \tJ{expansive} (\tJ{contractive}) matrix functions.
At the end of this section (\rsubsec{subsec3.4-0606}) we show that the signature of the operator $J$ determines the directions of the outer channels.

In \rch{S4}, we describe three classes of characteristic functions, which play an important role in this paper.
\rsecss{subsec4.1}{subsec4.2-0713} are devoted to the description of class $\Om_J$ of characteristic functions of bounded linear operators.
\rsec{subsec4.5-0713} is devoted to the investigation and description of the class $\Om_J^\kqk$ of characteristic functions of bounded linear operators with compact imaginary part.
\rsecss{subsec4.6-1001}{subsec4.8-0917} handle the classes $\Om_J\rk{\C^r}$ and $\Om_J^\kek$, \tresp, of matrix-valued characteristic functions.
In \rsecss{subsec4.3-0713}{subsec4.4-0715} basic facts on regular left divisors of characteristic functions are discussed.
Namely, these divisors correspond to invariant subspaces of the corresponding operator.

From the basic second property of characteristic functions (see \rthm{T3.7}) it follows that the chain of nested invariant subspaces of the basic operator corresponds to the multiplicative decomposition of the characteristic functions in corresponding factors.
In this framework we state necessary facts on multiplicative integrals in \rch{sec5-0725}.
Here we follow the remarkable (but in our opinion not very known) paper \cite{P2} of V.~P.~Potapov.

\rch{sec6-0917} deals with the factorization theorem for matrix-valued characteristic functions, \tie, characteristic functions of the class $\Om_J(\C^r)$ (see \rthm{T6.6A-0925}).
This theorem is a special case of V.~P.~Potapov's fundamental result on the multiplicative structure of $J$\nobreakdash-contractive matrix functions (see \cite{MR0076882}).
Here following the paper \cite{MR0100793}, we present an approach to the proof of this theorem.


As an application of these results we prove the known result due to M.~S.~Liv\v{s}ic that each completely non-selfadjoint dissipative Volterra operator $A$ with one-dimensional imaginary part is unitary equivalent to the integration operator $\cI$ in $L_2(0,\ell)$ (see \rthm{T7.13-0503}).
Furthermore, we obtain the known result that the integration operator $\cI$ in $L_2(0,\ell)$ is an unicellular operator (\rthm{T7.15-0519}).
It should be mentioned that from factorization \rthm{T6.6A-0925} it follows immediately the completeness criterion for the finite-dimensional invariant subspaces of a dissipative operator with finite-dimensional imaginary part (see \rsubsec{subsubsec7.5.3-0525}).

The volume of the paper does not allow to consider many other questions connected with the theory of characteristic functions.
We mention in this connection the monographs of M.~S.~Brodskii \cite{MR0322542}, S.~R.~Garcia, J.~Mashreghi, and W.~T.~Ross \cite{MR3526203}, I.~C.~Gohberg and M.~G.Krein \cite{MR0246142,MR0264447}, M.~S.~Liv\v{s}ic \cite{MR0347396}, M.~S.~Liv\v{s}ic and A.~A.~Yantsevich \cite{MR0634096}, N.~K.~Nikolskii \cite{MR1892647},  B.~Sz.-Nagy, C.~Foias, H.~Bercovici,  and L.~Kerchy \cite{SN} as well as the works J.~A.~Ball and N.~Cohen \cite{MR1115026}, L.~de~Branges and J.~Rovnyak \cite{MR0244795}, M.~S.~Brodskii \cite{MR0510672}, where closely related questions are considered.

\subsection{Basic notations}
Let $\cH$ and $\cG$ be Hilbert spaces (all Hilbert spaces considered in this paper are assumed to be complex and separable).
By $\ek{\cH,\cG}$ we denote the Banach space of bounded linear operators defined on $\cH$ and taking values in $\cG$.
If $\cG=\cH$ we use the notation $\ek{\cH}\defeq\ek{\cH,\cH}$.
If $\cH$ is a Hilbert space and $\cL$ is a subspace of $\cH$, then by $P_\cL$ we denote the orthogonal projection from $\cH$ onto $\cL$.

If $B\in\ek{\cH,\cG}$ and $\cH_0$ is a subspace of $\cH$, then the symbol $\rstr{B}{\cH_0}$ stands for the restriction of the operator $B$ to $\cH_0$, \tie{}
\[
\rstr{B}{\cH_0}\in\ek{\cH_0,\cG}
\qquad\text{and}\qquad
\rk{\rstr{B}{\cH_0}}h_0=Bh_0,
\quad h_0\in\cH_0.
\]
The symbol $\ran{B}$ stands for the image of the operator $B$.
We preserve these notations if $\cG=\cH$, \tie\ $B\in\ek{\cH}$.

If $A\in\ek{\cH}$, then by $\re{A}$ and $\im{A}$ we denote the real part and the imaginary part of $A$, \tresp{}
\[
\re{A}\defeq\frac{1}{2}\rk{A+A^\ad},
\qquad
\im{A}\defeq\frac{1}{2\iu}\rk{A-A^\ad}.
\]
Hence,
\[
 A
 =\re{A}+\iu\im{A}.
\]
By the symbol $\rho\rk{A}$ and $\si\rk{A}$, respectively, we denote the regular points and the spectrum of the operator $A$, respectively.

The symbol $\IH$ stands for the identity operator, acting in $\cH$, \tie\ $\IH\in\ek{\cH}$ and $\IH h=h$, $h\in\cH$.





\section{Open systems and operator colligations. Characteristic functions of colligations}

\subsection{Closed and open linear systems. Operator colligations. Characteristic function as transfer function of an open system}

\subsubsection{Closed linear systems}\label{subsubsec2.1.1-1223}
``Many problems of mathematical physics lead to equations of the form
\beql{E2.1}
\iu\frac{\dif h}{\dif t}+Ah
=0
\eeq
where $A$ is a bounded operator in a Hilbert space $\cH$ and $h\rk{t}$ is a state of the corresponding system.
In physical applications the energy (or the number of particles) in the state $h$ is proportional to the scalar product $\ipa{h}{h}$'' (see \cite{MR0837339}).
From \eqref{E2.1} we get
\beql{E2.2}
\frac{\dif}{\dif t}\ipa{h}{h}
=\ipa{\frac{\dif h}{\dif t}}{h}+\ipa{h}{\frac{\dif h}{\dif t}}
=\ipa{-\iu Ah}{h}+\iu\ipa{h}{Ah}
=\ipa{\frac{1}{\iu}\rk{A-A^\ad}h}{h}.
\eeq
If $A=A^\ad$ then in this case the energy $\ip{h}{h}$ is preserved and the linear system described by equation \eqref{E2.1} is called \emph{closed}.

We present a simple example (see \cite[\cch{I}, \S1]{MR0347396}).


\ \ \ \

Let us consider an oscillator (for example, a weight suspended on a spring).
Its equation of motion is of the form
\beql{E1.1}
m\ddot{x}+kx
=0.
\eeq
Introducing the coordinates $\xi  =\sqrt{m}\dot{x}$, $\eta=\sqrt{k}x$, we obtain the system
\beql{E1.2}
\frac{\dif\xi}{\dif t}=-\omega\eta,
\qquad\frac{\dif\eta}{\dif t}=\omega\xi
\qquad\rkb{\omega\defeq\sqrt{\frac{k}{m}}}.
\eeq
Equation \eqref{E1.2} can obviously be rewritten in matrix form
\[
-\iu\frac{\dif\bhh}{\dif t}
=\bhh\bA,
\]
where $\bA=\smat{0&-\iu\omega\\\iu\omega&0}$ is a Hermitian matrix and $\bhh=\mat{\xi,\eta}$ is a two-dimensional vector (row-matrix) describing the motion of the oscillator.
(Here, the operator is defined by a matrix that acts on the row vector.
Therefore, this matrix is on the right.
We will write the abstract operators on the left of the vectors on which they act.)
If we introduce the scalar product
\[
\ipa{\bhh_1}{\bhh_2}
=\xi_1\ko{\xi_2}+\eta_1\ko{\eta_2},
\]
the expression $\ipa{\bhh}{\bhh}=m\dot{x}^2+kx^2$ is equal to twice the energy of the oscillator in the state of the system described by the vector $\bhh$.

So we designate as \emph{closed} a system the state of which may be represented by a vector in some Hilbert space $\cH$, and whose motion is determined by an equation of the form
\beql{E1.3}
-\iu\frac{\dif\bhh}{\dif t}
=A\bhh ,
\eeq
where $A$ is a self-adjoint operator in $\cH$.




There (see \cite[\cch{I}, \S1]{MR0347396}) are also given other examples of closed physical systems.

\subsubsection{Open linear systems and operator colligations}\label{subsubsec2.1.2-0715}
If the operator $A$ in system \eqref{E2.1} is non-selfadjoint then, as it follows from \eqref{E2.2}, in this case the energy $\ipa{h}{h}$ in the state $h$ is not preserved.
This means that the system interacts with the external world.
In this case, the system is called \emph{open}.
As it follows from \eqref{E2.2}, this interaction is described by the imaginary part of the operator $A$.
The notion of an operator colligation, introduced below, makes it possible to describe this interaction.



\bdefnnl{\zitaa{MR0347396}{\csubsec{1.3}}, \zitaa{MR0634096}{\cch{1}, \S3}}{D2.1}
A set of two spaces $\cH$ and $\cG$, and three operators $A\in\ek{\cH}$, $\Phi\in\ek{\cH,\cG}$ and $J\in\ek{\cG}$,
\beql{E2.8}
J
=J^\ad
=J^\inv,
\eeq
connected by the relation
\beql{E2.9}
\frac{1}{\iu}\rk{A-A^\ad}
=\Phi^\ad J\Phi,
\eeq
is called \emph{an operator colligation} or simply \emph{a colligation} and is denoted by the symbol
\beql{E2.10}
\alpha
\defeq\hgapj .
\eeq

The operator A is called \emph{the fundamental operator} of the colligation $\alpha$.
The spaces $\cH$ and $\cG$ are called \emph{internal and external spaces}, respectively, and the operators $\Phi$ and $J$ are called \emph{channeled and directing operators}, respectively.
\edefn

From \eqref{E2.9} it  follows that
\beql{E10.1}
\ran{\imA  }
\subset\ran{\Phi^\ad}.
\eeq
The subspaces $\clo{\ran{\imA  }}$ and $\clo{\ran{\Phi^\ad}}$ are called \emph{non-Hermitian} and \emph{channeled}.

\bthmnl{\zitaa{MR0322542}{\cch{I}}}{T2.2}
For any operator $A\in\ek{\cH}$ and any subspace $\cE$ containing $\ran{\imA  }$ there exists a colligation $\hgapj $ for which the subspace $\cE$ is a channeled space.
\ethm
\bproof
Let
\[
B
\defeq2\rstr{\imA  }{\ran{\imA  }}
\]
and let $\set{E_\lambda}_{\lambda=-\infty}^{+\infty}$ be the resolution of the identity for the operator $B$, \tie{},
\[
B
=\int_{-\infty}^{+\infty}\lambda\dif E_\lambda.
\]
Consider the subspace
\[
\cEo
\defeq\cE\ominus\clo{\ran{\imA  }}.
\]
Then
\beql{E10.2}
\cH
=\clo{\ran{\imA  }}\oplus\cEo\oplus\cE^\oc,
\eeq
where $\cE^\oc\defeq\cH\ominus\cE$.

Now put
\beql{E10.3}
\cG
\defeq\clo{\ran{\imA  }}\oplus\cEo\oplus\cEo
\eeq
and let $\Phi\in\ek{\cH,\cG}$ and $J\in\ek{\cG}$ have the block representations
\beql{E10.400}
\Phi
\defeq
\begin{pmatrix}
    \abs{B}^{1/2}&0&0\\
    0&\Iu{\cEo}&0\\
    0&\Iu{\cEo}&0
\end{pmatrix},
\qquad J\defeq
\begin{pmatrix}
    \sign{B}&0&0\\
    0&\Iu{\cEo}&0\\
    0&0&-\Iu{\cEo}
\end{pmatrix},
\eeq
with respect to decompositions \eqref{E10.2} and \eqref{E10.3}, where 
\[
\abs{B}^{1/2}
\defeq\int_{-\infty}^{+\infty}\abs{\lambda}^{1/2}\dif E_\lambda,
\qquad \sign{B}
\defeq\int_{-\infty}^{+\infty}\sign{\lambda}\dif E_\lambda.
\]
It is directly verified that conditions \eqref{E2.8} and \eqref{E2.9} are fulfilled, where we note that for decomposition \eqref{E10.2} the representations
\[
\abs{\imA  }^{1/2}=\frac{1}{\sqrt2}\abs{B}^{1/2}\oplus0\oplus0,
\qquad\sign \imA  =\sign B\oplus0\oplus0
\]
hold.
\eproof

\breml{R2.3-0929}
If the subspace $\cE$ is finite dimensional then one can choose the space $\C^r$ as outer space where (see \eqref{E10.2})
\[
r
\defeq\dim\ran{\im A}+2\dim\cE_0
=\dim\cG.
\]
Indeed, if $V$ is a unitary mapping from $\cG$ onto $\C^r$, $\Phi'=V\Phi$, $J'=VJV^\inv$, then the quintuple
\[
\al'
\defeq\rk{\cH,\C^r;A,\Phi',J'}.
\]
is also a colligation together with the colligation $\hgapj$.
\erem

If we take the subspace $\cE\defeq\clo{\ran{\imA}}$ in \rthm{T2.2}, then $\cG=\clo{\ran{\imA}}$, and we obtain the colligation
\beql{E1330}
\al_{\imA  }
\defeq\hgapj ,
\eeq
where $J\defeq\sign B$ and $\Phi=\abs{B}^{1/2}\oplus0$ with respect to the decomposition
\[
\cH
=\clo{\ran{\imA}}\oplus\clo{\ran{\imA  }}^\oc.
\]


The operation of representing an operator $A\in\ek{\cH}$ as the fundamental operator of a colligation is called \emph{embedding} the operator $A$ into a colligation.
In \zitaa{MR0634096}{\cch{1}} 
a description of all embeddings of a given operator $A\in\ek{\cH}$ into colligations is obtained.

Let $\cH$ be a Hilbert space and $A\in\ek{\cH}$.
Let us embed the operator $A$ into a colligation $\al$ of form \eqref{E2.10}.
Assign to the colligation $\al$ the open system $\cO_\al$ 
defined by the relations (see, \teg, \cite{L54})

\beql{E2.11}
\left\{\begin{array}{l}
    \iu\frac{\dif h}{\dif t}+Ah=\Phi^\ad J\vphi^-,\\
    h\rk{0}=h_0,\\
    \vphi^+=\vphi^- - \iu\Phi h,
\end{array}\right.
\eeq
where $t$ is a real variable.  
In these relations we will assume that the \tval{$\cG$} function $\vphi^-=\vphi^-\rk{t}$ and the vector $h_0\in\cH$ are given.
This makes possible to find the \tval{$\cH$} function $h=h\rk{t}$. 
Substituting the given function $\vphi^-=\vphi^-\rk{t}$ and the found solution $h=h\rk{t}$ 
in the third relation, we determine the \tval{$\cG$} function $\vphi^+=\vphi^+\rk{t}$.
The functions $\vphi^-\rk{t}$, $h\rk{t}$, and $\vphi^+\rk{t}$ are called \emph{the input}, \emph{the internal state}, and \emph{the output} of the open system $\cO_\al$, respectively.


Colligation condition \eqref{E2.9} allows us to control the change of the quantity $\ipa{h\rk{t}}{h\rk{t}}$ (the energy of the internal state of system \eqref{E2.11}).
Indeed, from the first equation in \eqref{E2.11} the equalities
\begin{align*}
    \iu\ipab{\frac{\dif h}{\dif t}}{h}+\ipa{Ah}{h}&=\ipa{\Phi^\ad J\vphi^-}{h},\\
    -\iu\ipab{h}{\frac{\dif h}{\dif t}}+\ipa{h}{Ah}&=\ipa{h}{\Phi^\ad J\vphi^-}
\end{align*}
follow.
From here we obtain
\[
\frac{\dif}{\dif t}\ipa{h}{h}+\ipab{\frac{A-A^\ad}{\iu}h}{h}
=-\iu\ipa{\Phi^\ad J\vphi^-}{h}+\iu\ipa{h}{\Phi^\ad J\vphi^-}.
\]
Taking into account colligation condition \eqref{E2.9}, we derive that
\beql{E2.12}
\frac{\dif}{\dif t}\ipa{h}{h}
=-\ipa{\Phi^\ad J\Phi h}{h}-\iu\ipa{\Phi^\ad J\vphi^-}{h}+\iu\ipa{h}{\Phi^\ad J\vphi^-}.
\eeq
In view of the third relation in \eqref{E2.11}, the right side of this equality can be transformed as follows:
\beql{E2.13}\begin{split}
    &-\ipa{\Phi^\ad J\Phi h}{h}-\iu\ipa{\Phi^\ad J\vphi^-}{h}+\iu\ipa{h}{\Phi^\ad J\vphi^-}\\
    &=-\iu\ipa{\Phi^\ad J\rk{\vphi^--\iu \Phi h}}{h}+\iu\ipa{\Phi h}{J\vphi^-}\\
    &=-\iu\ipa{J\vphi^+}{\Phi h}+\iu\ipa{\Phi h}{J\vphi^-}.
\end{split}\eeq
The third relation in \eqref{E2.11} can be rewritten in the form:
\[
\Phi h
=\iu \rk{\vphi^+-\vphi^-}.
\]
Therefore, from \eqref{E2.12} and \eqref{E2.13} we derive that
\[
\frac{\dif}{\dif t}\ipa{h}{h}
=-\ipa{J\vphi^+}{ \vphi^+-\vphi^-}-\ipa{ \vphi^+-\vphi^-}{J\vphi^-}.
\]
Consequently,
\beql{E2.14}
\frac{\dif}{\dif t}\ipa{h}{h}
=\ipa{J\vphi^-}{\vphi^-}-\ipa{J\vphi^+}{\vphi^+}.
\eeq
where $\ipa{J\vphi^-}{\vphi^-}$ and $\ipa{J\vphi^+}{\vphi^+}$ can be interpreted as energy flows through the input
and the output, respectively.  
This important relation is called \emph{the law of conservation of metric (energy)} for the open system $\cO_\al$ (see, \teg, \cite{MR0837339}, \cite{L54}).

Note that \zitaa{MR0347396}{Subsections~1.2,~1.3} contains physical reasonings and physical examples leading to the concept of an open system.

\subsubsection{Characteristic function as transfer function of an open system}\label{subsubsec2.1.3-0604}
Assume that in system \eqref{E2.11} $\vphi^-\rk{t}$ has the form:
\beql{E2.15}
\vphi^-\rk{t}
\defeq\ec^{\iu zt}\vphi_0^-,
\qquad\vphi_0^-\in\cG,\;z\in\CC,\;t\in\RR.
\eeq
In this case,
\beql{E2.16}
h\rk{t}=\ec^{\iu zt}h_0,
\qquad\vphi^+\rk{t}=\ec^{\iu zt}\vphi^+_0,
\qquad h_0\in\cH,\;\vphi^+_0\in\cG,\,
\eeq
and relations \eqref{E2.11} take the form:
\beql{E2.17}
\left\{\begin{array}{l}
    \rk{A-z\Iu{\cH}}h_0=\Phi^\ad J\vphi_0^-,\\
    \vphi_0^+=\vphi_0^--\iu \Phi h_0.
\end{array}\right.
\eeq

Note that relations \eqref{E2.17} completely determine the open system $\cO_\al$ of form \eqref{E2.11}. We will consider relations \eqref{E2.17} for a fixed $z$.
Under these conditions, the role of the input in these relations is played by the vector $\vphi_0^-$, and the vectors $h_0$ and $\vphi_0^+$ do the roles of the internal state and the output, respectively.

If $z\in\rho\rk{A}$, then
\beql{E2.18}
h_0=Q_{\alpha}\rk{z}\vphi_0^-,
\qquad\vphi_0^+=S_{\alpha}\rk{z}\vphi_0^-,
\eeq
where
\begin{align}
    Q_{\alpha}\rk{z}&\defeq\rk{A-z\Iu{\cH}}^\inv\Phi^\ad J,\label{E2.19}\\
    S_{\alpha}\rk{z}&\defeq\Iu{\cG}-\iu \Phi\rk{A-z\Iu{\cH}}^\inv\Phi^\ad J.\label{E2.20}
\end{align}

\bdefnnl{\cite{MR0062955}}{D2.4}
Let $\al$ be an operator colligation of form \eqref{E2.10}.
The operator function $S_{\alpha}\rk{z}, \ z\in\rho\rk{A}$, of form \eqref{E2.20} is called \emph{the characteristic operator function} (\tcof{}) of the colligation $\al$ (of the operator $A$) or \emph{the transfer function} of the open system $\cO_\al$ given by relations \eqref{E2.11}.
\edefn

Taking into account the described construction, the open system $\cO_{\al}$ can be schematically illustrated in the following way:

\begin{figure}[H]
    \centering
    \begin{tikzpicture}[decoration={
            markings,
            mark=at position 0.5 with {\arrowreversed{latex}}}
        ]
        
        \draw[postaction={decorate}] (0,0) --node[anchor=south] {$\vphi_0^+$} (1,0);
        \draw (2,0) circle (1cm);
        \node[label=center:{$h_0$}] at (2,0) {};
        \node[label=below:{$\cO_{\al}$}] at (2,-1) {};
        \draw[postaction={decorate}] (3,0) --node[anchor=south] {$\vphi_0^-$} (4,0);
    \end{tikzpicture}
    \caption{}\label{F0}
\end{figure}

In applications the inner space $\cH$ with operator $A$ acting in it is often called \emph{the intermediate system}.
The vector $h_0$ is called \emph{the state} of the intermediate system.

We note that, besides colligation \eqref{E1330}, the tuple
\beql{E1411}
\tilde\al_{\imA  }^\ad
\defeq\rk{\cH,\cG;A^\ad,J\Phi,-J}
\eeq
is also an operator colligation.
It is immediately checked that function \eqref{E2.23A} coincides with the \tcof{}\ of the colligation $\tilde\al_{\imA  }^\ad$, \tie{},
\[
W\rk{z}
=S_{\tilde\al_{\imA  }^\ad}\rk{z},
\qquad z\in\rho\rk{A^\ad}.
\]

\breml{R2.4A}
First the \tcof{}\ of a bounded linear operator in Hilbert space was introduced by M.~S.~Liv\v{s}ic in \cite{MR0062955} by the following formula:
\beql{E2.23A}
W\rk{z}
=\Iu{\ran{\imA  }}+2\rk{\sign \imA  }\abs{\imA  }^{1/2}\rk{A^\ad-z\IH}^\inv\abs{\imA  }^{1/2},
\qquad z\in\rho\rk{A^\ad},
\eeq
where the right side was considered only in the space $\ran{\imA  }$.
A more general definition was proposed by M.~S.~Brodskii (\cite{MR0080269,MR0131161,MR0131178}), see also the survey paper of M.~S.~Brodskii and M.~S.~Liv\v{s}ic \cite{MR0100793}. 
According to this definition the \tcof{}\ of operator $A\in\ek{\cH}$ could be any function of the type
\beql{E1357}
w\rk{z}
=\IG-2\iu K^\ad\rk{A-z\EM}^\inv KJ,
\qquad z\in\rho\rk{A},
\eeq
where $\cG$ some Hilbert space, $K\in\ek{\cG,\cH}$, $J\in\ek{\cG}$, and the following properties are fulfilled
\beql{E1358}
J=J^\ad,
\qquad J^2=\IG,
\qquad 
\imA =KJK^\ad.
\eeq
Note that this definition is equivalent to representation \eqref{E2.20} (see \rdefn{D2.4}).
It suffices to put $\Phi=\sqrt2K^\ad$.
Definitions \eqref{E1357} made it possible to formulate the ``multiplication theorem'' (see \rthm{T3.7}) without the previous restrictions and led naturally to the notion of operator colligation.
The construction of characteristic functions in terms of operator colligations was done in papers of M.~S.~Brodskii, Yu.~L.~Shmul\cprime jan \cite{MR0165361} and M.~S.~Brodskii, G.~E.~Kisilevskii \cite{MR0203459}.
A brief survey of the development of the concept of the characteristic function can be found in the paper of A.~V.~Kuzhel \cite{MR1299959}.
\erem

\subsubsection{Characteristic operator function.
    Relationship with M\"obius transformation}
In papers \cite{MR0062955,MR1473265} M.~S.~Liv\v{s}ic proposed the following approach to defining the characteristic operator function.

Consider the M\"obius transformation
\beql{E2.32}
\te_a\rk{z}
\defeq\frac{z-\ko a}{z-a},
\qquad\im a\neq0,
\eeq
by which the number $a$ is completely determined.
From \eqref{E2.32} we derive that
\beql{E2.33}
\te_a\rk{z}
=1-\iu\frac{a-\ko a}{\iu}\rk{a-z}^\inv
=1-\iu\sign\frac{a-\ko a}{\iu}\absb{\frac{a-\ko a}{\iu}}^{1/2}\rk{a-z}^\inv\absb{\frac{a-\ko a}{\iu}}^{1/2}.
\eeq
For the one-dimensional Hilbert space $\cH\defeq\CC$ the space $\ek{\cH}$ can be also identified with $\CC$, \tie, we will write $A_a\defeq a$, keeping in mind that
\[
A_ah\defeq ah,
\qquad h\in\cH.
\]
In this case,
\[
2\imA _a
=\frac{a-\ko a}{\iu}.
\]
Therefore, for embedding the operator $A_a$ into some colligation, it is natural to chose the external space $\cG\defeq\CC$ and the operators
\[
\Phi_a\defeq\absb{\frac{a-\ko a}{\iu}}^{1/2}\in\ek{\cH,\cG},
\qquad J_a\defeq\sign\frac{a-\ko a}{\iu}\in\ek{\cG}
\]
as the channeled and directing ones, respectively.
Thus, the 5\nobreakdash-tuple
\[
\alpha_a\defeq\rk{\cH,\cG;A_a,\Phi_a,J_a},
\qquad\cH=\cG=\CC,
\]
is a colligation.
Note that, in this case,
\[
\Phi_a^\ad
=\absb{\frac{a-\ko a}{\iu}}^{1/2}
\in\ek{\cG,\cH}.
\]
The introduced designations allow us to rewrite representation \eqref{E2.33} in the form
\beql{E2.34}
\te_a\rk{z}=\IG-\iu \Phi_a\rk{A_a-z\IH}^\inv\Phi_a^\ad J_a,
\qquad z\in\rho\rk{A_a}.
\eeq
Consequently, $S_{\al_a}\rk{z}=\te_a\rk{z}$, $z\in\rho\rk{A_a}$, \tie, M\"obius transformation \eqref{E2.32} is the characteristic function of the colligation $\al_a$.

Thus, from representation \eqref{E2.34} it is seen that the characteristic operator function of form \eqref{E2.20} in the case of an arbitrary bounded operator $A$ can be considered as a generalization of the M\"obius transformation of form \eqref{E2.33} (see, \teg, \cite{MR1473265}).

\subsection{Coupling of open systems, product of operator colligations  and factorizations of characteristic operator functions}

Let
\beql{E2.22}
\al_j\defeq\rk{\cH_j,\cG;A_j,\Phi_j,J},
\qquad j=1,2,
\eeq
be operator colligations and let $S_{\al_j}\rk{z}$ be the \tcof{}\ of the colligation $\al_j$ ($j=1,2$).
Note that the external space and directing operator are common to theses colligation.
Consider realizations (see \eqref{E2.17})
\begin{align}
    &\left\{\begin{array}{l}
        \rk{A_1-z\Iu{\cH_1}}h_{10}=\Phi_1^\ad J\vphi_{10}^-,\\
        \vphi_{10}^+=\vphi_{10}^--\iu \Phi_1h_{10}
    \end{array}\right.\label{E2.23B}
    \intertext{and}
    &\left\{\begin{array}{l}
        \rk{A_2-z\Iu{\cH_2}}h_{20}=\Phi_2^\ad J\vphi_{20}^-,\\
        \vphi_{20}^+=\vphi_{20}^--\iu \Phi_2h_{20}.
    \end{array}\right.\label{E2.24}
\end{align}
Here for each $z\in\rho\rk{A_j}$
\beql{E2.25}
\vphi_{j0}^-,\vphi_{j0}^+\in\cG,
\qquad h_{j0}\in\cH_j,
\qquad\vphi_{j0}^+=S_{\al_j}\rk{z}\vphi_{j0}^-
\qquad(j=1,2).
\eeq
If for each $z\in\rho\rk{A_1}\cap\rho\rk{A_2}$ we put
\beql{E2.26}
\vphi_{10}^-
=\vphi_{20}^+,
\eeq
then we obtain new relations of type \eqref{E2.17} with the input $\vphi_{20}^-$, the internal state $h_0\defeq h_{10}+h_{20}$ and the output $\vphi_{10}^+$.
Really, substituting the expression for $\vphi_{20}^+$ from \eqref{E2.24} in relations \eqref{E2.23B} instead of $\vphi_{10}^-$, we come to the relations
\beql{E2.27}
\left\{\begin{array}{l}
    \begin{aligned}\rk{A_1-z\Iu{\cH_1}}h_{10}+\iu\Phi_1^\ad J\Phi_2h_{20}&=\Phi_1^\ad J\vphi_{20}^-,\\
        \rk{A_2-z\Iu{\cH_2}}h_{20}&=\Phi_2^\ad J\vphi_{20}^-,
    \end{aligned}\\
    \vphi_{10}^+=\vphi_{20}^--\iu\rk{\Phi_1h_{10}+\Phi_2h_{20}}
\end{array}\right.
\eeq
or
\beql{E2.28}
\left\{\begin{array}{l}
    \rk{A-z\Iu{\cH}}h_{0}=\Phi^\ad J\vphi_{20}^-,\\
    \vphi_{10}^+=\vphi_{20}^--\iu \Phi h_{0},
\end{array}\right.
\eeq
where
\beql{E2.29}
\cH\defeq\cH_1\oplus\cH_2,
\quad A\defeq\Mat{A_1&\iu\Phi_1^\ad J\Phi_2\\0&A_2}\in\ek{\cH},
\quad \Phi\defeq\ek{\Phi_1,\Phi_2}\in\ek{\cH,\cG}.
\eeq
Note that
\[
\frac{1}{\iu}\rk{A-A^\ad}
=\Mat{
    \frac{1}{\iu}\rk{A_1-A_1^\ad}&\Phi_1^\ad J\Phi_2\\
    \Phi_2^\ad J\Phi_1&\frac{1}{\iu}\rk{A_2-A_2^\ad}
},
\]
whence, taking into account the equalities
\[
\frac{1}{\iu}\rk{A_j-A_j^\ad}=\Phi_j^\ad J\Phi_j
\qquad(j=1,2),
\]
we obtain
\[
\frac{1}{\iu}\rk{A-A^\ad}
=\Phi^\ad J\Phi.
\]
Thus, the 5-tuple
\beql{E2.30}
\al
\defeq\hgapj ,
\eeq
where the space $\cH$ and the operators $A$ and $\Phi$ are given by \eqref{E2.29}, is an operator colligation.

Equivalently, we can say that the open system $\cO_\al$ corresponding to the colligation $\al$ of form \eqref{E2.30} is uniquely determined by the open systems $\cO_{\al_1}$ and $\cO_{\al_2}$ corresponding to the colligations of form \eqref{E2.23B} and \eqref{E2.24}, respectively, if condition \eqref{E2.26} is satisfied.

\bdefnl{D2.5}
The colligation $\al$ of form \eqref{E2.30} constructed as described above is called \emph{the product of the colligations $\al_1$ and $\al_2$}, and this will be written by the equality $\al=\al_1\al_2$.
The corresponding system $\cO_{\al}$ is called \emph{the coupling of the open systems $\cO_{\al_1}$ and $\cO_{\al_2}$}, and this will be written by the equality $\cO_{\al}=\cO_{\al_1} \curlyvee\cO_{\al_2}$.
\edefn

Taking into account the described construction, the equality $\cO_{\al}=\cO_{\al_1} \curlyvee\cO_{\al_2}$ can be schematically illustrated in the following way:

\begin{figure}[H]
    \centering
    \begin{tikzpicture}[decoration={
            markings,
            mark=at position 0.5 with {\arrowreversed{latex}}}
        ]
        
        \draw[postaction={decorate}] (0,0) --node[anchor=south] {$\vphi_1^+$} (1,0);
        \draw (2,0) circle (1cm);
        \node[label=center:{$h_1+h_2$}] at (2,0) {};
        \node[label=below:{$\cO_{\al}$}] at (2,-1) {};
        \draw[postaction={decorate}] (3,0) --node[anchor=south] {$\vphi_2^-$} (4,0);
        
        \node[label=center:{$=$}] at (4.5,0) {};
        
        \draw[postaction={decorate}] (5,0) --node[anchor=south] {$\vphi_1^+$} (6,0);
        \draw (6.5,0) circle (0.5cm);
        \node[label=center:{$h_1$}] at (6.5,0) {};
        \node[label=below:{$\cO_{\al_1}$}] at (6.5,-0.5) {};
        
        \draw[postaction={decorate}] (7,0) --node[anchor=south] {$\vphi_1^-=\vphi_2^+$} (9,0);
        
        \draw (9.5,0) circle (0.5cm);
        \node[label=center:{$h_2$}] at (9.5,0) {};
        \node[label=below:{$\cO_{\al_2}$}] at (9.5,-0.5) {};
        \draw[postaction={decorate}] (10,0) --node[anchor=south] {$\vphi_2^-$} (11,0);
    \end{tikzpicture}
    \caption{}\label{F1}
\end{figure}
It is immediately checked that the formula
\[
\rk{\al_1\al_2}\al_3
=\al_1\rk{\al_2\al_3}
\]
holds.

\bthmnl{\zitaa{MR0322542}{\cch{1}}}{T2.6-0623}
Let $\al=\hgapj $ the product of colligations $\al_1=\rk{\cH_1,\cG;A_1,\Phi_1,J}$ and $\al_2=\rk{\cH_2,\cG;A_2,\Phi_2,J}$.
Then $\rho\rk{A_1}\cap\rho\rk{A_2}\subseteq\rho\rk{A}$ and
\begin{multline}\label{E2.37-0623}
    \rk{A-z\IH}^\inv
    =\rk{A_1-z\Iu{\cH_1}}^\inv P_{\cH_1}+\rk{A_2-z\Iu{\cH_2}}^\inv P_{\cH_2}\\
    -\iu\rk{A_1-z\Iu{\cH_1}}^\inv\Phi_1^\ad J\Phi_2\rk{A_2-z\Iu{\cH_2}}^\inv P_{\cH_2},
    \qquad z\in\rho\rk{A_1}\cap\rho\rk{A_2},
\end{multline}
where $P_{\cH_j}$ is the orthoprojector from $\cH$ onto $\cH_j$, $j=1,2$.
\ethm
\bproof
From the second equality in \eqref{E2.29} it follows that
\beql{E2.38-0623}
A-z\IH
=\rk{A_1-z\Iu{\cH_1}}P_{\cH_1}+\rk{A_2-z\Iu{\cH_2}}P_{\cH_2}+\iu\Phi_1^\ad J\Phi_2P_{\cH_2}.
\eeq
Now it is easily checked that for each $z\in\rho\rk{A_1}\cap\rho\rk{A_2}$ the right hand side of identity \eqref{E2.37-0623} is an operator which is both the left and right inverse of operator \eqref{E2.38-0623}.
\eproof

By a factorization of a colligation we mean any of its representations as a product of two colligations.

\bthmnl{\zitaa{MR0182516}{\csubsec{2.2}}}{T2.7}
If a colligation $\al$ of form \eqref{E2.10} admits a factorization $\al=\al_1\al_2$, where $\al_j$ has form \eqref{E2.22}, $j=1,2$, then the factorization
\beql{E2.31}
S_{\al_1 \al_2} \rk{z}=S_{\al_1}\rk{z}S_{\al_2}\rk{z},
\qquad z\in\rho\rk{A_1}\cap\rho\rk{A_2},
\eeq
holds.
\ethm
\bproof
Obviously, if $z\in\rho\rk{A_1}\cap\rho\rk{A_2}$, then $z\in\rho\rk{A}$.
Further, on the one hand, it follows from \eqref{E2.28} that $\vphi_{10}^+=S_\al\rk{z}\vphi_{20}^-$, $z\in\rho\rk{A}$.
On the other hand, from \eqref{E2.25} and \eqref{E2.26} we obtain
\[
\vphi_{10}^+=S_{\al_1}\rk{z}\vphi_{10}^-=S_{\al_1}\rk{z}\vphi_{20}^+=S_{\al_1}\rk{z}S_{\al_2}\rk{z}\vphi_{20}^-,
\qquad z\in\rho\rk{A_1}\cap\rho\rk{A_2}.\qedhere
\]
\eproof

\bthmnl{\zitaa{MR0322542}{\cch{I}}}{T2.8-0715}
Let $\cH_0$ be an invariant subspace of operator $A\in\ek{\cH}$, $A_0\defeq\rstr{A}{\cH_0}$ and $G$ be a domain in the complex plane all points of which are regular with respect to $A$.
If there exists a point $z_0\in G$ which is regular for the operator $A_0$, then all points  of the set $G$ have this property.
\ethm
\bproof
First we note that a point $z\in G$ is regular with respect to operator $A_0$ if and only if $\rk{A-z\IH}^\inv\cH_0\subseteq\cH_0$.

For an arbitrary point $z_1\in G$ we construct a rectifiable curve $L$ in $G$ which connects $z_0$ with $z_1$.
We choose an $\epsilon>0$ and points
\[
z_0=w_0,w_1,w_2,\dotsc,w_n=z_1
\]
on $L$ such that the disks
\[
\ck_j:=\{z: \abs{z-w_j}<\epsilon\},\;j=0,1,\dotsc,n-1,
\]
are contained in $G$ and satisfy the inequalities
\[
\abs{w_{j+1}-w_j}<\epsilon,
\qquad j=0,1,2,\dotsc,n-1.
\]
Since $\rk{A-w_0\IH}^\inv\cH_0\subseteq\cH_0$ and since we have the Taylor expansion
\[
R_z=R_{w_0}+\rk{z-w_0}R_{w_0}^2+\rk{z-w_0}^2R_{w_0}^3+\dotsb,
\qquad z\in\ck_0,
\]
and then $\rk{A-z\IH}^\inv\cH_0\subseteq\cH_0$ for $z\in\ck_0$ and, in particular, $\rk{A-w_1\IH}^\inv\cH_0\subseteq\cH_0$.
Continuing this procedure, we get $\rk{A-w_n\IH}^\inv\cH_0\subseteq\cH_0$.
\eproof

\bthml{T2.9-0715}
Let the colligation $\al=\hgapj $ be the product of the colligations $\al_j=\rk{\cH_j,\cG;A_j,\Phi_j,J}$, $j=1,2$.
If $\rho\rk{A}$ is connected, then
\beql{E2.43-0715}
\rho\rk{A}=\rho\rk{A_1}\cap\rho\rk{A_2}
\eeq
and, hence,
\[
\sigma\rk{A}=\sigma\rk{A_1}\cup\sigma\rk{A_2}.
\]
\ethm
\bproof
From the connectedness of the set $\rho\rk{A}$ and \rthm{T2.8-0715} we get
\beql{E2.44-0715}
\rho\rk{A}\subseteq\rho\rk{A_1}.
\eeq
Analogously, for the operator $A^\ad$ we get the inclusion $\rho\rk{A^\ad}\subseteq\rho\rk{A_2^\ad}$ and hence
\beql{E2.45-0715}
\rho\rk{A}\subseteq\rho\rk{A_2}.
\eeq
From \eqref{E2.44-0715}, \eqref{E2.45-0715} and \rthm{T2.6-0623} equality \eqref{E2.43-0715} follows.
\eproof

\subsection{Operator colligations and their properties}\label{subsec2.3-0604}
\subsubsection{Completely non-selfadjoint operator}
Let $\cH$ be a Hilbert space and $A\in\ek{\cH}$.
The operator $A$ is called \emph{completely non-selfadjoint} if there exist no subspace $\cH_0(\neq\set{0})$ in $\cH$ such that
\begin{enumerate}
    \item[1)] $\cH_0$ reduces $A$;
    \item[2)] $A$ induces a selfadjoint operator on $\cH_0$.
\end{enumerate}

\bthmnl{\zitaa{MR0322542}{\cch{I}, \S1}}{T3.1}
The closure $\cH_1$ of the linear hull of vectors of the form
\beql{E3.1}
A^n({\imA })h,
\qquad(n=0,1,2,\dotsc;h\in\cH)
\eeq
and its orthogonal complement $\cH_0:=\cH\ominus\cH_1$ are invariant with respect to $A$.
The operator $A$ induces a completely non-selfadjoint operator on $\cH_1$ and a selfadjoint operator on $\cH_0$.
\ethm
\bproof
Since $\cH_1$ is obviously invariant relative to $A$, therefore $\cH_0$ is invariant relative to $A^\ad$.
Moreover, since the range of the operator ${\imA }$ lies in $\cH_1$ we have $({\imA })\cH_0=\{0\}$.
Thus, $Ah=A^\ad h$ for $h\in\cH_0$.
Hence, the subspace $\cH_0$ reduces the operator $A$, and $A$ induces a selfadjoint operator on $\cH_0$.

On the other side, if the subspace $\cH_0'$ reduces $A$ and $A$ induces a selfadjoint operator on $\cH_0'$ then obviously $\cH_0'\subseteq\cH_0$.
Hence the operator $A$ induces a completely non-selfadjoint operator on $\cH_1$.
\eproof

\bcorl{C3.2}
The space $\cH$ can be represented in one and only one way in the form of an orthogonal sum of subspaces $\cH_1$ and $\cH_0$ which are invariant relative to $A$ and on which $A$ induces a completely non-selfadjoint operator and a selfadjoint operator, respectively.
\ecor

\subsubsection{Simple and redundant  colligations}
We consider an operator colligation
\beql{E3.2}
\al
\defeq\hgapj 
\eeq
and denote by $\cH_\al$ the closure of the linear hull of all vectors of the form
\beql{E3.3}
A^n\Phi^\ad g,
\qquad n=0,1,2,\dotsc;g\in\cG.
\eeq
The subspaces $\cH_\al$ and $\cH_\al^{(0)}\defeq\cH\ominus\cH_\al$ are called \emph{principal} and \emph{redundant for the colligation $\al$}, respectively.
From \rthm{T3.1} and inclusion $\ran{\imA }\subseteq\ran{\Phi^\ad}$ easily follows that each of the subspaces $\cH_\al$ and $\cH_\al^{(0)}$ is invariant relative to $A$ and $A^\ad$ and that $Ah=A^\ad h$ for $h\in\cH_\al^{(0)}$.

Colligation \eqref{E3.2} is called \emph{simple} if $\cH_\al=\cH$.
Otherwise, it is called \emph{redundant}.
For a colligation to be simple, it is sufficient that its fundamental operator be completely non-selfadjoint.
The converse statement is not correct in general.
Indeed, putting $\cE\defeq\cH$ in \rthm{T2.2}, we find that every bounded linear operator may be embedded in a simple colligation.  

We note that, besides the operator colligation $\al$ of form \eqref{E3.2}, the tuple
\beql{E2.201A}
\al^\ad
\defeq\rk{\cH,\cG;A^\ad,\Phi,-J}
\eeq
is also an operator colligation.
The operator colligation $\al^\ad$ is called the \emph{adjoint colligation with respect to $\al$}.
It is easily checked that the formula
\beql{E2.44-0622}
\rk{\al_1\al_2}^\ad
=\al_2^\ad\al_1^\ad
\eeq
holds.

The principal subspaces of the colligations $\al$ and $\al^\ad$ coincide.
Indeed, since the subspace $\cH_\al$ is invariant with respect to $A^\ad$ and $\ran{\Phi^\ad}\subseteq\cH_\al$, we have
\[
A^{\ad n}\Phi^\ad g\in\cH_\al,
\qquad n=0,1,2,\dotsc,
\qquad g\in\cG,
\]
\tie{}, $\cH_{\al^\ad}\subseteq\cH_\al$.
An analogous argument shows that $\cH_\al\subseteq\cH_{\al^\ad}$.

\blemnl{\zitaa{MR0322542}{\cch{I}}}{L2.9-0622}
Suppose that $\al\defeq\hgapj $ is an operator colligation.
If the subspace $\cH_0\subseteq\cH$ is invariant with respect to $A$ and orthogonal to $\ran{\Phi^\ad}$, then it is contained in the redundant subspace.
\elem
\bproof
For any vector $h\in\cH_0$ the equalities
\[
\ipa{h}{{A^\ad}^n\Phi^\ad g}
=\ipa{A^n h}{\Phi^\ad g}
=0,
\qquad n=0,1,2,\dotsc,
\qquad g\in\cG,
\]
holds, which means that $h\perp\cH_{\al^\ad}$.
Since $\cH_\al=\cH_{\al^\ad}$, we have $h\perp\cH_\al$, \tie, $h\in\cH_\al^{\rk{0}}$.
\eproof

\bcorl{C2.10-0622}
For the colligation $\al\defeq\hgapj $ to be redundant, it is necessary and sufficient that there exists a subspace $\cH_0\subseteq\cH$ $(\cH_0\neq\set{0})$ which is invariant with respect to $A$ and orthogonal to $\ran{\Phi^\ad}$.
\ecor

\bdefnl{D2.15-1031}
Suppose that $\al\defeq\hgapj $ is a redundant colligation.
Denoting by $A_\al$ and $A_\al^{\rk{0}}$ the operators induced by the operator $A$ in the subspaces $\cH_\al$ and $\cH_\al^{\rk{0}}$, respectively.
We obtain the colligations
\[
\al_\pri\defeq\rk{\cH_\al,\cG;A_\al,\Phi,J}
\qquad\text{and}
\qquad \al_\red\defeq\rk{\cH_\al^{\rk{0}},\cG;A_\al^{\rk{0}},0,J}
\]
which are called \emph{the principal} and \emph{redundant parts} of the colligation $\al$, respectively.
\edefn

It is easy to see that $\al_\pri$ is a simple colligation.

Since the channel operator of the colligation $\al_\red$ is the null operator, then the open system $\cO_{\al_\red}$ is closed and
\[
S_{\al_\red}\rk{z}
\equiv\IG, \qquad z\in\ \rho(A_\al^{\rk{0}}).
\]
Obviously,
\[
\cO_\al
=\cO_{\al_\pri}\curlyvee\cO_{\al_\red}
=\cO_{\al_\red}\curlyvee\cO_{\al_\pri}
\]
and
\[
\al
=\al_\pri\cdot\al_\red
=\al_\red\cdot\al_\pri.
\]
From this it follows
\beql{E2.45-0622}
S_\al\rk{z}
=S_{\al_\pri}\rk{z}\cdot S_{\al_\red}\rk{z}
=S_{\al_\red}\rk{z}\cdot S_{\al_\pri}\rk{z}
=S_{\al_\pri}\rk{z}, \ \ z\in\ \rho(A).
\eeq
In this way, the system $\cO_{\al_\red}$ is the maximal closed subsystem of the open system $\cO_\al$.

\subsubsection{Projection of a colligation}\label{subsec2.3.3-0519}
We choose in the inner space $\cH$ of a colligation of form \eqref{E3.2} a subspace $\cH_0$ and define on it the operator $A_0h\defeq P_0Ah$, $h\in\cH_0$, where $P_0$ denotes the orthoprojector from $\cH$ onto $\cH_0$.
Moreover, we define the operator $\Phi_0\defeq\rstr{\Phi}{\cH_0}$ action from $\cH_0$ into $\cG$.
Then for $h\in\cH_0$ we have
\[
\frac{1}{\iu}\rk{A_0-A_0^\ad}h
=\frac{1}{\iu}P_0\rk{A-A^\ad}P_0h
=P_0\Phi^\ad J\Phi P_0h
=\Phi_0^\ad J\Phi_0h,
\]
and thus,
\beql{E3.8}
\al_0
\defeq\rk{\cH_0,\cG;A_0,\Phi_0,J}
\eeq
is an operator colligation.

The colligation $\al_0$ is called the \emph{projection of the colligation $\al$ on the subspace $\cH_0$}, and we write 
\[
\al_0=\proj{\al}{\cH_0}, \ \  \proj{S}{\cH_0}(z):= S_{\proj{\al}{\cH_0}}(z).
\]
We note the relations
\begin{gather}
    \proj{\al^\ad}{\cH_0}=\rk{\proj{\al}{\cH_0}}^\ad\label{E2.47-0622}\\
    \proj{\al}{\cH_1}=\proj{\rk{\proj{\al}{\cH_2}}}{\cH_1},\qquad\cH_1\subseteq\cH_2,\label{E2.48-0622}
\end{gather}
following immediately from the definition of projection.
If $\al_j\defeq\rk{\cH_j,\cG_j;A_j,\Phi_j,J}$, $j=1,2$, are operator colligations and
\[
\al
\defeq\al_1\al_2
=\hgapj ,
\]
then, as formula \eqref{E2.29} shows, $\al_1$ and $\al_2$ are the corresponding projections of the colligation $\al$ on $\cH_1$ and $\cH_2$, respectively, where $\cH_1$ is invariant with respect to $A$.
The converse is also true.

\bthmnl{\zitaa{MR0347396}{\cch{II}}, \zitaa{MR0322542}{\cch{I}}}{T3.6}
Let
\beql{E3.9A}
\al
\defeq\hgapj 
\eeq
be an operator colligation, $\cH_1$ be an invariant subspace with respect to the operator $A$, and $\cH_2\defeq\cH\ominus\cH_1$.
Then $\al=\al_1\al_2$, where $\al_j\defeq\proj{\al}{\cH_j}$, $j=1,2$.
\ethm
\bproof
Let the colligation $\al_j$ of form \eqref{E2.22} be the projection of the colligation $\al$ on the subspace $\cH_j$, $j=1,2$.
If we denote by $P_j$ the orthoprojection onto $\cH_j$, $j=1,2$, we get
\begin{gather*}
    P_2AP_1=0,
    \qquad P_1A^\ad P_2=0,\\
    P_1AP_2
    =2\iu P_1\frac{A-A^\ad}{2\iu}P_2
    =2\iu P_1\Phi^\ad J\Phi P_2
    =2\iu \Phi_1^\ad J\Phi_2P_2,\\
    A
    =\rk{P_1+P_2}A\rk{P_1+P_2}
    =A_1P_1+A_2P_2+2\iu \Phi_1^\ad J\Phi_2P_2,
\end{gather*}
where $\Phi_j\defeq\rstr{\Phi}{\cH_j}$, $j=1,2$.
\eproof

The colligation $\al_0$ will be called \emph{a left} (\emph{right}) \emph{divisor} of the colligation $\al\defeq\hgapj $ if it is the projection of $\al$ on the subspace $\cH_0\subseteq\cH$ which is invariant with respect to $A$ ($A^\ad$).

Let $\al\defeq\hgapj $ be an operator colligation and let
\beql{E2.56-0815}
\set{0}
=\cH_0
\subset\cH_1
\subset\cH_2
\subset\dotsb
\subset\cH_n
=\cH
\eeq
be an increasing chain of subspaces which are invariant with respect to $A$.
This chain $\set{\cH_k}_{k=0}^n$ generates a factorization of the colligation $\al$.
Namely, since the subspace $\cH_{k-1}$ is invariant with respect to the operators $A_k=\rstr{A}{\cH_k}$, $k=2,3,\dotsc,n$, then, in view of formula \eqref{E2.48-0622}, we have
\beql{E2.49-0622}
\al
=\proj{\al}{\cH_1\ominus\cH_0}\cdot\proj{\al}{\cH_2\ominus\cH_1}\dotsm\proj{\al}{\cH_n\ominus\cH_{n-1}}.
\eeq

Let the fundamental operators of the colligations $\proj{\al}{\cH_j\ominus\cH_{j-1}}, \ 
j=1,2,3,\dotsc,n,$ be regular at the point $z$. Then from \eqref{E2.49-0622} it follows
\beql{E2.57-0801}
S_{\al}(z)
=\proj{S}{\cH_1\ominus\cH_0}(z)\cdot\proj{S}{\cH_2\ominus\cH_1}(z)\dots\proj{S}{\cH_n\ominus\cH_{n-1}}(z).
\eeq

\subsubsection{Redundant part of product of operator colligations}

\bthmnl{\zitaa{MR0322542}{\cch{1}}}{T2.15-0622}
If $\al\defeq\hgapj $ is the product of colligations $\al_1\defeq\rk{\cH_1,\cG;A_1,\Phi_1,J}$ and $\al_2\defeq\rk{\cH_2,\cG;A_2,\Phi_2,J}$, then
\beql{E2.50-0622}
\cH_{\al_j}^{\rk{0}}=\cH_\al^{\rk{0}}\cap\cH_j,
\qquad j=1,2.
\eeq
\ethm
\bproof
The subspace $\cH_{\al_1}^{\rk{0}}$ is invariant with respect to $A$ and orthogonal to $\ran{\Phi_1^\ad}$.
In view of the third identity in \eqref{E2.29}, it is orthogonal to $\ran{\Phi^\ad}$.
In view of \rlem{L2.9-0622}, we have $\cH_{\al_1}^{\rk{0}}\subseteq\cH_\al^{\rk{0}}$ and thus, $\cH_{\al_1}^{\rk{0}}\subseteq\cH_\al^{\rk{0}}\cap\cH_1$.

On the other hand, the subspace $\cH_\al^{\rk{0}}\cap\cH_1$ is invariant with respect to $A_1$ and orthogonal to $\ran{\Phi_1^\ad}$.
Applying \rlem{L2.9-0622} to the colligation $\al_1$, we infer that $\cH_\al^{\rk{0}}\cap\cH_1\subseteq\cH_{\al_1}^{\rk{0}}$.

Thus, equality \eqref{E2.50-0622} is proved for $j=1$.
It remains to note that $\al^\ad=\al_2^\ad\al_1^\ad$ and thus, in view of what we have already, $\cH_{\al_2^\ad}^{\rk{0}}=\cH_{\al^\ad}^{\rk{0}}\cap\cH_2$.
Since $\cH_{\al_2^\ad}^{\rk{0}}=\cH_{\al_2}^{\rk{0}}$, $\cH_{\al^\ad}^{\rk{0}}=\cH_{\al}^{\rk{0}}$, then $\cH_{\al_2}^{\rk{0}}=\cH_{\al}^{\rk{0}}\cap\cH_2$.
\eproof

\bthmnl{\zitaa{MR0080269}{\cch{1}}}{T2.13-0622}
If $\al=\al_1\al_2\dotsm\al_n$, then
\beql{E2.51-0622}
\cH_{\al_j}^{\rk{0}}=\cH_\al^{\rk{0}}\cap\cH_j,
\qquad j=1,2,\dotsc,n,
\eeq
where $\cH_j$ is the internal space of the colligation $\al_j$.
\ethm
\bproof
In view of \rthm{T2.15-0622}, we have
\begin{gather*}
    \cH_{\al_j}^{\rk{0}}=\cH_{\al_1\al_2\dotsm\al_j}^{\rk{0}}\cap\cH_j,\\
    \cH_{\al_1\al_2\dotsm\al_j}^{\rk{0}}=\cH_\al^{\rk{0}}\cap\rk{\cH_1\oplus\cH_2\oplus\dotsb\oplus\cH_j}.
\end{gather*}
From this \eqref{E2.51-0622} follows.
\eproof

\bcorl{C2.14-0622}
If $\al=\al_1\al_2\dotsm\al_n$ is a simple colligation, then all colligations $\al_j$ ($j=1,2,\dotsc,n$) are simple.
\ecor

From this and \rthm{T3.6} it follows:

\bcorl{C2.20-0519}
Let $\al\defeq\hgapj$ be a simple operator colligation and let $\cH_1$ be an invariant subspace with respect to the operator $A$.
Then the operator colligation $\al_1\defeq\proj{\al}{\cH_1}$ is simple, too.
\ecor

Note that the converse assumption to \rcor{C2.14-0622} is not true:
the product of simple colligations can be redundant.
More precisely, the following is true.

\bthmnl{\zitaa{MR0322542}{\cch{I}}}{T2.15A-0622}
Let $\cH_0$ be a Hilbert space and $A_0$ be a bounded selfadjoint operator acting in $\cH_0$.
Then there exist simple colligations $\al_1$ and $\al_2$ such that $A_0$ is the fundamental operator of the redundant part of the colligation $\al_1\cdot\al_2$.
\ethm

\subsubsection{Unitarily equivalent colligations}

\bdefnl{D2.3}
We say that a colligation $\al_1\defeq\rk{\cH_1,\cG;A_1,\Phi_1,J}$ is \emph{unitarily equivalent} to a colligation $\al_2\defeq\rk{\cH_2,\cG;A_2,\Phi_2,J}$, if there exists a unitary mapping $U$ of the space $\cH_1$ onto $\cH_2$ such that
\beql{E3.4}
UA_1=A_2U,
\qquad\Phi_1=\Phi_2U.
\eeq
\edefn

Obviously, the relation of unitarily equivalence is reflexive, symmetric and transitive.
It is also obvious that a colligation which is unitarily equivalent to a simple colligation is simple itself.

\breml{R3.2-0630}
It is immediately checked that a unitary operator $U$ which realizes the unitarily equivalence of two simple colligations $\al_1$ and $\al_2$ is determined in unique way. If at the same time $\al_1 = \al_2 = \al$, then obviously $U = I_{\cH}$, where $\cH$ is the internal space $\cH$ of the colligation $\al$.  
\erem

\bthmnl{\zitaa{MR0322542}{\cch{I}}}{T3.3-0630}
Let the colligation $\al_j\defeq\rk{\cH_j,\cG;A_j,\Phi_j,J}$, $j=1,2$, be unitarily equivalent to the colligation $\al_j'\defeq\rk{\cH_j',\cG;A_j',\Phi_j',J}$, $j=1,2$, respectively.
Then the product $\al_1\al_2$ is unitarily equivalent to the product $\al_1'\al_2'$.
\ethm
\bproof
Let the unitary operators 
\[
U_1\in\ek{\cH_1,\cH_1'},
\qquad U_2\in\ek{\cH_2,\cH_2'},
\]
realize the unitary equivalence of the colligations $\al_1,\al_1'$ and $\al_2,\al_2'$, respectively.
It's immediately checked that the unitary operator $U\defeq U_1\oplus U_2$ 
realizes the unitary equivalence of the colligations $\al_1\al_2$ and $\al_1'\al_2'$.
\eproof

Analogously, the following result is verified. The difference is that now the operators $U_1$ and $U_2$ are determined by the operator $U$.

\bthmnl{\zitaa{MR0322542}{\cch{I}}}{T3.4-0630}
Let the colligation $\al$ be unitarily equivalent to the colligation $\al'$ and let $\al=\al_1\al_2$.
Then $\al'=\al_1'\al_2'$, where the colligations $\al_1'$ and $\al_2'$ are unitarily equivalent to $\al_1$ and $\al_2$, respectively.
\ethm


\section{Three basic properties of characteristic operator function}\label{sec3-1223}
In this section we state three basic properties of the \tcof{}
First (\rsubsec{subsec3.1-0606}), we note that the \tcof{}\ is the full unitary invariant of the operator colligation.
Second (see \rthm{T3.7}), it turns out that the invariant subspaces of the corresponding operator are associated with left divisors of the \tcof{}
Namely, this property enables us to construct a triangular model of the corresponding operator using the \tcof{}\ (see \rsec{S1551}).
Third, the $J$\nobreakdash-property of the \tcof{}\ (see \eqref{E3.11}--\eqref{E3.13}) is a basic property which determines the class of \tcof{}\ (see \rsec{S4}).

At the end of this section (\rsubsec{subsec3.4-0606}) we show that the signature of the operator $J$ determines the directions of the outer channels.

\subsection{The characteristic operator function is the full unitary invariant of the operator colligation}\label{subsec3.1-0606}
If the colligations $\al_j\defeq\rk{\cH_j,\cG;A_j,\Phi_j,J}$, $j=1,2$, are unitarily equivalent, then the set $\rho\rk{A_1}$ of all regular points of the operator $A_1$ coincides with the set $\rho\rk{A_2}$ of all regular points of the operator $A_2$ and
\beql{E3.5}
S_{\al_1}\rk{z}=S_{\al_2}\rk{z},
\qquad z\in\rho\rk{A_1}.
\eeq
Indeed, in view of \eqref{E3.4}, we have
\[\begin{split}
    S_{\al_2}\rk{z}
    &=\IG-\iu\Phi_2\rk{A_2-z\Iu{\cH_2}}^\inv\Phi_2^\ad J\\
    &=\IG-\iu\Phi_1U^\inv\ekb{U\rk{A_1-z\Iu{\cH_1}}^\inv U^\inv}U\Phi_1^\ad J\\
    &=\IG-\iu\Phi_1\rk{A_1-z\Iu{\cH_1}}^\inv\Phi_1^\ad J
    =S_{\al_1}\rk{z}.
\end{split}\]

\bthmnl{\zitaa{MR0062955}{\cthm{1}}}{T3.4}
Let
\beql{E3.6}
\al_j
\defeq\rk{\cH_j,\cG;A_j,\Phi_j,J},
\qquad j=1,2,
\eeq
be simple colligations.
If in some neighborhood $G$ of the infinitely distant point the equality $S_{\al_1}\rk{z}=S_{\al_2}\rk{z}$ is satisfied, then $\al_1$ and $\al_2$ are unitarily equivalent.
\ethm
\bproof
In the theorem it is assumed that
\[
\Phi_1\rk{A_1-z\Iu{\cH_1}}^\inv\Phi_1^\ad
=\Phi_2\rk{A_2-z\Iu{\cH_2}}^\inv\Phi_2^\ad,
\qquad z\in G.
\]
Since
\[\begin{split}
    &\rk{A_j-z\Iu{\cH_j}}^\inv-\rk{A_j^\ad-\ko{w}\Iu{\cH_j}}^\inv\\
    &=\rk{A_j^\ad-\ko{w}\Iu{\cH_j}}^\inv\ekb{\rk{A_j^\ad-\ko{w}\Iu{\cH_j}}-\rk{A_j-z\Iu{\cH_j}}}\rk{A_j-z\Iu{\cH_j}}^\inv\\
    &=\rk{z-\ko{w}}\rk{A_j^\ad-\ko{w}\Iu{\cH_j}}^\inv\rk{A_j-z\Iu{\cH_j}}^\inv\\
    &\qquad-2\iu\rk{A_j^\ad-\ko{w}\Iu{\cH_j}}^\inv\Phi_j^\ad J\Phi_j\rk{A_j-z\Iu{\cH_j}}^\inv,
    \qquad j=1,2;\;z,w\in G,
\end{split}\]
we infer
\[\begin{split}
    &\rk{z-\ko{w}}\Phi_1\rk{A_1^\ad-\ko{w}\Iu{\cH_1}}^\inv\rk{A_1-z\Iu{\cH_1}}^\inv\Phi_1^\ad\\
    &=\Phi_1\rk{A_1-z\Iu{\cH_1}}^\inv\Phi_1^\ad-\Phi_1\rk{A_1^\ad-\ko{w}\Iu{\cH_1}}^\inv\Phi_1^\ad\\
    &\qquad+2\iu\Phi_1\rk{A_1^\ad-\ko{w}\Iu{\cH_1}}^\inv\Phi_1^\ad J\Phi_1\rk{A_1-z\Iu{\cH_1}}^\inv\Phi_1^\ad\\
    &=\Phi_2\rk{A_2-z\Iu{\cH_2}}^\inv\Phi_2^\ad-\Phi_2\rk{A_2^\ad-\ko{w}\Iu{\cH_2}}^\inv\Phi_2^\ad\\
    &\qquad+2\iu\Phi_2\rk{A_2^\ad-\ko{w}\Iu{\cH_2}}^\inv\Phi_2^\ad J\Phi_2\rk{A_2-z\Iu{\cH_2}}^\inv\Phi_2^\ad\\
    &=\rk{z-\ko{w}}\Phi_2\rk{A_2^\ad-\ko{w}\Iu{\cH_2}}^\inv\rk{A_2-z\Iu{\cH_2}}^\inv\Phi_2^\ad.
\end{split}\]
Hence,
\begin{multline*}
    \Phi_1\rk{A_1^\ad-\ko{w}\Iu{\cH_1}}^\inv\rk{A_1-z\Iu{\cH_1}}^\inv\Phi_1^\ad\\
    =\Phi_2\rk{A_2^\ad-\ko{w}\Iu{\cH_2}}^\inv\rk{A_2-z\Iu{\cH_2}}^\inv\Phi_2^\ad,
    \qquad z,\ko{w}\in G.
\end{multline*}
From this, using the expansion
\[
\rk{A_j-z\Iu{\cH_j}}^\inv
=-\frac{1}{z}\Iu{\cH_j}-\frac{A_j}{z^2}-\frac{A_j^2}{z^3}-\dotsc,
\qquad\abs{z}>\norm{A_j},j=1,2,
\]
we obtain the equalities
\beql{E3.7} 
\ip{A_1^m\Phi_1^\ad g}{A_1^n\Phi_1^\ad g'}
=\ip{A_2^m\Phi_2^\ad g}{A_2^n\Phi_2^\ad g'},
\qquad m,n=0,1,2,\dotsc;\;g,g'\in\cG.
\eeq
We denote by $\hcH_j$, $j=1,2$, the linear hull of all vectors of the form
\[
A_j^m\Phi_j^\ad g,
\qquad m=0,1,2,\dotsc;\;g\in\cG,
\]
and consider the mapping $U$ from $\hcH_1$ onto $\hcH_2$ which maps each vector of the form $\sum_{n=0}^\ell A_1^n\Phi_1^\ad g_n$ to the vector $\sum_{n=0}^\ell A_2^n\Phi_2^\ad g_n$.
In view of \eqref{E3.7}, the mapping $U$ is isometric.
Since $\hcH_1$ and $\hcH_2$ are dense in $\cH_1$ and $\cH_2$, respectively, $U$ can be continuously extended to an isometry of the whole space $\cH_1$ onto the whole space $\cH_2$.
We again denote by $U$ the obtained unitary operator.
Then the equalities
\[
UA_1^n\Phi_1^\ad
=A_2^n\Phi_2^\ad,
\qquad n=0,1,2,\dotsc
\]
hold.
Thus, $U\Phi_1^\ad=\Phi_2^\ad$ and, moreover,
\[
UA_1A_1^n\Phi_1^\ad
=A_2A_2^n\Phi_2^\ad
=A_2UA_1^n\Phi_1^\ad,
\]
\tie{}, $UA_1=A_2U$.
\eproof

\bcorl{C3.5}
Let colligations $\al_1$ and $\al_2$ of form \eqref{E3.6} be simple.
If in some neighborhood of the infinitely distant point the equality $S_{\al_1}\rk{z}=S_{\al_2}\rk{z}$ holds, then $\rho\rk{A_1}=\rho\rk{A_2}$ and $S_{\al_1}\rk{z}\equiv S_{\al_2}\rk{z}$, $z\in\rho\rk{A_1}$.
\ecor

We note that from the identity $S_\al\rk{z}=S_{\al_\pri}\rk{z}$ (see \eqref{E2.45-0622}) follows that the assumption of simplicity of colligations in \rthm{T3.4} is necessary.

\bthmnl{\zitaa{MR0322542}{\cch{I}}}{T3.6-0630}
Let $\al_1\defeq\rk{\cH_1,\cG;A_1,\Phi_1,J}$ and $\al_1'\defeq\rk{\cH_1',\cG;A_1',\Phi_1',J}$ be divisors of a simple colligation $\al\defeq\hgapj $.
If $S_{\al_1}\rk{z}=S_{\al_1'}\rk{z}$, then $\al_1=\al_1'$.
\ethm
\bproof
Since
\[
\al_1=\proj{\al}{\cH_1},
\qquad \al_1'=\proj{\al}{\cH_1'},
\]
it is sufficient to prove that $\cH_1=\cH_1'$.

We consider the relations
\beql{E3.5-0630}
\al
=\al_1\al_2
=\al_1'\al_2'.
\eeq
In view of \rcor{C2.14-0622}, the colligations $\al_1,\al_2,\al_1',\al_2'$ are simple.
From \rthm{T2.7} it follows that for all $z$ from some neighborhood $G$ of the infinitely distant point the identity
\[
S_\al\rk{z}
=S_{\al_1}\rk{z}S_{\al_2}\rk{z}
=S_{\al_1'}\rk{z}S_{\al_2'}\rk{z}
\]
is satisfied.
In view of
\[
\lim_{z\to\infty}\normb{S_{\al_1}\rk{z}-\IG}=0,
\qquad\lim_{z\to\infty}\normb{S_{\al_1'}\rk{z}-\IG}=0,
\]
one can assume that
\[
S_{\al_1}^\inv\rk{z}=S_{\al_1'}^\inv\rk{z},
\qquad z\in G.
\]
Thus, $S_{\al_2}\rk{z}=S_{\al_2'}\rk{z}$, $z\in G$.
Since the colligations $\al_2$ and $\al_2'$ are simple, then, applying \rthm{T3.4-0630}, from this we get that these colligations are unitarily equivalent.
But then, in view of \rthm{T3.3-0630}, the products $\al_1\al_2$ and $\al_1'\al_2'$ are unitarily equivalent too.
In view of \eqref{E3.5-0630}, the simplicity of the colligation $\al$ and \rrem{R3.2-0630}, we obtain that the unitary operator $U$ which realizes this equivalence is equal to $\IH$.
Thus, $\cH_1'=U\cH_1=\cH_1$.
\eproof

\subsection{To invariant subspaces of the operator there correspond left divisors of its characteristic function}

\bthmnl{\zitaa{MR0347396}{\cch{II}}, \zitaa{MR0322542}{\cch{I}}}{T3.7}
Let $ \al \defeq\hgapj $
be an operator colligation, $\cH_1$ be an invariant subspace with respect to the ope\-rator $A$, $\cH_2\defeq\cH\ominus\cH_1$ and $\al_j\defeq\proj{\al}{\cH_j}$, $j=1,2$. 
Then
\[
\al=\al_1\al_2
\]
and
\beql{E3.10}
S_\al\rk{z}
=S_{\al_1}\rk{z}S_{\al_2}\rk{z},
\qquad z\in\rho\rk{A_1}\cap\rho\rk{A_2},
\eeq
where $A_j$  is the fundamental operator of the colligation $\al_j, \ j=1,2$, \tie{}, to invariant subspaces of the operator $A$ there correspond left divisors of its characteristic operator function.
\ethm

The assertion follows from \rthmss{T3.6}{T2.7}.

A complete answer to the question of connections between invariant subspaces of operator $A$ and left divisors of its \tcof{}\ will be given below in \rthm{T4.14-0711} (see \rsubsec{subsec4.4-0715}).

\breml{R3.8}
An important thing about factorization \eqref{E3.10} is that each of the factors is the characteristic operator function of the corresponding colligation, \tie{}, factorization \eqref{E3.10} is realized in the class of characteristic operator functions.
The connection of invariant subspaces of the operator with factorizations of its characteristic operator function was first observed in \cite{MR0034961} for quasi-unitary operators.
\erem


\subsection{$J$-property of characteristic operator function}
Let $ \al \defeq\hgapj $ be an operator colligation.
Then (\zitaa{MR0062955}{\csec{3}}, \zitaa{MR0347396}{\csubsec{1.3}})
\begin{align}
    S_\al^\ad\rk{z}JS_\al\rk{z}-J&\geq0,&\im z&>0,\quad z\in\rho\rk{A},\label{E3.11}\\
    S_\al^\ad\rk{z}JS_\al\rk{z}-J&=0,&\im z&=0,\quad z\in\rho\rk{A},\label{E3.12}\\
    S_\al^\ad\rk{z}JS_\al\rk{z}-J&\leq0,&\im z&<0,\quad z\in\rho\rk{A}.\label{E3.13}
\end{align}
Indeed, substituting $\vphi^-\rk{t}$, $h\rk{t}$, and $\vphi^+\rk{t}$ of form \eqref{E2.15}--\eqref{E2.16} in equation \eqref{E2.14} and taking into account \eqref{E2.18}, we obtain
\[
\iu\rk{z-\ko z}\ipa{h_0}{h_0}
=\ipa{J\vphi^-_0}{\vphi^-_0}-\ipa{JS_{\alpha}\rk{z}\vphi^-_0}{S_{\alpha}\rk{z}\vphi^-_0},\qquad z\in\rho\rk{A},
\]
\tie,
\[
\ipa{\rk{S_{\alpha}^\ad\rk{z}JS_{\alpha}\rk{z}-J}\vphi^-_0}{\vphi^-_0}
=\frac{z-\ko z}{\iu}\ipa{R_{\alpha}\rk{z}\vphi^-_0}{R_{\alpha}\rk{z}\vphi^-_0},\qquad z\in\rho\rk{A}.
\]
Hence (see \zitaa{MR0347396}{\csubsec{1.3}}),
\beql{E3.14}
S_{\alpha}^\ad\rk{z}JS_{\alpha}\rk{z}-J
=\frac{z-\ko z}{\iu}R_{\alpha}^\ad\rk{z}R_{\alpha}\rk{z},
\qquad z\in\rho\rk{A}.
\eeq
This implies \eqref{E3.11}--\eqref{E3.13}.

Together with colligation $\al$ we consider colligation $\al^\ad$ (see \eqref{E2.201A}).
From \eqref{E2.20} 
we infer
\beql{E3.15}
S_{\al^\ad}\rk{z}=JS_\al^\ad\rk{\ko z}J,
\qquad z\in\rho\rk{A^\ad}.
\eeq
From \eqref{E2.19} and \eqref{E2.201A} we obtain
\beql{E3.16}
R_{\al^\ad}\rk{z}=\rk{A^\ad-z\IG}^\inv\Phi^\ad\rk{-J},
\qquad z\in\rho\rk{A^\ad}.
\eeq
From this and \eqref{E3.14} we find that for $z\in\rho\rk{A}$
\[\begin{split}
    S_\al\rk{z}JS_\al^\ad\rk{z}-J
    &=JS_{\al^\ad}^\ad\rk{\ko z}J\cdot J\cdot JS_{\al^\ad}\rk{\ko z}J-J\\
    &=-J\rkb{S_{\al^\ad}^\ad\rk{\ko z}\rk{-J}S_{\al^\ad}\rk{\ko z}-\rk{-J}}J\\
    &=\frac{z-\ko z}{\iu}JR_{\al^\ad}^\ad\rk{\ko z}R_{\al^\ad}\rk{\ko z}J\\
    &=\frac{z-\ko z}{\iu}\Phi\rk{A-z\IG}^\inv\rk{A^\ad-\ko z\IG}^\inv\Phi^\ad.
\end{split}\]
Hence, analogously to \eqref{E3.11}--\eqref{E3.13} we obtain
\begin{align}
    S_\al\rk{z}JS_\al^\ad\rk{z}-J&\geq0,&\im z&>0,\quad z\in\rho\rk{A},\label{E3.17}\\
    S_\al\rk{z}JS_\al^\ad\rk{z}-J&=0,&\im z&=0,\quad z\in\rho\rk{A},\label{E3.18}\\
    S_\al\rk{z}JS_\al^\ad\rk{z}-J&\leq0,&\im z&<0,\quad z\in\rho\rk{A}.\label{E3.19}
\end{align}
We say that the operator $Y\in\ek{\cG}$ is \emph{$J$\nobreakdash-expansive} (\emph{$J$-contractive}), if
\begin{align}
    Y^\ad JY-J&\geq0\label{E3.20}\\
    (Y^\ad JY-J&\leq0).\label{E3.21}
\end{align}
We say that the operator $Y$ is \emph{two-sided} $J$\nobreakdash-expansive (\emph{two-sided} $J$\nobreakdash-contractive) if the operators $Y$ and $Y^\ad$ are both $J$\nobreakdash-expansive (both $J$\nobreakdash-contractive).
This is equivalent to the fact that the conditions
\begin{align}
    Y^\ad JY-J&\geq0,&YJY^\ad -J&\geq0\label{E3.22}\\
    (Y^\ad JY-J&\leq0,&Y JY^\ad-J&\leq0)\label{E3.23}
\end{align}
are fulfilled.
The operator $Y$ is called \emph{$J$-unitary} in $\cG$, 
if there is equality in both inequalities \eqref{E3.22} (or equivalently in both inequalities \eqref{E3.23}).
Hence, from \eqref{E3.11}--\eqref{E3.13} and \eqref{E3.17}--\eqref{E3.19} it follows that in regular points $z$ of the operator $A$ the characteristic function $S_\al\rk{z}$ of the colligation $\al$
\begin{enumerate}
    \item[(1)] is two-sided $J$-expansive in the upper half plane;
    \item[(2)] is two-sided $J$-contractive in the lower half plane;
    \item[(3)] is $J$-unitary on the real axis.
\end{enumerate}

\breml{R3.9}
There exist $J$-expansive operators which are not two-sided $J$-expansive (see \teg{}\ \cite{MR0094691}).
In \cite{MR0094691} there are given sufficient conditions that a $J$-expansive operator is two-sided $J$-expansive.
In particular this holds if $\dim\cG<+\infty$ (\cite{MR0094691,MR0076882}).
\erem

Since by transition from $J$ to $-J$ the property of $J$-expansivity passes over to $\rk{-J}$-contractivity then all what is said holds also for $J$-contractive operators.
Consequently, in the case $\dim\cG<+\infty$ it is sufficient for the $J$-unitarity of the operator $Y$ that one of the inequalities \eqref{E3.22} and \eqref{E3.23} is satisfied (and this implies that there is equality in all remaining inequalities).


\subsection{Role of the operator $J$}\label{subsec3.4-0606}
Let $ \al \defeq\hgapj $ be an operator colligation of form \eqref{E3.2}.
In this section we assume that $\dim\cG<+\infty$.
In view of \eqref{E2.8} the space $\cG$ can be written as
\beql{E3.24}
\cG
=\cG_1\oplus\cG_{-1},
\eeq
where
\[
Jg
=
\begin{cases}
    g,&g\in\cG_1,\\
    -g,&g\in\cG_{-1}.
\end{cases}
\]
Let $\set{e_j'}_{j=1}^p$, $\set{e_j''}_{j=1}^q$ be orthonormal bases of the spaces $\cG_1$ and $\cG_{-1}$, respectively.
How it was mentioned above (see \rsubsec{subsubsec2.1.2-0715}), the space $\cG$ plays the role of a window in the open system $\cO_\al$.
In physical examples (see \zitaa{MR0182516}{\cch{1}}) the basis vectors $\set{e_j'}_{j=1}^p$ and $\set{e_j''}_{j=1}^q$ play the role of \emph{elementary channels}.
M.~S.~Liv\v{s}ic called the channels $\set{e_j'}_{j=1}^p$ \emph{direct} and the channels $\set{e_j''}_{j=1}^q$ \emph{inverse}, dependent on the sign of the eigenvalue of the operator $J$:
\beql{E3.25}
Je_j'=e_j',\qquad
Je_k''=-e_k'',\qquad
j=1,2,\dotsc,p,\quad
k=1,2,\dotsc,q.
\eeq
In schematic form this can be expressed in the following way:
\begin{figure}[H]
    \centering
    \begin{tikzpicture}
        \draw (2,0) circle (1cm);
        \draw[postaction={decorate, decoration={markings,mark=at position 0.5 with {\arrowreversed{latex}}}}] (0,.5) --node[anchor=south] {$e_j'$} (1.1,.5);
        \draw[postaction={decorate, decoration={markings,mark=at position 0.5 with {\arrowreversed{latex}}}}] (2.9,.5) --node[anchor=south] {$e_j'$} (4,.5);
        \draw[postaction={decorate, decoration={markings,mark=at position 0.5 with {\arrow{latex}}}}] (0,-.5) --node[anchor=north] {$e_k''$} (1.1,-.5);
        \draw[postaction={decorate, decoration={markings,mark=at position 0.5 with {\arrow{latex}}}}] (2.9,-.5) --node[anchor=north] {$e_k''$} (4,-.5);
        \node[label=center:{$\cH$}] at (2,0) {};
        \node[label=below:{$\cO_{\al}$}] at (2,-1) {};
    \end{tikzpicture}
    \caption{}\label{F2}
\end{figure}
In physical examples the sign of the eigenvalue of operator $J$ given in \eqref{E3.25} determines the character of relating wave to the  corresponding channel.
If all channels are direct then $J=\IG$ and, in view of \eqref{E2.9}, the imaginary part of the operator $A$ is non-negative and, thus, the operator $A$ is \emph{dissipative}.

According to the decomposition \eqref{E3.24}, let $P$ and $Q$ be the orthogonal projections from $\cG$ onto $\cG_1$ and $\cG_{-1}$, respectively.
Then
\begin{gather*}
    \IG=P+Q,\qquad
    J=P-Q,\\
    P=\frac{1}{2}\rk{\IG+J},\qquad
    P=\frac{1}{2}\rk{\IG-J}.
\end{gather*}
Let $S_\al\rk{z}$ be the \tcof{}\ of the colligation $\al$, $z_0\in\rho\rk{A}$, $\im z_0<0$ and $S_0\defeq S_\al\rk{z_0}$.
In view of \eqref{E3.13} and \eqref{E3.19}, then
\beql{E3.26}
S_0^\ad JS_0\leq J,\qquad
S_0JS_0^\ad\leq J.
\eeq
As it follows from \rrem{R3.9}, in the considered case the validity of one of the inequalities in \eqref{E3.26} implies the validity of the other one.
Hence, the operator $S_0$ is two-sided $J$-contractive.

In relations (see \eqref{E2.17})
\[
\left\{\begin{array}{l}
    \rk{A-z_0\Iu{\cH}}h_0=\Phi^\ad J\vphi_0^-,\\
    \vphi_0^+=\vphi_0^--\iu \Phi h_0
\end{array}\right.
\]
from the known input $\vphi_0^-=P\vphi_0^-+Q\vphi_0^-$ one can determine the internal state $h_0$ and the output $\vphi_0^+=P\vphi_0^++Q\vphi_0^+$, where (see \eqref{E2.18})
\beql{E3.27}
\vphi_0^+
=S_0\vphi_0^-.
\eeq
In accordance with \rfig{F2} we have
\begin{figure}[H]
    \centering
    \begin{tikzpicture}
        \draw (2,0) circle (1cm);
        \draw[postaction={decorate, decoration={markings,mark=at position 0.5 with {\arrowreversed{latex}}}}] (0,.5) --node[anchor=south] {$P\vphi_0^+$} (1.1,.5);
        \draw[postaction={decorate, decoration={markings,mark=at position 0.5 with {\arrowreversed{latex}}}}] (2.9,.5) --node[anchor=south] {$P\vphi_0^-$} (4,.5);
        \draw[postaction={decorate, decoration={markings,mark=at position 0.5 with {\arrow{latex}}}}] (0,-.5) --node[anchor=south] {$Q\vphi_0^+$} (1.1,-.5);
        \draw[postaction={decorate, decoration={markings,mark=at position 0.5 with {\arrow{latex}}}}] (2.9,-.5) --node[anchor=south] {$Q\vphi_0^-$} (4,-.5);
        \node[label=center:{$\cH$}] at (2,0) {};
        \node[label=below:{$\cO_{\al}$}] at (2,-1) {};
    \end{tikzpicture}
    \caption{}\label{F3}
\end{figure}

Let us consider a new open system with the input vector $\psi_0^-=P\vphi_0^-+Q\vphi_0^+$ (\tie{}, the vectors which ensure the flow of information inside the system in \rfig{F3}) whereas the output vector is given by  $\psi_0^+\defeq P\vphi_0^++Q\vphi_0^-$.
Actually, the transition to the new open system mean that we change the direction of the inverse channels.

The transfer operator $W_0$ (the value of the transfer function at the point $z_0$) of the new open system satisfies the condition
\beql{E3.28}
\psi_0^+
=W_0\psi_0^-.
\eeq
We express $W_0$ by $S_0$.
From \eqref{E3.27} we get
\begin{align*}
    P\vphi_0^++Q\vphi_0^+&=S_0P\vphi_0^-+S_0Q\vphi_0^-\\
    P\vphi_0^+-S_0Q\vphi_0^-&=S_0P\vphi_0^--Q\vphi_0^+.
\end{align*}
Thus,
\beql{E3.29}
\rk{P-S_0Q}\psi_0^+
=\rk{S_0P-Q}\psi_0^-.
\eeq

\begin{lem}\label{L3.10}
    If the operator $S_0$ is $J$-contractive, then operator $P-S_0Q$ is invertible.
\end{lem}
\begin{proof}
    We rewrite the second inequality in \eqref{E3.26} in the form
    \[
    S_0\rk{P-Q}S_0^\ad
    \leq P-Q,
    \]
    \tie{},
    \beql{E3.30}
    S_0PS_0^\ad+Q
    \leq P+S_0QS_0^\ad.
    \eeq
    If the operator $P-S_0Q$ is not invertible, then also the operator $P-QS_0^\ad=\rk{P-S_0Q}^\ad$ is not invertible.
    Thus, there exists a vector $g\in\cG\setminus\set{0}$ such that
    \[
    \rk{P-QS_0^\ad}g
    =0
    \]
    This is equivalent to
    \beql{E3.31}
    Pg=0,\qquad
    QS_0^\ad g=0.
    \eeq
    But then
    \[
    \rk{P+S_0QS_0^\ad}g
    =0
    \]
    and, in view of \eqref{E3.30}, we get
    \[
    S_0PS_0^\ad g=0,\qquad
    Qg=0.
    \]
    From this and \eqref{E3.31} it follows that
    \[
    g
    =Pg+Qg
    =0.
    \]
    This contradiction completes the proof.
\end{proof}

Taking into account the invertibility of the operator $P-S_0Q$ we rewrite \eqref{E3.29} in the form
\[
\psi_0^+
=\rk{P-S_0Q}^\inv\rk{S_0P-Q}\psi_0^-.
\]
Comparing with \eqref{E3.28}, we obtain
\beql{E3.32}
W_0
=\rk{P-S_0Q}^\inv\rk{S_0P-Q}.
\eeq
Important is the fact that it follows from the $J$-contractivity of the operator $S_0$ that $W_0$ is contractive.
Indeed,
\[\begin{split}
    &I-W_0W_0^\ad\\
    &=\rk{P-S_0Q}^\inv\ekb{\rk{P-S_0Q}\rk{P-QS_0^\ad}-\rk{S_0P-Q}\rk{PS_0^\ad-Q}}\rk{P-S_0Q}^\invad\\
    &=\rk{P-S_0Q}^\inv\rk{P+S_0QS_0^\ad-S_0PS_0^\ad-Q}\rk{P-S_0Q}^\invad\\
    &=\rk{P-S_0Q}^\inv\rk{J-S_0JS_0^\ad}\rk{P-S_0Q}^\invad
    \geq0.
\end{split}\]
The transformation \eqref{E3.32} which maps $J$-contractive operators to contractions is called \emph{Potapov--Ginzburg transform}.
It plays an important role in the theory of $J$-contractive operator functions.
In terms of open systems this transform corresponds to the change of directions in inverse channels of the system to the opposite directions.
More detailed information on the Potapov--Ginzburg transform can be found in the monograph \zitaa{MR2474532}{\cch{2}}.

\subsection{Example.
    Characteristic function of the Leibniz--Newton integral}\label{subsec3.5-0912}

In the space $L_2\rk{0,\ell}$ we consider  the operator
\beql{E3.29-0801}
\rk{\cI f}\rk{x}
\defeq\iu\int_x^\ell f\rk{t}\dif t.
\eeq
We note that
\beql{E3.30-0801}
\rk{\cI ^\ad f}\rk{x}
=-\iu\int_0^x f\rk{t}\dif t.
\eeq
From this it follows
\beql{E3.31-0801}
\rkb{\frac{\cI -\cI ^\ad}{\iu}f}\rk{x}
=\int_0^\ell f\rk{t}\dif t
=\ell\ipa{f}{h_0}h_0\rk{x},
\eeq
where
\beql{E3.32-0801}
h_0\rk{x}\defeq\frac{1}{\sqrt{\ell}},
\qquad x\in\ek{0,\ell};
\qquad \norm{h_0}=1.
\eeq
Hence, the operator $\cI $ has a one-dimensional imaginary part and
\beql{E3.33-0801}
\ipab{\frac{\cI -\cI ^\ad}{\iu}f}{f}
=\ell\absb{\ipa{f}{h_0}}^2
=\absb{\int_0^\ell f\rk{t}\dif t}
\geq0,
\eeq
\ednote{$\absb{\int_0^\ell f\rk{t}\dif t}$ oder $\absb{\int_0^\ell f\rk{t}\dif t}^2$ (Quadrat)?}\tie{}, the operator $\cI $ is dissipative.

Consider the operator $\Phi\in\ek{L_2\rk{0,\ell},\C}$ which is given by
\beql{E3.34-0514}
\Phi f
\defeq\sqrt{\ell}\ipa{f}{h_0}\tau.
\eeq
Here $\tau\in\C$ and $\tau=1$.
We note that
\[
\ipa{f}{\Phi^\ad \tau}
=\ipa{\Phi f}{\tau}
=\sqrt{\ell}\ipa{f}{h_0},
\]
\tie{}, $\Phi^\ad \tau=\sqrt{\ell}h_0$.
This enables us to rewrite \eqref{E3.33-0801}\ednote{\eqref{E3.33-0801} or \eqref{E3.31-0801}?} in the following way:
\[
\frac{\cI -\cI ^\ad}{\iu}
=\Phi^\ad J\Phi,
\]
where $J\defeq\IC $.
Thus, 
we see that
\beql{E3.34-0429}
\hat{\al}
\defeq\rkb{L^2\rk{0,\ell},\C;\cI ,\Phi,\IC}
\eeq
\ednote{$L^2\rk{0,\ell}$ or $L_2\rk{0,\ell}$?}is an operator colligation.
We note that
\[
\rk{\cI^nh_0}\rk{x}
=\frac{\iu^n}{\sqrt{\ell}}\rk{\ell-x}^n,
\qquad n=0,1,2,\dotsc.
\]
From this, turning our attention to the completeness of the system of polynomials in $L_2(0,\ell)$ and \rthm{T3.1}, \ednote{``we see'' einf\"ugen?}that the operator $\cI$ is complete nonselfadjoint and, consequently, the colligation $\hat{\al}$ is simple.

In the given case, $\cH=L_2(0,\ell)$ and
\beql{E3.34-0801}\begin{split}
    S_{\hat{\al}}\rk{z}
    &=\IC-\iu\Phi\rk{\cI -z\IH}^\inv\Phi^\ad\IC\\
    &=1-\iu\ipab{\Phi\rk{\cI -z\IH}^\inv\Phi^\ad \tau}{\tau}
    =1-\iu\ell\ipab{\rk{\cI -z\IH}^\inv h_0}{h_0}.
\end{split}\eeq

For arbitrary $z\neq0$, we have
\begin{multline*}
    \sum_{n=0}^\infty\frac{1}{z^{n+1}}\rk{\cI ^nf}\rk{x}
    =\frac{1}{z}f\rk{x}+\frac{\iu}{z^2}\biggl\{\int_x^\ell f\rk{t}\dif t+\frac{\iu}{z}\int_x^\ell\rk{t-x}f\rk{t}\dif t\\
    +\dotsb+\rkb{\frac{\iu}{z}}^{n-1}\int_x^\ell\frac{\rk{t-x}^{n-1}}{\rk{n-1}!}f\rk{t}\dif t+\dotsb\biggr\}.
\end{multline*}
Using the dominated convergence theorem for the term-by-term integration of the series, we obtain
\[
\rkb{\rk{\cI -z\IH}^\inv f}\rk{x}
=-\frac{1}{z}f\rk{x}-\frac{\iu}{z^2}\int_x^\ell \ec^{\frac{\iu}{z}\rk{t-x}}f\rk{t}\dif t.
\]
From this it follows $\sigma\rk{\cI }=\set{0}$.
Further, from \eqref{E3.34-0801} we get
\beql{E3.36-0505}\begin{split}
    S_{\hat{\al}}\rk{z}
    &=1-\iu\ell\ipab{-\frac{1}{z}h_0\rk{x}-\frac{\iu}{z^2}\int_x^\ell\ec^{\frac{\iu}{z}\rk{t-x}}h_0\rk{t}\dif t}{h_0\rk{x}}\\
    &=1-\iu\int_0^\ell\rkb{-\frac{1}{z}-\frac{\iu}{z^2}\int_x^\ell\ec^{\frac{\iu}{z}\rk{t-x}}\dif t}\dif x
    =\ec^{\iu\frac{\ell}{z}}.
\end{split}\eeq

In the given case, \tJ{properties} \eqref{E3.11}--\eqref{E3.13} (\tresp{}\ \eqref{E3.17}--\eqref{E3.19}) have the form:
\begin{align*}
    \absb{S\rk{z}}&>1,\quad\im z>0,\\
    \absb{S\rk{z}}&=1,\quad\im z=0,\\
    \absb{S\rk{z}}&<1,\quad\im z<0.
\end{align*}

\section{Some classes of characteristic functions.
    Divisors of characteristic operator functions}\label{S4}

\subsection{Linear fractional transformation of \tcof{}}\label{subsec4.1}
Starting from the colligation $\al\defeq\hgapj $, we form the operator function (see, \teg, \zitaa{MR0322542}{\cch{I}})
\beql{E4.1}
V_\al\rk{z}
\defeq\frac{1}{2}\Phi\rk{\reA  -z\IH}^\inv\Phi^\ad.
\eeq
The function $V_\al\rk{z}$ is holomorphic on the set $\rho(\reA  )$ of regular points of the operator $\reA  $, and its values, as the values of the function $S_\al\rk{z}$, are operators acting in $\cG$.
We note that $\rho(\reA  )$ contains all non-real points.
From the identity
\beql{E4.2}
V_\al\rk{z}-V_\al^\ad\rk{z}
=\iu\im z\Phi\rk{\reA  -\ko{z}\IH}^\inv\rk{\reA  -z\IH}^\inv\Phi^\ad
\eeq
it follows that
\begin{gather}
    \frac{V_\al\rk{z}-V_\al^\ad\rk{z}}{2\iu}\geq0,\qquad\im z>0,\label{E4.3}\\
    V_\al\rk{z}=V_\al^\ad\rk{z},\qquad\im z=0,\;z\in \rho(\reA  ),\label{E4.4}\\
    \frac{V_\al\rk{z}-V_\al^\ad\rk{z}}{2\iu}\leq0,\qquad\im z<0.\label{E4.5}
\end{gather}
\bthmnl{\cite{MR0062955}, \cite{MR0100793}, \zitaa{MR0322542}{\cch{I}}}{T4.1}
At each point of the set $\rho(A)\cap \rho(\reA  )$ there exist the operators
\[
\rkb{S_\al\rk{z}+\IG}^\inv,
\qquad\rkb{\IG+\iu V_\al\rk{z}J}^\inv,
\]
and the following identities
\begin{align}
    V_\al\rk{z}
    &=\iu\rkb{S_\al\rk{z}-\IG}\rkb{S_\al\rk{z}+\IG}^\inv J\notag\\
    &=\iu\rkb{S_\al\rk{z}+\IG}^\inv\rkb{S_\al\rk{z}-\IG}J,\label{E4.6}\\
    S_\al\rk{z}
    &=\rkb{\IG-\iu V_\al\rk{z}J}\rkb{\IG+\iu V_\al\rk{z}J}^\inv\notag\\
    &=\rkb{\IG+\iu V_\al\rk{z}J}^\inv\rkb{\IG-\iu V_\al\rk{z}J}\label{E4.7}
\end{align}
hold.
\ethm
\bproof
Indeed, in view of
\beql{E4.8}
\rk{\reA  -z\IH}^\inv-\rk{A-z\IH}^\inv
=\iu\rk{A-z\IH}^\inv \imA  \rk{\reA  -z\IH}^\inv,
\eeq
and, taking into account \eqref{E2.9}, we get
\begin{multline*}
    \Phi\rk{\reA  -z\IH}^\inv\Phi^\ad-\Phi\rk{A-z\IH}^\inv\Phi^\ad\\
    =\frac{1}{2}\iu\Phi\rk{A-z\IH}^\inv\Phi^\ad J\Phi\rk{\reA  -z\IH}^\inv\Phi^\ad.
\end{multline*}
Then
\[
V_\al\rk{z}+\frac{\iu}{2}\rkb{\IG-S_\al\rk{z}}J
=\frac{1}{2}\rkb{\IG-S_\al\rk{z}}V_\al\rk{z}.
\]
From this we infer
\beql{E4.9}
\rkb{S_\al\rk{z}+\IG}\rkb{\IG+\iu V_\al\rk{z}J}
=2\IG.
\eeq
Analogously, from the relation
\beql{E4.10}
\rk{\reA  -z\IH}^\inv-\rk{A-z\IH}^\inv
=\iu\rk{\reA-z\IH}^\inv\imA  \rk{A-z\IH}^\inv
\eeq
we obtain
\beql{E4.11}
\rkb{\IG+\iu V_\al\rk{z}J}\rkb{S_\al\rk{z}+\IG}
=2\IG.
\eeq
In view of \eqref{E4.9} and \eqref{E4.11}, each of the operators $S_\al\rk{z}+\IG$ and $\IG+\iu V_\al\rk{z}J$ has a bounded inverse for $z\in \rho(A)\cap \rho(\reA  )$.
Now formulas \eqref{E4.6} and \eqref{E4.7} follow from \eqref{E4.9}.
\eproof

We note that from formula \eqref{E4.6} it follows
\beql{E4.12-0711}\begin{split}
    \frac{V_\al\rk{z}-V_\al^\ad\rk{z}}{2\iu}
    &=\frac{1}{2}J\rkb{S_\al^\ad\rk{z}+\IG}^\inv\Bigl\{\rkb{S_\al^\ad\rk{z}+\IG}J\rkb{S_\al\rk{z}-\IG}\\
    &\qquad+\rkb{S_\al^\ad\rk{z}-\IG}J\rkb{S_\al\rk{z}+\IG}\Bigr\}\rkb{S_\al\rk{z}+\IG}^\inv J\\
    &=J\rkb{S_\al^\ad\rk{z}+\IG}^\inv\gkb{S_\al^\ad\rk{z}JS_\al\rk{z}-J}\rkb{S_\al\rk{z}+\IG}^\inv J.
\end{split}\eeq
Analogously, from \eqref{E4.7} we find
\beql{E4.13-0711}\begin{split}
    &S_\al^\ad\rk{z}JS_\al\rk{z}-J\\
    &=\rkb{\IG-\iu J V_\al^\ad\rk{z}}^\inv\Bigl\{\rkb{\IG+\iu JV_\al^\ad\rk{z}}J\rkb{\IG-\iu V_\al\rk{z}J}\\
    &\qquad\qquad\qquad-\rkb{\IG-\iu JV_\al^\ad\rk{z}}J\rkb{\IG+\iu JV_\al\rk{z}}\Bigr\}\rkb{\IG+\iu V_\al\rk{z}J}^\inv\\
    &=2\rkb{\IG-\iu J V_\al^\ad\rk{z}}^\inv J\frac{V_\al\rk{z}-V_\al^\ad\rk{z}}{\iu}J\rkb{\IG+\iu V_\al\rk{z}J}^\inv.
\end{split}\eeq
Hence, for $z\in\rho\rk{A}\cap\rho\rk{\reA  }$ conditions \eqref{E4.3}, \eqref{E4.4}, \eqref{E4.5} are equivalent to conditions \eqref{E3.11}, \eqref{E3.12}, \eqref{E3.13}, respectively.

\subsection{Class $\Om_J $}\label{subsec4.2-0713}
Let the linear operator $J$ which acts in the Hilbert space $\cG$ satisfy the conditions $J=J^\ad$, $J^2=\IG$.

\bdefnnl{\zitaa{MR0322542}{\cch{I}}}{D4.2-0629}
We say that the function $S(z)$ of the complex variable with values in $\ek{\cG}$ belongs to the class $\Om_J $ if it has the following properties:
\begin{Aeqi}{0}
    \il{D4.2-0629.1} $S(z)$ is holomorphic in some neighborhood $G_S$ of the infinitely distant point.
    \il{D4.2-0629.2} $\lim_{z\to\infty}\norm{S\rk{z}-\IG}=0$
    \il{D4.2-0629.3} For all $z\in G_S$ the operator $S\rk{z}+\IG$ has a bounded inverse.
    Moreover,
    \beql{E4.12-0629}
    V\rk{z}
    \defeq\iu\rkb{S\rk{z}-\IG}\rkb{S\rk{z}+\IG}^\inv J
    =\iu\rkb{S\rk{z}+\IG}^\inv\rkb{S\rk{z}-\IG}J
    \eeq
    is analytically continuable to the domain $G_V$ which is the extended complex plane with the exception of some bounded set of real points.
    \il{D4.2-0629.4}
    \beql{E4.13-0629}
    \frac{V\rk{z}-V^\ad\rk{z}}{2\iu}\geq0,\qquad\im z>0.
    \eeq
    \il{D4.2-0629.5}
    \beql{E4.14-0629}
    V\rk{z}=V^\ad\rk{z},\qquad\im z=0,\;z\in G_V.
    \eeq
\end{Aeqi}
\edefn

In view of \eqref{E4.12-0629}, we have
\beql{E4.15-0629}
\rkb{S\rk{z}+\IG}\rkb{\IG+\iu V\rk{z}J}
=\rkb{\IG+\iu V\rk{z}J}\rkb{S\rk{z}+\IG}
=2\IG.
\eeq
Thus, at each point $z\in G_S$ the operator
\[
\IG+\iu V\rk{z}J
\]
has a bounded inverse and
\beql{E4.16-0629}
S\rk{z}
=\rkb{\IG-\iu V\rk{z}J}\rkb{\IG+\iu V\rk{z}J}^\inv
=\rkb{\IG+\iu V\rk{z}J}^\inv\rkb{\IG-\iu V\rk{z}J}.
\eeq
From formulas of \rsubsec{subsec4.1} it follows that the \tcof{}\ $S_\al\rk{z}$ of an arbitrary colligation $\al\defeq\hgapj $ belongs to the class $\Om_J $.

\bthmnl{\zitaa{MR0322542}{\cthm{5.1}}}{T4.3-0629}
If the operator function $S(z)$ belongs to the class $\Om_J $, then  there exists a simple colligation $\al$ with the directing operator $J$ such that on some neighborhood of the infinitely distant point $S_\al\rk{z}=S\rk{z}$ holds.
\ethm
\bproof
We sketch the basic steps of the proof.
A detailed proof can be found, \teg{}, in \zitaa{MR0322542}{\cch{I}}.

From conditions~\eqref{D4.2-0629.1}--\eqref{D4.2-0629.5} of \rdefn{D4.2-0629} it follows that the function $V\rk{z}$ admits the integral representation
\beql{E4.17-0629}
V\rk{z}=\int_a^b\frac{\dif F\rk{t}}{t-z},
\qquad-\infty<a<b<+\infty,\;z\notin[a,b],
\eeq
where $F\rk{t}$ ($a\leq t\leq b$) is an operator-valued function with values in $\ek{\cG}$ such that
\[
F\rk{t}=F^\ad\rk{t},
\qquad F\rk{t_1}\leq F\rk{t_2},
\qquad t,t_1,t_2\in[a,b],\;t_1<t_2.
\]
Obviously, one can assume that $F\rk{a}=0$.

In view of M.~A.~Naimark (see, \teg{}, \zitaa{MR0044034}{\cch{IX}}, \zitaa{MR0322542}{\capp{1}}),
\beql{E4.18-0629}
F\rk{t}=\Psi E\rk{t}\Psi^\ad,
\qquad t\in[a,b],
\eeq
where $E\rk{t}$ is an orthogonal partition of unity in some Hilbert space $\cH$ and $\Psi\in\ek{\cH,\cG}$.
We define on $\cH$ the operators
\beql{E4.19-0629}
A\defeq\int_a^bt\dif E\rk{t}+\iu\Phi^\ad J\Phi,
\qquad\Phi\defeq\sqrt2\Psi
\eeq
and consider the colligation $\tilde\al\defeq\hgapj $.
From \eqref{E4.17-0629}--\eqref{E4.19-0629} and \eqref{E4.1} we obtain
\[
V\rk{z}
=\frac{1}{2}\Phi\int_a^b\frac{\dif E\rk{t}}{t-z}\Phi^\ad
=\frac{1}{2}\Phi\rk{\reA  -z\EM}^\inv\Phi^\ad
=V_{\tilde\al}\rk{z},\qquad z\notin[a,b].
\]
For this reason, in view of \eqref{E4.16-0629} and \eqref{E4.7}, there exists a neighborhood of the infinite distant point such that
\[\begin{split}
    S\rk{z}
    &=\rkb{\IG-\iu V\rk{z}J}\rkb{\IG+\iu V\rk{z}J}^\inv\\
    &=\rkb{\IG-\iu V_{\tilde\al}\rk{z}J}\rkb{\IG+\iu V_{\tilde\al}\rk{z}J}^\inv
    =S_{\tilde\al}\rk{z}.
\end{split}\]
Now it is sufficient to set $\al\defeq\tilde\al_\pri$ and take \eqref{E2.45-0622} into account.
\eproof

\bcorl{C4.4-0629}
If the function $S\rk{z}$ belongs to the class $\Om_J $, then in some neighborhood of the infinitely distant point it satisfies the conditions \eqref{E3.11}--\eqref{E3.13} and \eqref{E3.17}--\eqref{E3.19}.
\ecor

\bcorl{C4.5-0629}
There exists a neighborhood of the infinitely distant point such that the function $S\rk{z}\in\Om_J $ can be represented on it in the form
\beql{E4.20-0629}
S\rk{z}=\IG+\frac{\iu}{z}HJ+\dotsb,
\eeq
where $H\geq0$.
If $H=0$, then $S\rk{z}\equiv\IG$.
\ecor

\bcorl{C4.6-0629}
The product of two operator functions from class $\Om_J $ is also contained in $\Om_J $.
\ecor

The assertion follows from \rthm{T2.7}.

\bcorl{C4.7-0629}
If $S_j\rk{z}\in\Om_J $, $j=1,2$, and $S_1\rk{z}S_2\rk{z}=\IG$, $z\in G_{S_1}\cap G_{S_2}$, then $S_1\rk{z}=S_2\rk{z}=\IG$.
\ecor

The assertion follows from \rcor{C4.5-0629}.

\subsection{Divisors and regular divisors of functions of class $\Om_J $}\label{subsec4.3-0713}

\bdefnnl{\zitaa{MR0322542}{\cch{I}}}{D4.8-0707}
We say that the function $S_1\rk{z}\in\Om_J $ is a \emph{left} (\emph{right}) \emph{divisor} of a function $S_2\rk{z}\in\Om_J $ if there exists such function $S_{12}\rk{z}\in\Om_J $ ($S_{21}\rk{z}\in\Om_J $) such that on some neighborhood of the infinitely distant point the identity
\[
S_2\rk{z}=S_1\rk{z}S_{12}\rk{z}
\qquad(S_2\rk{z}=S_{21}\rk{z}S_1\rk{z})
\]
is satisfied.
The symbol $S_1\rk{z}\prec S_2\rk{z}$ means that $S_1\rk{z}$ is a left divisor of the function $S_2\rk{z}$.
\edefn

The relation $\prec$ generates a partial ordering in $\Om_J $.
Indeed, if $S_1\rk{z}\prec S_2\rk{z}$ and $S_2\rk{z}\prec S_1\rk{z}$ then
\begin{align*}
    S_2\rk{z}&=S_1\rk{z}S_{12}\rk{z},\qquad S_{12}\rk{z}\in\Om_J ,\\
    S_1\rk{z}&=S_{2}\rk{z}S_{21}\rk{z},\qquad S_{21}\rk{z}\in\Om_J ,
\end{align*}
and thus, $S_{1}\rk{z}=S_{1}\rk{z}S_{12}\rk{z}S_{21}\rk{z}$.
Hence, $S_{12}\rk{z}S_{21}\rk{z}=\IG$ which, in view of \rcor{C4.7-0629}, implies $S_{12}\rk{z}=S_{21}\rk{z}=\IG$ and $S_{1}\rk{z}=S_{2}\rk{z}$.
Moreover, it is immediately checked that the relation $\prec$ is transitive.

\bdefnnl{\zitaa{MR0322542}{\cch{I}}}{D4.9-0707}
Let $S_{1}\rk{z}\in\Om_J $ be a left divisor of the function $S_{2}\rk{z}\in\Om_J $:
\[
S_{2}\rk{z}=S_{1}\rk{z}S_{12}\rk{z},
\qquad S_{12}\rk{z}\in\Om_J .
\]
Then the function $S_{1}\rk{z}$ is called a \emph{regular left divisor} of the function $S_{2}\rk{z}$ if the product of simple colligations $\al_{1}\defeq\rk{\cH_{1},\cG;A_{1},\Phi_{1},J}$ and $\al_{12}\defeq\rk{\cH_{12},\cG;A_{12},\Phi_{12},J}$ for which $S_{\al_{1}}\rk{z}=S_{1}\rk{z}$ and $S_{\al_{12}}\rk{z}=S_{2}\rk{z}$ is a simple colligation.
\edefn

From \rthmss{T3.3-0630}{T3.4} it follows that this definition does not depend on the choice of the simple colligations $\al_{1}$ and $\al_{12}$.

Analogously, \emph{regular right divisors} are defined.
We will write $S_1\rk{z}\llcurly S_2\rk{z}$ if $S_1\rk{z}$ is a regular left divisor of $S_2\rk{z}$.

\bthmnl{\zitaa{MR0322542}{\cch{I}}}{T4.10-0707}
The relation $\llcurly$ is a partial ordering in the set $\Om_J $.
\ethm

The set of all left divisors of a given function $S\rk{z}\in\Om_J $ is in general larger than the set of all regular left divisors.
We give an example (see \zitaa{MR0322542}{\cch{I}}).
Let $\cH$ be a one-dimensional and $\cG$ be a two-dimensional space, $h\in\cH$ ($\norm{h}=1$), let $\set{g_1,g_2}$ be an orthonormal basis in $\cG$.
We introduce the operators $\Phi$ and $J$ by setting $\Phi h\defeq g_1+g_2$, $Jg_1\defeq g_1$, $Jg_2\defeq-g_2$.
Then $\Phi^\ad J\Phi=0$ and hence $\al\defeq\rk{\cH,\cG;0,\Phi,J}$ is a simple colligation.
It is easily checked that all functions
\[
S^{\rk{\sigma}}\rk{z}
\defeq\IG+\frac{\iu\sigma}{z}\Phi\Phi^\ad J,
\qquad0\leq\sigma\leq1,
\]
are left divisors of the function
\[
S_\al\rk{z}
=\IG+\frac{\iu}{z}\Phi\Phi^\ad J.
\]
On the other side, $S_\al\rk{z}$ has only two regular left divisors:
\begin{align*}
    S^{\rk{0}}\rk{z}&=\IG&
    &\text{and}&
    S^{\rk{1}}\rk{z}&=S_\al\rk{z}.
\end{align*}

\bthmnl{\zitaa{MR0322542}{\cch{I}}}{T4.11-0707}
On the set of all regular left divisors of the functions $S\rk{z}\in\Om_J $ the relations $\prec$ and $\llcurly$ coincide.
\ethm

\bthmnl{\zitaa{MR0322542}{\cch{I}}}{T4.12-0707}
Let $S\rk{z}\in\Om_J $ and $\al\defeq\hgapj $ be a simple colligation such that $S_\al\rk{z}=S\rk{z}$.
The function $S_1\rk{z}$ is a regular left divisor of $S\rk{z}$ if and only if it is the \tcof{}\ of some left divisor of the colligation $\al$.
\ethm
\bproof
Let $S_1\rk{z}$ be a regular divisor of the function $S\rk{z}$.
Then there exist simple colligations $\al_1'$ and $\al_2'$ such that
\[
S_{\al_1'}\rk{z}=S_1\rk{z},
\qquad
S\rk{z}=S_{\al_1'}\rk{z}S_{\al_2'}\rk{z}
\]
and $\al'\defeq\al_1'\al_2'$ is a simple colligation.
In view of \rthm{T2.7}, we have
\[
S_{\al'}\rk{z}
=S_{\al_1'\al_2'}\rk{z}
=S_{\al_1'}\rk{z}S_{\al_2'}\rk{z}
=S\rk{z}.
\]
From this and \rthm{T3.4} we obtain that the colligations $\al$ and $\al'$ are unitarily equivalent.
But then we infer from \rthm{T3.4-0630} that $\al=\al_1\al_2$, where $\al_1$ and $\al_2$ are unitarily equivalent to the colligations $\al_1'$ and $\al_2'$, respectively.
Then
\[
S_{\al_1}\rk{z}=S_{\al_1'}\rk{z}=S_{1}\rk{z}.
\]
The sufficiency of the condition of the theorem follows from \rthm{T2.7} and \rcor{C2.14-0622}.
\eproof

\bthmnl{\zitaa{MR0322542}{\cch{I}}}{T4.13-0707}
Let $\al_1$ and $\al_2$ be left divisors of the simple colligation $\al$.
Then the colligation $\al_1$ is a left divisor of the colligation $\al_2$ if and only if the function $S_{\al_1}\rk{z}$ is a left divisor of the function $S_{\al_2}\rk{z}$.
\ethm
\bproof
The necessity follows from \rthm{T2.7}.
If $S_{\al_1}\rk{z}\prec S_{\al_2}\rk{z}$, then from \rthmss{T4.12-0707}{T4.11-0707} it follows that $S_{\al_1}\rk{z}\llcurly S_{\al_2}\rk{z}$.
Applying again \rthm{T4.12-0707}, we infer $S_{\al_1}\rk{z}=S_{\al_1'}\rk{z}$, where $\al_1'$ is some left divisor of the colligation $\al_2$.
Since $\al_1$ and $\al_1'$ are left divisors of the colligation $\al$, then \rthm{T3.6-0630} implies $\al_1=\al_1'$.
\eproof

\subsection{Invariant subspaces of an operator and regular divisors of the characteristic operator function}\label{subsec4.4-0715}
We embed a given operator $A\in\ek{\cH}$ in a simple colligation $\al\defeq\hgapj $.
Let us consider the mapping $\cB$ which assigns to each subspace $\cH_0\subseteq\cH$ that is invariant with respect to $A$ the operator function $S_{\al_0}\rk{z}$, where $\al_0\defeq\proj{\al}{\cH_0}$.
The function $S_\al\rk{z}$ belongs to the class $\Om_J $ and $S_{\al_0}\rk{z}$ is a regular left divisor of it.

The following assertion is an important sharpening of \rthm{T3.7}.

\bthmnl{\zitaa{MR0322542}{\cch{I}}}{T4.14-0711}
The mapping $\cB$ has the following properties:
\begin{enuiard}
    \il{T4.14-0711.1} Every regular left divisor of the function $S_\al\rk{z}$ belongs to the set of values of the mapping $\cB$.
    \il{T4.14-0711.2} The mapping $\cB$ is bijective.
    \il{T4.14-0711.3} Let $\cH_1$ and $\cH_2$ be invariant subspaces with respect to $A$ and let $S_1\rk{z}$ and $S_2\rk{z}$ be the corresponding left regular divisors.
    Then $S_1\rk{z}$ is a left divisor of $S_2\rk{z}$ if and only if $\cH_1\subseteq\cH_2$.
\end{enuiard}
\ethm
\bproof
The assertion \ref{T4.14-0711.1}, \ref{T4.14-0711.2}, and \ref{T4.14-0711.3} follows from Theorem~\ref{T4.12-0707},~\ref{T3.6-0630}, and~\ref{T4.13-0707}, respectively.
\eproof

\bdefnl{D4.15-0429}
A linear bounded operator $A$ is called \emph{unicellular operator}, if of two arbitrary invariant subspaces \ednote{Fehlt hier etwas?}is a part of the other one.
\edefn

From \rthm{T4.14-0711} it follows easily:

\bthmnl{\zitaa{MR0322542}{\cch{I}}}{T4.16-0429}
Let $S\rk{z}$ be the characteristic function of the simple colligation $\al=\ocol{\cH}{\cG}{A}{\Phi}{J}$.
Then the operator $A$ is a unicellular operator if and only if the set of regular left divisors of the function $S\rk{z}$ is an ordered set.
\ethm

The definition of a unicellular operator is due to M.~S.~Brodskii\footnote{M.~S.~Brodskii: \emph{On Jordan cells of infinite-dimensional operators. (Russian)}
Dokl. Akad. Nauk SSSR (N.S.) 111, No.~5 (1956), 926--929.}.
Obviously the notion ``one block operator'' generalizes the notion ``Jordan block'' to the infinite dimensional case.

\subsection{Quasihermitian colligations. Class $\Om_J^\kqk$}\label{subsec4.5-0713}
\subsubsection{Invertible holomorphic operator functions}\label{subsubsec4.5.1-0711}
Let a holomorphic function $T\rk{z}$ be given in a domain $G$ of the complex plane the values of which are bounded linear operators acting in a Hilbert space $\cH$.
We denote by $G_0$ the set of all points of the domain $G$ for which the operator $T\rk{z}$ has not a bounded inverse which is defined on the whole $\cH$.
Then $G\setminus G_0$ is an open set and the function $T^\inv\rk{z}$ is holomorphic on it.
We will call the function $T\rk{z}$ \emph{invertible in} $G$ if it satisfies the following conditions:
\begin{enuiard}
    \il{subsubsec4.5.1-0711.1} The set $G_0$ has not limit points within $G$.
    \il{subsubsec4.5.1-0711.2} Every point $z\in G_0$ is a pole of the function $T^\inv\rk{z}$.
    \il{subsubsec4.5.1-0711.3} The coefficients associated with negative powers of $z-z_0$ in the Laurent series of $T^\inv\rk{z}$ in the neighborhood of an arbitrary pole $z_0$ are finite dimensional.
\end{enuiard}

\bthml{T4.15-0711}
Let $K\rk{z}$ be a function holomorphic on the domain $G$ the values of which are compact operators in Hilbert space $\cH$.
If for some $z_0\in G$ the operator $\EM+K\rk{z_0}$ has a bounded inverse which is defined on $\cH$, then the function $T\rk{z}\defeq\EM+K\rk{z}$ is invertible in $G$.
\ethm
\bproof
We first consider the case that all vectors $K\rk{ z }h$ ($ z \in G$, $h\in\cH$) belong to a finite dimensional subspace $\cH_0\subseteq\cH$ and denote by $T_0\rk{ z }$ the operator which is induced by $T\rk{ z }$ on $\cH_0$.
The operator $T_0\rk{ z _0}$ has an inverse, since in the opposite case there would exist a nonzero vector $h_0\in\cH_0$ which is annihilated by the operator $T_0\rk{ z _0}$.
This vector is annihilated even by the operator $T\rk{ z _0}$ which contradicts the assumptions of the theorem.
Thus, $\det T_0\rk{ z _0}\neq0$ and hence the set of all zeros of the holomorphic function $\det T_0\rk{ z }$ does not have limit points within $G$.
Applying the rules of determining the elements of the inverse matrix, we infer that the singularities of the function $T_0^\inv\rk{ z }$ can be only points which belong to $G_0$.
Denoting by $P_0$ and $P_1$ the orthoprojectors onto $\cH_0$ and $\cH_1\defeq\cH\ominus\cH_0$, we obtain
\[\begin{split}
    T\rk{ z }
    &=\rk{P_0+P_1}T\rk{ z }\rk{P_0+P_1}\\
    &=P_0T\rk{ z }P_0+P_0T\rk{ z }P_1+P_1T\rk{ z }P_1
    =T_0\rk{ z }P_0+P_0T\rk{ z }P_1+P_1.
\end{split}\]
Thus,
\[
T^\inv\rk{ z }
=T_0^\inv\rk{ z }P_0-T_0^\inv\rk{ z }P_0T\rk{ z }P_1+P_1,
\]
from which it can be seen that $T\rk{ z }$ has an inverse function in all points of its domain.

Now let the dimension of the subspace $\cH_0$ be arbitrary and let $G$ be the set of inner points of some disk.
Without loss of generality one can assume that its center is located in the origin and its radius is one.
In the disk $\abs{ z }<\rho$ ($\abs{ z _0}<\rho<1$) the function $K\rk{ z }$ can be represented by the series
\[
K\rk{ z }
=K_0+ z  K_1+ z ^2 K_2+\dotsb
\]
which converges uniformly in the norm with respect to $ z $, whence
\[
\normb{K\rk{ z }-\sum_{j=0}^N z ^jK_j}<\frac{1}{4}
\qquad(\abs{ z }<\rho)
\]
holds for some $N$.
Using that all operators $K_j$ are compact, we choose finite dimensional operators $R_j$ such that
\[
\norm{K_j-R_j}<\frac{1}{4\rk{N+1}}
\qquad(j=0,1,\dotsc).
\]
Setting $R\rk{ z }\defeq\sum_{j=0}^N z ^jR_j$, we obtain the equality $K\rk{ z }=R\rk{ z }+S\rk{ z }$, where $\norm{S\rk{ z }}<\frac{1}{2}$ ($\abs{ z }<\rho$).
For $\abs{ z }<\rho$ the series $\sum_{j=0}^\infty\rk{-1}^jS^j\rk{ z }$ converges uniformly and is a holomorphic function equal to $\rk{\EM+S\rk{ z }}^\inv$.
Writing $T\rk{ z }$ in the form
\[
T\rk{ z }
=\EM+R\rk{ z }+S\rk{ z }
=\ekb{\EM+R\rk{ z }\rkb{\EM+S\rk{ z }}^\inv}\rkb{\EM+S\rk{ z }}
\]
and taking into account that, as mentioned above, the function in square brackets is invertible in the disk $\abs{ z }<\rho$, we infer that the function $T\rk{ z }$ is also invertible in this disk.
In view of the arbitrariness of the choice of $\rho$, the function $T\rk{ z }$ is invertible in the disk $\abs{ z }<1$.

In order to obtain the proof of the theorem free of restrictions we consider an arbitrary circle $\ga$ whose interior is contained in the domain $G$ and construct circles $\ga_0,\ga_1,\dotsc,\ga_n=\ga$ having the following properties:
1) The center of the circle $\ga_0$ is located at the point $ z _0$;
2) interior points of all circles $\ga_j$ are located in $G$;
3) the circles $\ga_j$ and $\ga_{j+1}$ ($j=0,1,\dotsc,n-1$) have common interior points.
Since $T\rk{ z }$ is invertible within $\ga_0$, the operator $T^\inv\rk{ z }$ exists in some point located within $\ga_1$.
Continuing this procedure, we obtain the invertibility of $T\rk{ z }$ within $\ga$.
It remains to remark that the invertibility of $T\rk{ z }$ in the interior of an arbitrary open disk contained in $G$ implies the invertibility of $T\rk{ z }$ in the whole domain $G$.
\eproof

\rthm{T4.15-0711} (without assumption of finite dimension of the coefficients of Laurent series) was obtained by I.~Ts.~MR0264447 Reviewed Gohberg, I. C.; Kre\u{\i}n, M. G. Theory and applications of Volterra operators in Hilbert space. Translated from the Russian by A. Feinstein Translations of Mathematical Monographs, Vol. 24 American Mathematical Society, Providence, R.I. 1970 x+430 pp. 47.45 (47.40)
Review PDF Clipboard Series Book 271 Citations
MR0246142 Reviewed Gohberg, I. C.; Kre\u{\i}n, M. G. Introduction to the theory of linear nonselfadjoint operators. Translated from the Russian by A. Feinstein Translations of Mathematical Monographs, Vol. 18 American Mathematical Society, Providence, R.I. 1969 xv+378 pp. 47.10berg \cite{MR0042060}.
The given proof is due to Yu.~P.~Ginzburg.
In this section we followed \zitaa{MR0322542}{\capp{II}}.

\subsubsection{Quasihermitian operators}
A bounded linear operator is called \emph{quasihermitian} if its imaginary part is compact.
This class of non-selfadjoint operators could be investigated more completely by the methods of \tcof{}
One of the basic achievements of the theory of \tcof{}\ was the construction 
of triangular models for this class of operators 
(see \cite{MR0062955,MR0080269,MR0100793}).
\rsec{S1551} of this paper is dedicated to this question.
An important role in these researches is played by the nature of the spectrum of 
quasihermitian operators.


\bthmnl{\cite{MR0062955}, \cite{MR0100793}, \zitaa{MR0322542}{\cch{I}}}{T4.16-0715}
The limit points of the spectrum of a quasihermitian operator $A$ belong to the spectrum of its real part.
Every non-real point of the spectrum of the operator $A$ is an eigenvalue of finite multiplicity.
\ethm
\bproof
Let $\rho>\norm{A}$.
In view of $\norm{\reA  }\leq\norm{A}$, the spectrum of $\reA  $ is contained in the disk $\abs{z}<\rho$.
Removing from this disk the spectrum of the operator $\reA  $, we obtain a domain $G$ on which the function $\rk{\reA  -z\IH}^\inv$ is holomorphic.
The obvious identity
\[
A-z\IH
=\rk{\reA  -z\IH}\ekb{\IH+\iu\rk{\reA  -z\IH}^\inv \imA  },
\qquad z\in G,
\]
and \rthm{T4.15-0711} applied to the function
\[
K\rk{z}
\defeq\iu\rk{\reA  -\iu\IH}^\inv \imA  
\]
show that the spectrum of the operator does not have limit points in $G$.
Thus, the first assertion of the theorem is proved.

Let a non-real $z_0$ belong to the spectrum of the operator $A$.
In view of \rthm{T4.15-0711}, again there exists a neighborhood of $z_0$ in which we have the expansion
\[
\rk{A-z\IH}^\inv
=\frac{A_{-n}}{\rk{z-z_0}^n}+\dotsb+\frac{A_{-1}}{z-z_0}+A_0+\rk{z-z_0}A_1+\dotsb,
\]
where $A_{-n},\dotsc,A_{-1}$ are finite dimensional operators.
In view of a theorem due to F.~Riesz (see \cite{MR0056821}), the space $\cH$ in which acts the operator $A$ can be uniquely represented as a direct sum of two invariant subspaces $\cH_0$ and $\cH_1$ with respect to $A$ such that
\[
\sigma\rk{\rstr{A}{\cH_0}}=\set{z_0},
\qquad\sigma\rk{\rstr{A}{\cH_1}}=\sigma\rk{A}\setminus\set{z_0}.
\]
Here the operator $P_0$ which projects onto $\cH_0$ parallel to $\cH_1$ is computed by the formula
\[
P_0
=-\frac{1}{2\pi\iu}\int_\ga\rk{A-z\IH}^\inv\dif z,
\]
where $\ga$ is a sufficiently small oriented neighborhood of the point $z_0$.
In the considered case
\[
P_0
=-\frac{1}{2\pi\iu}\int_\ga\frac{A_{-1}}{z-z_0}\dif z
=-A_{-1},
\]
which implies that $\cH_0$ is finite dimensional.
\eproof

\subsubsection{Quasihermitian colligations}
The colligation $\al\defeq\hgapj $ is called \emph{quasihermitian} if the operator $\Phi$ is compact. 
In view of the conditions of the colligation \eqref{E2.9} the fundamental operator of a quasihermitian colligation is quasihermitian.
The converse assertion is not generally true.
However, each completely non-selfadjoint quasihermitian operator admits an embedding into a simple quasihermitian colligation.
This can be checked by choosing $\cE\defeq\ran{\imA  }$ in \rthm{T2.2}.
This means embedding operator $A$ in the colligation $\al_{\imA  }$ (see \eqref{E1330}).

From the last relation in \eqref{E2.29} and form \eqref{E3.8} of the projection of the colligation, respectively, it follows that the product of quasihermitian colligations and the projection of a quasihermitian colligation onto an arbitrary subspace are again quasihermitian.

\bthmnl{\zitaa{MR0322542}{\cch{I}}}{T4.17-0715}
Let the colligation $\al\defeq\hgapj $ be the product of quasihermitian colligations $\al_j\defeq\rk{\cH_j,\cG;A_j,\Phi_j,J}$, $j=1,2$.
Then
\[
\sigma\rk{A}
=\sigma\rk{A_1}\cup\sigma\rk{A_2}.
\]
\ethm
\bproof
Since the operator $A$ is quasihermitian, then the set $\rho\rk{A}$ is connected.
For this reason the assertion follows from \rthm{T2.9-0715}.
\eproof

\subsubsection{The class $\Om_J^\kqk $}
Let $\cG$ be a Hilbert space and let $J\in\ek{\cG}$ be an operator which fulfills the conditions $J=J^\ad$ and $J^2=\IG$.
We say that an operator function $S\rk{z}$ with values in $\ek{\cG}$ belongs to \emph{the class} $\Om_J^\kqk $ if
\begin{Aeqiq}{0}
    \il{0721.I} $S\rk{z}$ is holomorphic in a domain $G_S$ which is obtained by removing from the extended complex plane some bounded subset which does not have non-real limit points;
    \il{0721.II} $\lim_{z\to\infty}\norm{S\rk{z}-\IG}=0$;
    \il{0721.III} $S^\ad\rk{z}JS\rk{z}-J\geq0$, $\im z>0$, $z\in G_S$;
    \il{0721.IV} $S^\ad\rk{z}JS\rk{z}-J=0$, $\im z=0$, $z\in G_S$;
    \il{0721.V} All operators $S\rk{z}-\IG$, $z\in G_S$, are compact.
\end{Aeqiq}
Let $S_j\rk{z}\in\Om_J^\kqk $, $j=1,2$, and $S\rk{z}=S_1\rk{z}S_2\rk{z}$.
From the identities
\begin{gather}
    S^\ad\rk{z}JS\rk{z}-J=S_2^\ad\rk{z}\rkb{S_1^\ad\rk{z}JS_1\rk{z}-J}S_2\rk{z}+S_2^\ad\rk{z}JS_2\rk{z}-J,\label{E4.23-0721}\\
    S\rk{z}-\IG=\rkb{S_1\rk{z}-\IG}\rkb{S_2\rk{z}-\IG}+\rkb{S_1\rk{z}-\IG}+\rkb{S_2\rk{z}-\IG}\label{E4.24-0721}
\end{gather}
it follows that the class $\Om_J^\kqk $ contains, together with the functions $S_1\rk{z}$ and $S_2\rk{z}$, also their product $S\rk{z}$.

It is obvious that the \tcof{}\ of a quasihermitian colligation belongs to the class $\Om_J^\kqk $.

\bthmnl{\cite{MR0062955,MR0100793}, \zitaa{MR0322542}{\cch{I}}}{T4.18-0721}
The inclusion
\[
\Om_J^\kqk 
\subset\Om_J
\]
holds.
\ethm
\bproof
Let $S\rk{z}\in\Om_J^\kqk $.
Denote by $G_S^{\rk{0}}$ the set of all points $z$ of the domain $G_S$ for which the operator $S\rk{z}+\IG$ has a bounded inverse which is defined on the whole space $\cG$.
From \eqref{0721.II} it follows that there is contained in $G_S^{\rk{0}}$ some neighborhood of the infinitely distant point.
We note that
\[
S\rk{z}+\IG
=2\rkb{\IG+K\rk{z}},
\]
where $K\rk{z}\defeq\frac{1}{2}\rk{S\rk{z}-\IG}$ is a holomorphic function in $G_S$ which takes compact values.
Thus, it follows from \rthm{T4.15-0711} that the set $G_S\setminus G_S^{\rk{0}}$ has not limit points in $G_S$.

We consider the function (see \eqref{E4.6} and \eqref{E4.12-0629})
\[
V\rk{z}
\defeq\iu\rkb{S\rk{z}-\IG}\rkb{S\rk{z}+\IG}^\inv J,\qquad z\in G_S^{\rk{0}}.
\]
In view of \eqref{E4.12-0711}, we have
\begin{gather}
    \frac{V\rk{z}-V^\ad\rk{z}}{2\iu}\geq0,\qquad\im z>0,\;z\in G_S^{\rk{0}},\label{E4.25-0721}\\
    V\rk{z}=V^\ad\rk{z},\qquad\im z=0,\;z\in G_S^{\rk{0}}.\label{E4.26-0721}
\end{gather}
We assume that $V\rk{z}$ has an isolated singularity $z_0$ in the upper half plane.
Then, using a corresponding choice of the vector $g_0\in\cG$, the point $z_0$ becomes also an isolated singularity for the function $\ipa{V\rk{z}g_0}{g_0}$ which contradicts inequality \eqref{E4.25-0721}.
Taking into account the properties of the sets $G_S$ and $G_S^{\rk{0}}$, we see that $V\rk{z}$ is analytically continuable in the whole upper half plane.
In view of \eqref{E4.26-0721} and the principle of symmetry, it can be analytically continued also in the whole lower half plane.
\eproof

\bthmnl{\cite{MR0062955,MR0100793}, \zitaa{MR0322542}{\cch{I}}}{T4.19-0721}
If $S\rk{z}\in\Om_J^\kqk $, then there exists a simple quasihermitian colligation $\al\defeq\hgapj $ such that $G_S\subseteq\rho\rk{A}$ and $S_\al\rk{z}=S\rk{z}$, $z\in G_S$.
\ethm
\bproof
Since $\Om_J^\kqk \subset\Om_J$, then in view of \rthm{T4.3-0629}, there exists a simple colligation $\al=\hgapj $ such that $S_\al\rk{z}=S\rk{z}$ for all $z$ which belong to some neighborhood of the infinitely distant point.
Thus, for sufficiently large values of $\abs{z}$ we have
\[
z\rkb{S\rk{z}-\IG}
=z\rkb{S_\al\rk{z}-\IG}
=-\iu\Phi\Phi^\ad J-\frac{\iu}{z}\Phi A\Phi^\ad J-\frac{\iu}{z^2}\Phi A^2\Phi^\ad J-\dotsb.
\]
From this and conditions \eqref{0721.II} and \eqref{0721.V} it follows that the operator
\[
\Phi\Phi^\ad
=\iu\lim_{z\to\infty}\rkb{S\rk{z}-\IG}J
\]
is compact.
This means that $\Phi$ is a compact operator and $\al$ a quasihermitian colligation.
Now the assertions $G_S\subseteq\rho\rk{A}$ and $S_\al\rk{z}=S\rk{z}$, $z\in G_S$, are obtained from the following statement.
\eproof

\bthmnl{\cite{MR0020719,MR0100793}, \zitaa{MR0322542}{\cch{I}}}{T4.20-0721}
Let $\al\defeq\hgapj $ be a simple quasihermitian colligation.
If $z_0\in\sigma\rk{A}$, then there does not exist a disk $\ck\defeq\setaca{z}{\abs{z-z_0}<r}$ on which a holomorphic function can be defined that coincides with $S_\al\rk{z}$ at all points $z\in\rho\rk{A}\cap\ck$.
\ethm

\subsection{$\cG$-finite dimensional colligations.
    Class $\Om_J^\kek$}\label{subsec4.6-1001}
The colligation $\al=\hgapj $ is called \emph{\tG{finite} dimensional} if its external space is finite dimensional.
All results of the preceding subsection are true for \tG{finite} dimensional colligations.

Let $\cG$ be a finite dimensional Hilbert space and let the operator $J\in\ek{\cG}$ satisfy the conditions $J=J^\ad$ and $J^2=\IG$.
An operator function $S\rk{z}$ of the complex variable $z$ with values in $\ek{\cG}$ is said to belong to \emph{the class} $\Om_J^\kek $ if
\begin{Aeqie}{0}
    \il{0721.Ie} $S\rk{z}$ is holomorphic in a domain $G_S$, which is obtained by taking off from the extended complex plane some bounded subset which has not non-real limit points;
    \il{0721.IIe} $\lim_{z\to\infty} S\rk{z}=\IG$;
    \il{0721.IIIe} $S^\ad\rk{z}JS\rk{z}-J\geq0$, $\im z>0$, $z\in G_S$;
    \il{0721.IVe} $S^\ad\rk{z}JS\rk{z}-J=0$, $\im z=0$, $z\in G_S$.
\end{Aeqie}

Obviously, a \tcof{}\ of a finite dimensional colligation $\al\defeq\hgapj $ belongs to the class $\Om_J^\kek $.
On the other side, as it follows from \rthm{T4.19-0721}, in the case of a finite dimensional space $\cG$ the equalities $\Om_J=\Om_J^\kqk =\Om_J^\kek $ hold.
Thus, we have the following statement.

\bthmnl{\cite{MR0062955,MR0100793}, \zitaa{MR0322542}{\cch{I}}}{T4.21-0721}
If $S\rk{z}\in\Om_J^\kek $, then there exists a simple \tG{finite} dimensional colligation $\al=\hgapj $ such that $G_S\subseteq\rho\rk{A}$ and $S_\al\rk{z}=S\rk{z}$, $z\in G_S$.
\ethm

\breml{R4.22-10-01}
Let $\cG$ be a finite dimensional Hilbert space, $\dim\cG=r<\infty$.
We will identify linear operators in $\cG$ with square matrices of $r$\nobreakdash-th order, assuming that an orthonormal basis in $\cG$ is chosen.
In this case we will identify the class $\Om_J^\kek $ of operator functions with the class $\Om_J\rk{\C^r}$ of matrix functions of $r$\nobreakdash-th order that satisfy conditions~\eqref{0721.Ie}--\eqref{0721.IVe}.
Note that the following results can be generalized to the infinite dimensional case.
\erem

\subsection{On one model representation of a bounded linear operator which is simple with respect to the finite dimensional subspace}
Let $\cH$ be a Hilbert space and $A\in\ek{\cH}$.

\begin{defn}\label{D4.23-1015}
    The operator $A$ is called \emph{simple with respect to a subspace} $\cE \subseteq \cH$ if the following conditions are satisfied
    \begin{enumerate}
        \item[1)] $\ran{\imA} \subseteq \cE$;
        \item[2)] the space $\cH$ is the closure of the linear hull of vectors of the form
        \[
        A^n h,
        \qquad n=0,1,2,\dotsc;h\in\cE.
        \]
    \end{enumerate}
\end{defn}

Assume that the operator A is simple with respect to a finite dimensional
subspace $\cE$. In this case, taking into account \rthm{T2.2} and \rrem{R2.3-0929}, we 
embed the operator $A$ in a simple \tG{finite} dimensional colligation
\beql{E4.27-0801}
\al
:=\rk{\cH,\C^r;A,\Phi,J}
\eeq
such that $\cE = \ran{{\Phi}^*}$.
Further, consider the matrix function of $r$\nobreakdash-th order (see \eqref{E4.1})
\[
V_\al\rk{z}
:=\frac{1}{2}\Phi\rk{\reA -z\IH}^\inv\Phi^\ad.
\]
We note that the \tcof{}\ $S_\al\rk{z}$ is expressed in terms of $V_\al\rk{z}$ by \eqref{E4.7}.
The function $V_\al\rk{z}$ admits the representation (see \eqref{E4.17-0629})
\beql{E4.28-0801}
V_\al\rk{z}=\int_a^b\frac{\dif\sigma\rk{t}}{t-z},
\qquad-\infty<a<b<+\infty,\;z\notin[a,b],
\eeq
where $\sigma\rk{t}$ is a monotonically increasing matrix function of $r$\nobreakdash-th order on $[a,b]$.

Let $\cHt$ be the set of all continuous functions  $h\rk{t}$ on $[a,b]$ with values in $\C^r$.
We introduce in $\cHt$ a scalar product by setting
\[
\ipa{h'}{h''}_{\cHt}
\defeq2\int_a^b\ipa{\dif\sigma\rk{t}h'\rk{t}}{h''\rk{t}},
\]
where the functions $h'\rk{t}$ and $h''\rk{t}$ be considered as equivalent if
\[
\norm{h'-h''}_{\cHt}^2
=\ipa{h'-h''}{h'-h''}_{\cHt}
=0.
\]
After identifying equivalent elements we complete $\cHt$ to a Hilbert space which we denote by $\cH_{\cm}$.

We define on $\cHt$ an operator $\At$ by the formula
\[
(\At h)\rk{t}
\defeq th\rk{t}+\iu J\int_a^b\dif\sigma\rk{s}h\rk{s},
\]
where $J$ is the directing operator of the colligation \eqref{E4.27-0801}.
Obviously, each of the operators
\[
(\At_1h)\rk{t}\defeq th\rk{t},
\qquad(\At_2h)\rk{t}\defeq J\int_a^b\dif\sigma\rk{s}h\rk{s}
\]
is bounded and selfadjoint.
The operator $\At_2$ maps the space $\cHt$ into the finite dimensional space consisting of constant functions on $[a,b]$.
Moreover the operator $\At$ maps equivalent functions to equivalent functions.
Thus, the operator $\At$ can be continuously extended to the whole space $\cH_{\cm}$ by maintaining the norm.
We denote by $A_{\cm}$ the extension of $\At$.

Consider the operator $\Phit\in\ek{\cHt,\C^r}$, which is defined by
\beql{E4.29-0801}
(\Phit h)\rk{t}
\defeq2\int_a^b\dif\sigma\rk{s}h\rk{s}.
\eeq
Obviously, $\Phit$ is a bounded operator which images on equivalent functions coincide.
The continuous extension of the operator $\Phit$ on $\cH_{\cm}$ we denote by $\Phi_{\cm}$.
For arbitrary $h\in\cHt$ and $g\in\C^r$ it holds
\[
\ipa{\Phi_{\cm}h}{g}
=\ipab{2\int_a^b\dif\sigma\rk{s}h\rk{s}}{g}
=2\int_a^b\ipab{\dif\sigma\rk{s}h\rk{s}}{g}
=\ipa{h}{h_g}_{\cH_{\cm}},
\]
where $h_g\rk{t}\defeq g$, $t\in[a,b]$.
Thus,
\beql{E4.30-0801}
(\Phi_{\cm}^\ad g)\rk{t}=h_g\rk{t},
\qquad g\in\C^r,\;t\in[a,b].
\eeq
Thus, for $h\in\cH_{\cm}$ we have
\[
\frac{1}{\iu}\rk{A_{\cm}-A_{\cm}^\ad}h\rk{t}
=2J\int_a^b\dif\sigma\rk{s}h\rk{s}
=\Phi_{\cm}^\ad J\Phi_{\cm}h\rk{t}.
\]
Hence,
\beql{E4.31-0801}
\al_{\cm}
\defeq\rk{\cH_{\cm},\C^r;A_{\cm},\Phi_{\cm},J}
\eeq
is an operator colligation.
The colligation $\al_{\cm}$ will be called \emph{model} colligation for
the operator $A$.

We note that
\beql{E4.32-0801}
\rkb{\reA_{\cm} -z\Iu{\cH_{\cm}}}^{-1}h\rk{t}
=\frac{h\rk{t}}{t-z}.
\eeq
Taking into account \eqref{E4.29-0801}, \eqref{E4.30-0801} and \eqref{E4.28-0801}, for $g\in\C^r$ we obtain
\beql{E4.33-0801}
V_{\al_{\cm}}\rk{z}g
=\frac{1}{2}\Phi_{\cm}\rkb{\reA_{\cm} -z\Iu{\cH_{\cm}}}^\inv\Phi_{\cm}^\ad g
=\int_a^b\frac{\dif\sigma\rk{t}}{t-z}g
=V_\al\rk{z}g.
\eeq
Now it follows from \eqref{E4.7} that in some neighborhood of the infinitely distant point we have
\beql{E4.34-0801}
S_{\al_{\cm}}\rk{z}
=S_\al\rk{z}.
\eeq

\bthmnl{\cite{MR0100793}}{T4.22-0801}
The operator colligation $\al$ (see \eqref{E4.27-0801}) is unitarily equivalent to the model colligation $\al_{\cm}$ (see \eqref{E4.31-0801}).
In particular, each bounded linear operator $A$ with finite dimensional imaginary part is unitarily equivalent to the model operator $A_{\cm}$ defined on $\cHt$ by the formula
\[
(A_{\cm}h)\rk{t}
:=th\rk{t}+\iu J\int_a^b\dif\sigma\rk{s}h\rk{s}, \ \ h \in \cHt,
\]
and extended on $\cH_{\cm}$ as described above.
\ethm
\bproof
The colligation $\al$ is simple.
Thus, in view of \eqref{E4.34-0801}, it follows from \rthm{T3.4} that it is sufficient to verify the simplicity of the model colligation $\al_M$.

Note that the vector function
\[
g_z\rk{t}:=\frac{1}{2\iu\rk{t-z}}J\rkb{S_{\al_{\cm}}\rk{z}+\IG}g,
\qquad g\in\C^r,\;\norm{z}>\norm{A_{\cm}},\;\im z\neq0,
\]
belongs to the space $\cHt$ for every fixed $z$ and, taking into account \eqref{E4.33-0801}, \eqref{E4.6} and \eqref{E4.30-0801} we obtain
\[\begin{split}
    A_{\cm}g_z\rk{t}
    &=tg_z\rk{t}+\iu J\int_a^b\dif\sigma\rk{s}g_z\rk{s}\\
    &=zg_z\rk{t}+\rk{t-z}g_z\rk{t}+\frac{1}{2}J\int_a^b\frac{\dif\sigma\rk{s}}{t-s}J\rkb{S_{\al_{\cm}}\rk{z}+\IG}g\\
    &=zg_z\rk{t}+\frac{1}{2\iu}J\rkb{S_{\al_{\cm}}\rk{z}+\IG}g-\frac{1}{2\iu}J\rkb{S_{\al_{\cm}}\rk{z}-\IG}g\\
    &=zg_z\rk{t}-\iu Jg
    =zg_z\rk{t}-\iu h_{Jg}\rk{t}
    =zg_z\rk{t}-\iu\Phi_{\cm}^\ad Jg.
\end{split}\]
This means
\beql{E4.35-0801}
\rk{A_{\cm}-z\Iu{\cH_{\cm}}}^\inv\Phi_{\cm}^\ad Jg
=\iu g_z\rk{t}
=\frac{J\rk{S_{\al_{\cm}}\rk{z}+\IG}}{2\rk{t-z}}g.
\eeq

Let there exist a vector $h\in\cH_{\cm}$ such that
\[
\ipa{h}{A_{\cm}^n\Phi_{\cm}^\ad Jg}_{\cH_{\cm}}=0,
\qquad g\in\C^r,\;n=0,1,2,\dotsc
\]
Then
\[
\ipab{h}{\rk{A_{\cm}-z\Iu{\cH_{\cm}}}^\inv\Phi_{\cm}^\ad Jg}_{\cH_{\cm}}=0,
\qquad g\in\C^r.
\]
From this, in view of \eqref{E4.35-0801} it follows
\[
\rkb{S_{\al_{\cm}}\rk{z}+\Iu{\cH_{\cm}}}^\ad J\int_a^b\frac{\dif\sigma\rk{s}h\rk{s}}{s-\ko{z}}
=0.
\]
Since the operator $\rk{S_{\al_{\cm}}\rk{z}+\Iu{\cH_{\cm}}}$ is invertible, from this we obtain
\[
\int_a^b\frac{\dif\sigma\rk{s}h\rk{s}}{s-z}
=0
\]
and this means that the function $h(t)$ is equivalent to the null function.
\eproof

We note that with this method, but with additional technical means, this result can be obtained for quasihermitian operators (see \cite{MR0080269}).

\subsection{Multiplicative representation of the characteristic function of an operator acting in a finite dimensional space}\label{subsec4.8-0917}
Let $A$ be a linear operator acting in a finite dimensional space $\cH$.
We embed the operator $A$ in a simple \tG{finite} dimensional colligation (see \rthm{T2.2} and \rrem{R2.3-0929})
\beql{E4.36-0805}
\al
:=\rk{\cH,\C^r;A,\Phi,J}.
\eeq

Let $\set{\lambda_j}_{j=1}^n$ be the eigenvalues of the operator $A$, taking into account their multiplicity.
We choose an orthonormal basis $\set{h_j}_{j=1}^n$ in $\cH$ such that with respect to it the matrix of the operator $A$ has upper triangular form:
\begin{align*}
    Ah_1&= \lambda_1h_1,\\
    Ah_2&=a_{12}h_1+ \lambda_2h_2,\\
    &\vdots\\
    Ah_n&=a_{1n}h_1+a_{2n}h_2+\dotsb+ \lambda_nh_n,
\end{align*}
\tie{},
\beql{E4.37-0805}
A
=
\begin{pmatrix}
    \lambda_1&a_{12}&a_{13}&\hdots&a_{1n}\\
    0& \lambda_2&a_{23}&\hdots&a_{2n}\\
    0&0& \lambda_3&&a_{3n}\\
    \vdots&\vdots&&\ddots\\
    0&0&0&& \lambda_n
\end{pmatrix}.
\eeq

Let $\cH_k$ be the subspace of $\cH$ with orthonormal basis $\set{h_j}_{j=1}^k$, $k=1,2,\dots,n$.
The subspace $\cH_k$ is invariant with respect to $A$, where
\[
\set{0}
=\cH_0
\subset\cH_1
\subset\cH_2
\subset\dotsb
\subset\cH_n
=\cH.
\]
Let $\cF_k\defeq\cH_k\ominus\cH_{k-1}$ and
\beql{E4.38-0805}
\al_k:=\proj{\al}{\cF_k},\qquad k=1,2,\dots,n.
\eeq
Then, in accordance with \eqref{E2.49-0622} and \eqref{E2.57-0801},
\begin{gather}
    \al=\al_1\al_2\dotsm\al_n,\label{E4.39-0805}\\
    S_\al\rk{z}=S_{\al_1}\rk{z}S_{\al_2}\rk{z}\dotsm S_{\al_n}\rk{z}.\label{E4.40-0805}
\end{gather}
Note that
\beql{E4.41-0805}
\al_k=\ocol{\cF_k}{\C^r}{A_k}{\tau_k}{J},\qquad k=1,2,\dotsc,n,
\eeq
where
\beql{E4.42-0805}
A_k\defeq \lambda_kI_{\cF_k},\qquad \tau_k:=\rstr{\Phi}{\cF_k},\qquad k=1,2,\dotsc,n,
\eeq
and the condition of colligation has the form:
\beql{E4.43-0805}
\tau_k^\ad J\tau_k=2\rk{\im \lambda_k}\Iu{\cF_k},\qquad k=1,2,\dotsc,n.
\eeq
Thus,
\beql{E4.44-0805}
S_{\al_k}\rk{z}\defeq\Iu{r} +\frac{\iu}{z-\lambda_k}\tau_k\tau_k^\ad J,\qquad k=1,2,\dotsc,n.
\eeq
If we denote by $P_k$ the orthoprojector from $\cH$ onto $\cF_k$, $k=1,2,\dots,n$, we obtain from \eqref{E4.42-0805} then
\[
\sum_{k=1}^n\tau_k\tau_k^\ad
=\sum_{k=1}^n\Phi P_k\Phi^\ad
=\Phi\Phi^\ad.
\]
From this and \eqref{E4.43-0805} we get
\[
2\sum_{k=1}^n\abs{\im \lambda_k}
=\sum_{k=1}^n\abs{\tr\tau_k^\ad J\tau_k}
\leq\sum_{k=1}^n\tr\tau_k^\ad\tau_k
=\tr\sum_{k=1}^n\tau_k\tau_k^\ad
=\tr\Phi\Phi^\ad.
\]
Thus, the proof of the following statement is complete.

\bthmnl{\cite{MR0100793}}{T4.23-0805}
The colligation $\al$ $($see \eqref{E4.36-0805}$)$ admits the decomposition \eqref{E4.39-0805}, where the colligations $\al_k$ have form \eqref{E4.41-0805} where the \tcof{}\ $S_\al\rk{z}$ admits a representation in form \eqref{E4.40-0805}, where the factors $S_{\al_k}\rk{z}$ have form \eqref{E4.44-0805}.
Moreover, the relations
\begin{align}
    \tau_k^\ad J\tau_k&=2\rk{\im \lambda_k}\Iu{\cF_k},\qquad k=1,2,\dotsc,n,\label{E4.45-0815}\\
    \sum_{k=1}^n\tau_k\tau_k^\ad&=\Phi\Phi^\ad,\label{E4.45-0805}\\
    2\sum_{k=1}^n\abs{\im \lambda_k}&\leq\tr\Phi\Phi^\ad\label{E4.46-0805}
\end{align}
hold.
\ethm

The decomposition of an open system in a chain of elementary open systems is closely related to the transformation of its basic operator $A$ to triangular form.
Indeed, let $\cO_\al$ and $\cO_{\al_k}$ be the open systems which correspond to the colligation $\al$ (see \eqref{E4.36-0805}) and $\al_k$ (see \eqref{E4.38-0805}), respectively.
The transformation of the operator $A$ to triangular form \eqref{E4.37-0805} generates a decomposition of the colligation $\al$ in the form of the product \eqref{E4.39-0805} which, in view of \rthm{T2.7}, leads to the decomposition
\[
\cO_\al
=\cO_{\al_1}\curlyvee\cO_{\al_2}\curlyvee\dotsb\curlyvee\cO_{\al_n}.
\]
Thus, for $j<k$ the system $\cO_{\al_k}$ influences the system $\cO_{\al_j}$.
From the form of factors \eqref{E4.44-0805} it is seen that this influence depends on
\beql{E4.48-0805}\begin{split}
    \tau_\ell^\ad J\tau_r
    =\rstr{P_\ell\Phi^\ad J\Phi}{\cF_r}
    &=\frac{1}{\iu}\rstr{P_\ell\rk{A-A^\ad}}{\cF_r}\\
    &=\frac{1}{\iu}\rstr{P_\ell A}{\cF_r}
    =\frac{1}{\iu}a_{\ell r},
    \qquad j\leq\ell<r\leq n.
\end{split}\eeq
If the operator $A$ is selfadjoint, then in this case the eigenstates $h_j$, as it is known, are independent.
This is connected with the fact that in this case we have $a_{\ell r}=0$, $\ell<r$ in \eqref{E4.37-0805}.
But then, as it follows from \eqref{E4.48-0805}, the multiplicative decomposition \eqref{E4.40-0805} has additive character
\[
S_\al\rk{z}
=\Iu{r}+\iu\sum_{k=1}^n\frac{\tau_k\tau_k^\ad}{z-\lambda_k}J.
\]
This emphasizes again the additive character of the spectral theory of selfadjoint operators and explains the reason for the multiplicative nature of the spectral theory of non-selfadjoint operators.

If $r=1$ and $J=1$, then it follows from the colligation condition \eqref{E4.43-0805} and formula \eqref{E4.44-0805} that
\[
S_{\al_k}\rk{z}=\frac{\ko{ \lambda_k}-z}{ \lambda_k-z},
\qquad k=1,2,\dotsc,n.
\]
Thus, in this case representation \eqref{E4.40-0805} has the form
\beql{E4.47-0805}
S_\al\rk{z}
=\prod_{k=1}^n\frac{\ko{ \lambda_k}-z}{ \lambda_k-z},
\eeq
where, taking into account that $J=1$, from \eqref{E4.43-0805} and the simplicity of the colligation $\al$ we infer
\[
\im \lambda_k>0,
\qquad k=1,2,\dotsc,n.
\]
From form \eqref{E4.47-0805} of the \tcof{}\ $S_\al\rk{z}$ it follows that in this case the operator $A$ is determined by its spectrum up to unitary equivalence.

\breml{R4.24-0928}
In what follows we will write representation \eqref{E4.44-0805} of the factors $S_{\al_k}\rk{z}$ in a form that does not depend on the space $\cH$.
For this taking into account that $\dim\cF_k=1$ and choosing the basis vector $e_k\in\cF_k$, $k=1,2,3,\dotsc$, we identify the operator $\tau_k\in\ek{\cF_k,\cG}$ with the vector $\eta_k\defeq\tau_k\rk{e_k}$.
Then for $f\in\cF_k$ and $g\in\cG$ we get:
\begin{align*}
    \tau_k\rk{f}
    =\tau_k\rkb{\ipa{f}{e_k}e_k}
    &=\ipa{f}{e_k}\eta_k,\\
    \tau_k\tau_k^\ad\rk{g}
    =\tau_k\rkb{\ipab{\tau_k^\ad\rk{g}}{e_k}e_k}
    =\ipa{g}{\tau_k\rk{e_k}}\tau_k\rk{e_k}
    &=\ipa{g}{\eta_k}\eta_k\\
    &=\eta_k\ipa{g}{\eta_k}
    =\eta_k\eta_k^\ad g,
\end{align*}
where $\eta_k^\ad g\defeq\ipa{g}{\eta_k}$.
In this case, we consider vectors of $\C^r$ as column vectors.
Further, from \eqref{E4.45-0815} it follows
\[
\eta_k^*J\eta_k=\ipa{J\eta_k}{\eta_k}=\ipab{J\tau_k\rk{e_k}}{\tau_k\rk{e_k}}=\ipab{\tau_k^\ad J\tau_k\rk{e_k}}{e_k}=2\im\lambda_k.
\]
Here it is obvious $\norm{\eta_k}=\norm{\tau_k}$.
\erem

Now \rthm{T4.23-0805} can be reformulated as follows.

\bthmnl{\cite{MR0100793}}{T4.25-1001}
The \tcof{}\ $S_\al\rk{z}$ of the colligation \eqref{E4.36-0805} admits the representation
\beql{E4.50-1006}
S_{\al}\rk{z}=\prodr\limits_{k=1}^n\rkb{\Iu{r} +\frac{\iu}{z-\lambda_k}\eta_k\eta_k^\ad J},
\eeq
where $\set{\lambda_k}_{k=1}^n$ are the eigenvalues of the operator $A$, and $\set{\eta_k}_{k=1}^n$ are column vectors from $\C^r$, and the relations
\begin{align}
    \eta_k^\ad J\eta_k&=2\im \lambda_k,\qquad k=1,2,\dotsc,n,\label{E4.50-1001}\\
    \sum_{k=1}^n\eta_k\eta_k^\ad&=\Phi\Phi^\ad,\label{E4.51-1001}\\
    2\sum_{k=1}^n\abs{\im \lambda_k}&\leq\tr\Phi\Phi^\ad\label{E4.52-1001}
\end{align}
hold.
\ethm

\section{Multiplicative integral}\label{sec5-0725}


In this section we follow \cite{P2} and \zitaa{MR0100793}{Appendix}.
With further applications in mind, we assume that all Hilbert spaces considered in this section are finite dimensional.
Note that many of the statements under consideration, as it is easy to see, carry over to the infinite dimensional case.

\subsection{Infinite products}
For products of square matrices we will use the notations
\[
\prodr\limits_{j=1}^nB_j\defeq B_1B_2\dotsm B_n,
\qquad
\prodl\limits_{j=1}^nB_j\defeq B_n\dotsm B_2B_1.
\]

\bdefnl{D5.1-0725}
The infinite product $\prodr\limits_{j=1}^\infty B_j$ is called \emph{convergent} if there exists the limit
\[
\lim_{n\to\infty}C_n\eqdef C
\]
of the partial products
\beql{E5.1-0725}
C_n\defeq\prodr\limits_{j=1}^nB_j,\qquad n=1,2,3,\dotsc,
\eeq
and the matrix $C$ is invertible.
\edefn

Obviously, in this case all matrices $\set{B_j}_{j=1}^\infty$ are invertible and
\[
C^\inv=\lim_{n\to\infty}C_n^\inv,
\]
where
\beql{E5.2-0725}
C_n^\inv=\prodl\limits_{j=1}^nB_j^\inv,\qquad n=1,2,3,\dotsc
\eeq
Hence, it is necessary and sufficient for the convergence of the product $\prodr\limits_{j=1}^\infty B_j$ that the sequences \eqref{E5.1-0725} and 
\eqref{E5.2-0725} converge.
Since $B_j$ is invertible, it admits a representation
\[
B_j=\ec^{H_j},
\]
where it is known that the matrix $H_j$ is not uniquely determined.
However, the specific of multiplicative products enables us to determine in natural way and thus uniquely choosing the exponent.

This enables us to write the products \eqref{E5.1-0725} and \eqref{E5.2-0725} in the equivalent form
\beql{E5.3-0725}
C_n=\prodr\limits_{j=1}^n\ec^{H_j},
\qquad
C_n^\inv=\prodl\limits_{j=1}^n\ec^{-H_j},\qquad n=1,2,\dotsc
\eeq

\bleml{L5.2-0725}
The following inequalities hold
\begin{gather}
    \norm{\ec^{H_1}\ec^{H_2}\dotsm\ec^{H_n}}\leq\ec^{\rho_n}\label{E5.4-0725}\\
    \norm{\ec^{H_1}\ec^{H_2}\dotsm\ec^{H_n}-\Iu{r}}\leq \rho_n\ec^{\rho_n}\label{E5.5-0725}\\
    \normb{\ec^{H_1}\ec^{H_2}\dotsm\ec^{H_n}-\rk{\Iu{r}+H_1+H_2+\dotsb+H_n}}\leq\frac{1}{2}\rho_n^2\ec^{\rho_n},\label{E5.6-0725}
\end{gather}
where $r$ is the order of the matrices $H_1,\dotsc,H_n$ and 
\beql{E5.7-0725}
\rho_n\defeq\norm{H_1}+\norm{H_2}+\dotsb+\norm{H_n}.
\eeq
\elem
\bproof
Let us prove, for example, the inequality \eqref{E5.6-0725}.
It is easily seen that
\begin{multline*}
    \normb{\ec^{H_1}\ec^{H_2}\dotsm\ec^{H_n}-\rk{\Iu{r}+H_1+H_2+\dotsb+H_n}}\\
    \leq\ec^{\norm{H_1}}\ec^{\norm{H_2}}\dotsm\ec^{\norm{H_n}}-\rkb{1+\norm{H_1}+\norm{H_2}+\dotsb+\norm{H_n}}.
\end{multline*}
Now it suffices to note that
\[
\ec^\rho-\rk{1+\rho}
\leq\frac{1}{2}\rho^2\ec^\rho,\qquad\rho\geq0.
\]
Inequalities \eqref{E5.4-0725} and \eqref{E5.5-0725} can be analogously proved.
\eproof

\bthml{T5.3-0725}
The convergence of the series $\sum_{j=1}^\infty\norm{H_j}$ is sufficient for the convergence of the product $\prodr\limits_{j=1}^\infty\ec^{H_j}$.
\ethm
\bproof
For $p<q$ it follows from \eqref{E5.5-0725} that
\[\begin{split}
    \normb{\prodr\limits_{j=1}^p\ec^{H_j}-\prodr\limits_{j=1}^q\ec^{H_j}}
    &=\normb{\prodr\limits_{j=1}^p\ec^{H_j}\rkb{\Iu{r}-\prodr\limits_{j=p+1}^q\ec^{H_j}}}\\
    &\leq\ec^{\rho_p}\rkb{\sum_{j=p+1}^q\norm{H_j}}\ec^{\rho_q-\rho_p},
\end{split}\]
where $\rho_p$ has form \eqref{E5.7-0725}.
Replacing $H_j$ by $-H_j$ we get an analogous estimate.
Hence, in this case sequences \eqref{E5.3-0725} are Cauchy sequences.
\eproof

\bleml{L5.4-0725}
Let
\[
C_n\defeq\prodr\limits_{j=1}^nK_j,
\qquad
D_n\defeq\prodr\limits_{j=1}^nL_j
\]
and let
\[
\mu_n
\defeq\max_j\setb{\norm{L_1}\cdot\norm{L_2}\dotsm\norm{L_{j-1}}\cdot\norm{K_{j+1}}\dotsm\norm{K_n}}.
\]
Then
\beql{E5.8-0725}
\norm{C_n-D_n}
\leq\mu_n\sum_{j=1}^n\norm{K_j-L_j}.
\eeq
\elem
\bproof
The assertion follows from the equality
\begin{multline*}
    C_n-D_n
    =\rk{K_1K_2\dotsm K_{n-1}K_n-L_1K_2\dotsm K_{n-1}K_n}\\
    +\rk{L_1K_2\dotsm K_{n-1}K_n-L_1L_2K_3\dotsm K_{n-1}K_n}\\
    +\dotsb+\rk{L_1L_2\dotsm L_{n-1}K_n-L_1L_2\dotsm L_{n-1}L_n}.\qedhere
\end{multline*}
\eproof

\subsection{Multiplicative Stieltjes integral}\label{subsec5.2-0928}
Let a scalar function $f\rk{t}$ and an operator-valued function $H\rk{t}\in\ek{\cG}$ be given on the interval $[a,b]$.
Denote by $T$ a partition of the interval $[a,b]$ into $n$ parts
\[
a=t_0<t_1<t_2<\dotsb<t_n=b.
\]
Denote by $\xi$ a set of chosen intermediate points $\set{\xi_j}_{j=1}^n$:
$t_{j-1}\leq\xi_j\leq t_j$, $j=1,2,\dots,n$.
Let $\Dl_j\defeq H\rk{t_j}-H\rk{t_{j-1}}$, $j=1,2,\dotsc,n$.
Consider the product
\[
\Pi\rk{T,\xi}
\defeq\ec^{f\rk{\xi_1}\Dl_1}\ec^{f\rk{\xi_2}\Dl_2}\dotsm\ec^{f\rk{\xi_n}\Dl_n}
=\prodr\limits_{j=1}^n\ec^{f\rk{\xi_j}\Dl_j}.
\]
The product $\Pi\rk{T,\xi}$ is called \emph{the integral product of the Stieltjes type for the function $f\rk{t}$ with respect to the weight $H\rk{t}$} formed to the partition $T$ with the set of intermediate points $\xi$.
Now we suppose that under the condition that the maximal length of the partition $T$
\[
\la\rk{T}
\defeq\max_{1\leq j\leq n}\rk{t_j-t_{j-1}}
\]
tends to zero there exists the invertible limit of the product $\Pi\rk{T,\xi}$ regardless of the choice of partitions $T$ and sets of intermediate points $\xi$.
This limit will be called \emph{the multiplicative Stieltjes integral} and denoted by
\beql{E5.9-0805}
\intr\limits_a^b\ec^{f\rk{t}\dif H\rk{t}}
\defeq\lim_{\la\rk{T}\to0}\prodr\limits_{j=1}^n\ec^{f\rk{\xi_j}\Dl_j}.
\eeq

A question naturally arises.
Formulate conditions on $f\rk{t}$  and $H\rk{t}$ which ensure that the multiplicative integral exists.
Regardless of the fact that, in view of the non-commutativity of the matrix multiplication,
\[
\intr\limits_a^b\ec^{f\rk{t}\dif H\rk{t}}
\neq\ec^{\int_a^bf\rk{t}\dif H\rk{t}},
\]
it is interesting to compare the multiplicative Stieltjes integral of the form \eqref{E5.9-0805} with the ordinary additive Stieltjes integral.

For the additive Stieltjes integral we have the following classic existence theorem:
\begin{quote}
    If $f\rk{t}$ is continuous and $H\rk{t}$ is of bounded variation on $[a,b]$, then the Stieltjes integral $\int_a^bf\rk{t}\dif H\rk{t}$ exists.
\end{quote}

An analogous statement for the multiplicative integral \eqref{E5.9-0805} is not true, how it can be seen from the following example.

Let $f\rk{t}=1$ on $[a,b]$, whereas
\[
H\rk{t}
=
\begin{cases}
    H_1,&a\leq t<c,\\
    H_2,&t=c,\\
    H_3,&c<t\leq b.
\end{cases}
\]
In this case,
\[
\Pi\rk{T,\xi}
=
\begin{cases}
    \ec^{H_3-H_1},&c\in T;\\
    \ec^{H_2-H_1}\ec^{H_3-H_2},&c\in T.
\end{cases}
\]
Choosing $H_1$, $H_2$, and $H_3$ such that
\[
\ec^{H_2-H_1}\ec^{H_3-H_2}
\neq\ec^{H_3-H_1},
\]
we see that the limit
\[
\lim_{\la\rk{T}\to0}\Pi\rk{T,\xi}
\]
does not exist.

\bthml{T5.5-0805}
Let a scalar function $f\rk{t}$ be Riemann-integrable at $[a,b]$ and assume that a function $H\rk{t}$ with values in $\ek{\cG}$ satisfies a Lipschitz condition on $[a,b]$.
Then the multiplicative integral \eqref{E5.9-0805} exists.
\ethm
\bproof
In view of the assumptions of the theorem, there exist constants $K>0$ and $L>0$, such that
\begin{enuiap}{0}
    \il{T5.5-0805.1}
    \beql{E5.10-0805}
    \sup_{[a,b]}\absb{f\rk{t}}
    \leq K;
    \eeq
    \il{T5.5-0805.2} for all $t',t''\in[a,b]$
    \beql{E5.11-0805}
    \normb{H\rk{t'}-H\rk{t''}}
    \leq L\abs{t'-t''}
    \eeq
    holds.
\end{enuiap}

In view of the Cauchy criterion, it suffices to prove that the integral products $\Pi\rk{T^{\rk{1}},\xi^{\rk{1}}}$ and $\Pi^\inv\rk{T^{\rk{1}},\xi^{\rk{1}}}$ differ little from $\Pi\rk{T^{\rk{2}},\xi^{\rk{2}}}$ and $\Pi^\inv\rk{T^{\rk{2}},\xi^{\rk{2}}}$, respectively, if $\la\rk{T^{\rk{1}}}$ and $\la\rk{T^{\rk{2}}}$ are sufficiently small.

Consider the partition $T$ obtained by union of the points of partition $T^{\rk{1}}$ and $T^{\rk{2}}$.
Let $\xi$ be a set of intermediate points for the partition $T$.
Then
\begin{multline*}
    \normb{\Pi^{\pm1}\rk{T^{\rk{1}},\xi^{\rk{1}}}-\Pi^{\pm1}\rk{T^{\rk{2}},\xi^{\rk{2}}}}\\
    \leq\normb{\Pi^{\pm1}\rk{T^{\rk{1}},\xi^{\rk{1}}}-\Pi^{\pm1}\rk{T,\xi}}+\normb{\Pi^{\pm1}\rk{T,\xi}-\Pi^{\pm1}\rk{T^{\rk{2}},\xi^{\rk{2}}}}.
\end{multline*}
From this and the symmetric form of \eqref{E5.3-0725} it follows that it suffices to convince of the smallness of
\[
\normb{\Pi\rk{T^{\rk{1}},\xi^{\rk{1}}}-\Pi\rk{T,\xi}}.
\]

Let
\begin{gather*}
    T^{\rk{1}}\colon a\eqdef t_0<t_1<\dotsb<t_n\defeq b,
    \qquad\xi^{\rk{1}}\defeq\set{\xi_j}_{j=1}^n;\\
    T\colon t_{j-1}\eqdef t^{\rk{j}}_0<t^{\rk{j}}_1<\dotsb<t^{\rk{j}}_{k_j}\defeq t_j,
    \qquad j=1,2,\dotsc,n;\\
    \xi\defeq\set{\xi^{\rk{j}}_{p_j}}_{j=1,p_j=1}^{n,k_j},
    \qquad t^{\rk{j}}_{p_j-1}\leq\xi^j_{p_j}\leq t^{\rk{j}}_{p_j};\\
    \Dl^{\rk{j}}_{p_j}\defeq H\rk{t^{\rk{j}}_{p_j}}-H\rk{t^{\rk{j}}_{p_j-1}}.
\end{gather*}
From inequality \eqref{E5.8-0725} we get
\beql{E5.12-0805}
\normb{\Pi\rk{T^{\rk{1}},\xi^{\rk{1}}}-\Pi\rk{T,\xi}}
\leq\mu_n\sum_{j=1}^n\normb{\ec^{f\rk{\xi_j}\Dl_j}-\prodr\limits_{p_j=1}^{k_j}\ec^{f\rk{\xi^{\rk{j}}_{p_j}}\Dl^{\rk{j}}_{p_j}}},
\eeq
where
\begin{multline}\label{E5.13-0805}
    \mu_n
    \defeq\max_j\Biggl\{\norm{\ec^{f\rk{\xi_1}\Dl_1}}\norm{\ec^{f\rk{\xi_2}\Dl_2}}\dotsm\norm{\ec^{f\rk{\xi_{j-1}}\Dl_{j-1}}}\\
    \normb{\prodr\limits_{p_{j+1}=1}^{k_{j+1}}\ec^{f\rk{\xi^{\rk{j+1}}_{p_{j+1}}}\Dl^{\rk{j+1}}_{p_{j+1}}}}\dotsm\normb{\prodr\limits_{p_n=1}^{k_n}\ec^{f\rk{\xi^{\rk{n}}_{p_n}}\Dl^{\rk{n}}_{p_n}}}\Biggr\}.
\end{multline}
From \eqref{E5.4-0725}, \eqref{E5.10-0805}, and \eqref{E5.11-0805} it follows that
\begin{align*}
    \norm{\ec^{f\rk{\xi_j}\Dl_j}}
    &\leq\ec^{\abs{f\rk{\xi_1}}\norm{\Dl_j}}
    \leq\ec^{KL\rk{t_j-t_{j-1}}},\\
    \normb{\prodr\limits_{p_{j}=1}^{k_{j}}\ec^{f\rk{\xi^{\rk{j}}_{p_{j}}}\Dl^{\rk{j}}_{p_{j}}}}
    &\leq\ec^{\sum_{p_j=1}^{k_j}\abs{f\rk{\xi^{\rk{j}}_{p_j}}}\norm{\Dl_{p_j}^{\rk{j}}}}\\
    &\leq\ec^{KL\sum_{p_j=1}^{k_j}\rk{t^{\rk{j}}_{p_j}-t^{\rk{j}}_{p_{j-1}}}}
    =\ec^{KL\rk{t_j-t_{j-1}}}.
\end{align*}
Using these estimates, we obtain from \eqref{E5.13-0805} that
\beql{E5.14-0805}
\mu_n\leq M,
\qquad M\defeq\ec^{KL\rk{b-a}}.
\eeq

For estimating the $j$\nobreakdash-th term of the sum on the right side of inequality \eqref{E5.12-0805}, we consider the representation
\begin{gather}
    \ec^{f\rk{\xi_j}\Dl_j}
    =\Iu{r}+f\rk{\xi_j}\Dl_j+R_j,\label{E5.15-0805}\\
    \prodr\limits_{p_j=1}^{k_j}\ec^{f\rk{\xi^{\rk{j}}_{p_j}}\Dl^{\rk{j}}_{p_j}}
    =\Iu{r}+\sum_{p_j=1}^{k_j}f\rk{\xi^{\rk{j}}_{p_j}}\Dl^{\rk{j}}_{p_j}+R^{\rk{j}}_{p_j}\label{E5.16-0805}
\end{gather}
where, in view of \eqref{E5.6-0725}, \eqref{E5.10-0805}, and \eqref{E5.11-0805}, we have
\[
\norm{R^{\rk{j}}_{p_j}}\leq\frac{1}{2}\rho_{k_j}^2\ec^{\rho_{k_j}},
\qquad j=1,2,\dotsc,n,
\]
where
\[
\rho_{k_j}
\defeq\sum_{p_j=1}^{k_j}\absb{f\rk{\xi^{\rk{j}}_{p_j}}}\norm{\Dl^{\rk{j}}_{p_j}}
\leq KL\sum_{p_j=1}^{k_j}\rk{t^{\rk{j}}_{p_j}-t^{\rk{j}}_{p_j-1}}
=KL\rk{t_j-t_{j-1}}.
\]
Thus,
\beql{E5.17-0805}\begin{split}
    \norm{R^{\rk{j}}_{p_j}}
    &\leq\frac{1}{2}K^2L^2\rk{t_j-t_{j-1}}^2\ec^{KL\rk{t_j-t_{j-1}}}\\
    &\leq\frac{1}{2}K^2L^2\rk{t_j-t_{j-1}}\ec^{KL\rk{b-a}}\la\rk{T^{\rk{1}}}\\
    &=\frac{1}{2}K^2L^2M\rk{t_j-t_{j-1}}\la\rk{T^{\rk{1}}},
    \qquad j=1,2,\dotsc,n.
\end{split}\eeq
Analogously, we obtain
\beql{E5.18-0805}
\norm{R_j}
\leq\frac{1}{2}K^2L^2M\rk{t_j-t_{j-1}}\la\rk{T^{\rk{1}}},
\qquad j=1,2,\dotsc,n.
\eeq
Inserting the expressions \eqref{E5.15-0805} and \eqref{E5.16-0805} in the $j$\nobreakdash-th term of the sum in \eqref{E5.12-0805} and taking into account \eqref{E5.14-0805}, \eqref{E5.17-0805}, and \eqref{E5.18-0805}, we obtain
\[\begin{split}
    &\normb{\Pi\rk{T^{\rk{1}},\xi^{\rk{1}}}-\Pi\rk{T,\xi}}\\
    &\leq M\sum_{j=1}^n\rkb{\normb{f\rk{\xi_j}\Dl_j-\sum_{p_j=1}^{k_j}f\rk{\xi^{\rk{j}}_{p_j}}\Dl^{\rk{j}}_{p_j}}+\norm{R_j}+\norm{R_{j,p_j}}}\\
    &\leq M\sum_{j=1}^n\rkb{\normb{\sum_{p_j=1}^{k_j}\rkb{f\rk{\xi_j}-f\rk{\xi^{\rk{j}}_{p_j}}}\Dl^{\rk{j}}_{p_j}}+K^2L^2M\rk{t_j-t_{j-1}}\la\rk{T^{\rk{1}}}}\\
    &\leq ML\sum_{j=1}^n\sum_{p_j=1}^{k_j}\absb{f\rk{\xi_j}-f\rk{\xi^{\rk{j}}_{p_j}}}\rk{t^{\rk{j}}_{p_j}-t^{\rk{j}}_{p_j-1}}+K^2L^2M^2\rk{b-a}\la\rk{T^{\rk{1}}}\\
    &\leq ML\sum_{j=1}^n\omega_j\rk{t_j-t_{j-1}}+K^2L^2M^2\rk{b-a}\la\rk{T^{\rk{1}}},
\end{split}\]
where $\omega_j$ is the oscillation of the function $f$ on the interval $[t_{j-1},t_j]$.
It remains to apply the criterion for the Riemann integrability of a function in terms of the oscillation.
\eproof

Similarly (however, much simpler), it can be proved under these conditions of the existence of the additive integral
\beql{E5.19-0805}
\int_a^bf\rk{t}\dif H\rk{t}.
\eeq

The Lipschitz condition for $H\rk{t}$ in \rthm{T5.5-0805} is too arduous.
It can be considerably weakened by requiring some condition for $f\rk{t}$.
Let, for example, the function $\vphi\rk{x}\in C\albe$ be strictly increasing and mapping the interval $\albe$ on the interval $\ab$ then the multiplicative integrals
\[
\intr\limits_a^b\ec^{f\rk{t}\dif H\rk{t}},
\qquad\intr\limits_\al^\be\ec^{f_1\rk{x}\dif H_1\rk{x}},
\]
where $f_1\rk{x}=f\rk{\vphi\rk{x}}$, $H_1\rk{x}=H\rk{\vphi\rk{x}}$, $x\in\albe$, exist simultaneously and are equal.
This follows from the fact that the limit procedures of
\[
\max\rk{t_j-t_{j-1}}\to0
\qquad\text{and}\qquad
\max\rk{x_j-x_{j-1}}\to0
\]
are equivalent.
It is clear that
\[
\bigvee_a^bH
=\bigvee_\al^\be H_1,
\]
and that a countable set by the mapping $t=\vphi\rk{x}$ is transformed again into a countable set and nothing can be said about sets of zero measure.

We show now that outgoing from a continuous operator function $H\rk{t}$ of bounded variation one can always pass over to the function $H_1\rk{x}=H\rk{\vphi\rk{x}}$ satisfying a Lipschitz condition.

We consider the function $v\colon\ab\to\R$ given by
\beql{E5.20-0829}
v\rk{t}=\bigvee_a^tH,\qquad a\leq t\leq b.
\eeq
Obviously,
\[
v\rk{t''}-v\rk{t'}=\bigvee_{t'}^{t''}H,\qquad t',t''\in\ab,\;t'<t'',
\]
where $v\rk{t}$ is continuous simultaneously with $H\rk{t}$.
We introduce a new variable by setting
\[
x=v\rk{t}+t,\qquad t\in\ab.
\]
The function $x=x\rk{t}$ is continuous, strictly increasing and maps the interval $\ab$ on the interval $\albe$, where
\[
\al=a,\qquad\be=v\rk{b}+b=\bigvee_a^bH+b.
\]

Let $t=\vphi\rk{x}$, $x\in\albe$, be the function inverse with respect to the function $x=x\rk{t}$, $t\in\ab$.
The operator function
\[
H_1\rk{x}=H\rk{\vphi\rk{x}},\qquad x\in\albe,
\]
satisfies a Lipschitz condition.
Indeed, if $x',x''\in\albe$, $x'<x''$, and $t'=\vphi\rk{x'}$, $t''=\vphi\rk{x''}$, then
\[\begin{split}
    \normb{H_1\rk{x''}-H_1\rk{x'}}
    =\normb{H\rk{t''}-H\rk{t'}}
    &\leq v\rk{t''}-v\rk{t'}\\
    &\leq x\rk{t''}-x\rk{t'}
    =x''-x'.
\end{split}\]
Thus, we have the following statement.

\bthml{T5.6-0829}
Let $f\rk{t}$ be a bounded function on the interval $\ab$ having at most countable jumps and let $H\rk{t}$ be a continuous operator function on $\ab$ of bounded variation.
Then there exists the multiplicative integral \eqref{E5.9-0805}.
\ethm

Obviously under this conditions exists also the additive integral \eqref{E5.19-0805}.

The limit in the inequalities of \rlem{L5.2-0725} lead to the following estimates of the multiplicative integral.

\bleml{L5.7-0829}
It hold the following inequalities
\begin{align}
    \normb{\intr\limits_a^b\ec^{f\rk{t}\dif H\rk{t}}}&\leq\ec^\rho,\label{E5.21-0829}\\
    \normb{\intr\limits_a^b\ec^{f\rk{t}\dif H\rk{t}}-\Iu{r}}&\leq\rho\ec^\rho,\label{E5.22-0829}\\
    \normb{\intr\limits_a^b\ec^{f\rk{t}\dif H\rk{t}}-\rkb{\Iu{r}+\int_a^bf\rk{t}\dif H\rk{t}}}&\leq\frac{1}{2}\rho^2\ec^\rho,\label{E5.23-0829}
\end{align}
where $\rho\defeq\int_a^b\abs{f\rk{t}}\dif v\rk{t}$ and $v\rk{t}$ has the form \eqref{E5.20-0829}.
\elem

Now we mention the following easily checked properties of the multiplicative integral:
\begin{enuiard}
    \il{1.-0829} If $a<c<b$ then
    \beql{E2.24-0829}
    \intr\limits_a^b\ec^{f\rk{t}\dif H\rk{t}}
    =\intr\limits_a^c\ec^{f\rk{t}\dif H\rk{t}}\intr\limits_c^b\ec^{f\rk{t}\dif H\rk{t}}.
    \eeq
    \il{2.-0829} 
    \beql{E2.25-0829}
    \rkb{\intr\limits_a^b\ec^{f\rk{t}\dif H\rk{t}}}^\inv
    =\intl\limits_a^b\ec^{-f\rk{t}\dif H\rk{t}}.
    \eeq
\end{enuiard}

We mention a special but important case of \rthm{T5.6-0829}.

\bthml{T5.8-0829}
If $f\rk{t}=1$, $t\in\ab$, and $H\rk{t}$ is absolutely continuous on $\ab$, then the multiplicative integral \eqref{E5.9-0805} exists.
\ethm

For multiplicative integrals an analogous statement to the Helly theorem is true.

\bthml{T5.9-0829}
Let $\set{f_n\rk{t}}_{n=1}^\infty$ be a sequence of scalar functions on $\ab$ and $\set{H_n\rk{t}}_{n=1}^\infty$ be a sequence of operator functions on $\ab$, $H_n\rk{t}\in\ek{\cG}$, $a\leq t\leq b$, which satisfy the following conditions:
\begin{enuiap}{0}
    \il{1p-0829} The functions $f_n\rk{t}$, $n=1,2,\dotsc$, are Riemann integrable on $\ab$ and uniformly bounded on $\ab$:
    \beql{E5.26-0829}
    \sup_{\ab}\absb{f_n\rk{t}}\leq K,\qquad n=1,2,3,\dotsc
    \eeq
    \il{2p-0829} $\lim_{n\to\infty}f_n\rk{t}=f\rk{t}$ for each $t\in\ab$, where the function $f\rk{t}$ is Riemann integrable on $\ab$.
    \il{3p-0829} The functions $H_n\rk{t}$, $n=1,2,\dotsc$, satisfy a Lipschitz condition
    \beql{E5.27-0829}
    \normb{H_n\rk{t'}-H_n\rk{t''}}\leq L\abs{t'-t''},\qquad t',t''\in\ab,\;n=1,2,3,\dotsc,
    \eeq
    where the sequence $\set{H_n\rk{t}}_{n=1}^\infty$ converges pointwise on $\ab$ to the function $H\rk{t}$
\end{enuiap}
Then
\beql{E5.28-0829}
\lim_{n\to\infty}\intr\limits_a^b\ec^{f_n\rk{t}\dif H_n\rk{t}}
=\intr\limits_a^b\ec^{f\rk{t}\dif H\rk{t}}.
\eeq
\ethm
\bproof
First note that, in view of \rthm{T5.5-0805}, all integrals in \eqref{E5.28-0829} exist.

We partition the interval $\ab$ in equidistant parts by the points
\[
a
=x_0
<x_1
<x_2
<\dotsb
<x_n
=b.
\]
Let
\beql{E5.29-0829}
    K_j^{\rk{p}}
    =\intr\limits_{x_{j-1}}^{x_j}\ec^{f_p\rk{t}\dif H_p\rk{t}},\;
    L_j
    =\intr\limits_{x_{j-1}}^{x_j}\ec^{f\rk{t}\dif H\rk{t}},\qquad
    j=1,2,\dotsc,n;\;p=1,2,\dotsc
\eeq
Then taking into account \rlem{L5.4-0725}, we obtain
\beql{E5.30-0829}
\begin{split}
    \normb{\intr\limits_a^b\ec^{f_p\rk{t}\dif H_p\rk{t}}-\intr\limits_a^b\ec^{f\rk{t}\dif H\rk{t}}}
    &=\norm{K_1^{\rk{p}}K_2^{\rk{p}}\dotsm K_n^{\rk{p}}-L_1L_2\dotsm L_n}\\
    &\leq\mu_n\sum_{j=1}^n\norm{K_j^{\rk{p}}-L_j},
\end{split}
\eeq
where, in view of \eqref{E5.26-0829}, \eqref{E5.27-0829}, and \eqref{E5.29-0829}, we have
\beql{E5.31-0829}
\mu\leq M,\qquad M\defeq\ec^{KL\rk{b-a}}.
\eeq
We note that
\begin{multline}\label{E5.32-0829}
    \norm{K_j^{\rk{p}}-L_j}
    \leq\normb{K_j^{\rk{p}}-\rkb{\Iu{r}+\int_{x_{j-1}}^{x_j}f_p\rk{t}\dif H_p\rk{t}}}\\
    +\normb{\int_{x_{j-1}}^{x_j}f_p\rk{t}\dif H_p\rk{t}-\int_{x_{j-1}}^{x_j}f\rk{t}\dif H\rk{t}}\\
    +\normb{L_j-\rkb{\Iu{r}+\int_{x_{j-1}}^{x_j}f\rk{t}\dif H\rk{t}}}
\end{multline}
Taking into account \eqref{E5.26-0829} and \eqref{E5.27-0829}, we infer from the estimate \eqref{E5.23-0829} then
\beql{E5.33-0829}\begin{split}
    \normb{K_j^{\rk{p}}-\rkb{\Iu{r}+\int_{x_{j-1}}^{x_j}f_p\rk{t}\dif H_p\rk{t}}}
    &\leq\frac{1}{2}K^2L^2\rk{x_j-x_{j-1}}^2\ec^{KL\rk{x_j-x_{j-1}}}\\
    &\leq\frac{1}{2n^2}K^2L^2M.
\end{split}\eeq
We obtain an analogous estimate for the third term in \eqref{E5.32-0829}.
Returning to \eqref{E5.30-0829} and taking into account \eqref{E5.31-0829}--\eqref{E5.33-0829}, we get
\begin{multline*}
    \normb{\intr\limits_a^b\ec^{f_p\rk{t}\dif H_p\rk{t}}-\intr\limits_a^b\ec^{f\rk{t}\dif H\rk{t}}}
    \leq\mu_n\sum_{j=1}^n\norm{K_j^{\rk{p}}-L_j}\\
    \leq\frac{1}{n}K^2L^2M^2+M\sum_{j=1}^n\normb{\int_{x_{j-1}}^{x_j}f_p\rk{t}\dif H_p\rk{t}-\int_{x_{j-1}}^{x_j}f\rk{t}\dif H\rk{t}}.
\end{multline*}
Now fixing $n$ and increasing $p$ the second summand can be made arbitrarily small in view of the Helly theorem for the additive integral.
\eproof

\subsection{Multiplicative Lebesgue integral}\label{subsubsec5.1.3-0912}

Let $M\rk{t}$, $t\in\ab$, be a Lebesgue integrable operator function.
By $H\rk{t}$ we denote the Lebesgue integral with variable upper border
\[
H\rk{t}=\int_0^tM\rk{\tau}\dif\tau,\qquad t\in\ab.
\]
As it is known $H\rk{t}$ is absolutely continuous, $\frac{\dif H}{\dif t}\rk{t}=M\rk{t}$ \tae{}\ on $\ab$ and $\bigvee_a^bH=\int_a^b\norm{M\rk{x}}\dif x$.

But then from \rthm{T5.8-0829} it follows the existence of the multiplicative Stieltjes integral
\[
\intr\limits_a^b\ec^{\dif H\rk{t}}
=\intr\limits_a^b\ec^{\dif\{\int_a^tM\rk{x}\dif x\}}.
\]
This integral we denote by
\[
\intr\limits_a^b\ec^{\dif M\rk{t}}
\defeq\intr\limits_a^b\ec^{\dif H\rk{t}}
\]
and call it multiplicative Lebesgue integral of the operator function $M\rk{t}$.

The multiplicative Lebesgue integral with variable border
\beql{E5.34-0912}
W\rk{t}=\intr\limits_a^t\ec^{M\rk{x}\dif x},\qquad t\in\ab.
\eeq
plays an important role.

\bthml{T5.10-0912}
Let $M\rk{t}$, $t\in\ab$, be a Lebesgue integrable operator-valued function.
Then the operator-valued function $W\rk{t}$ is absolutely continuous on $\ab$ when \tae{}\ on $\ab$ it holds
\beql{E5.35-0912}
\frac{\dif W}{\dif t}\rk{t}
=W\rk{t}M\rk{t}.
\eeq
\ethm
\bproof
For $t',t''\in\ab$, $t'<t''$, we have
\beql{E5.36-0912}
W\rk{t''}-W\rk{t'}
=\intr\limits_a^{t'}\ec^{M\rk{x}\dif x}\rkb{\intr\limits_{t'}^{t''}\ec^{M\rk{x}\dif x}-\Iu{r}}.
\eeq
Taking into account that in this case ($f\rk{t}\equiv1$, $H\rk{t}=\int_0^tM\rk{x}\dif x$) the inequalities \eqref{E5.21-0829}--\eqref{E5.23-0829} hold with constant
\beql{E5.37-0912}
\rho
=\int_a^b\normb{M\rk{x}}\dif x
\eeq
then from \eqref{E5.28-0829} it follows
\[\begin{split}
    \normb{W\rk{t''}-W\rk{t'}}
    &\leq\ec^{\int_a^{t'}\norm{M\rk{x}}\dif x}\int_{t'}^{t''}\normb{M\rk{x}}\dif x\ec^{\int_{t'}^{t''}\norm{M\rk{x}}\dif x}\\
    &\leq C\rkb{v\rk{t''}-v\rk{t'}},
\end{split}\]
where $C\defeq\ec^{\int_{a}^{b}\norm{M\rk{x}}\dif x}$, $v\rk{t}=\int_{a}^{t}\norm{M\rk{x}}\dif x$.
Now the absolutely continuity of $W\rk{t}$ follows from the absolutely continuity of $v\rk{t}$.

Fixing $t$ we compute the right derivative $W_+'\rk{t}$, assuming that $\Dl t>0$.
Obviously
\beql{E5.38-0912}\begin{split}
    W\rk{t+\Dl t}-W\rk{t}
    &=W\rk{t}\rkb{\intr\limits_t^{t+\Dl t}\ec^{M\rk{x}\dif x}-\Iu{r}}\\
    &=W\rk{t}\rkb{\int_t^{t+\Dl t}M\rk{x}\dif x+R_{\Dl t}},
\end{split}\eeq
where
\[
R_{\Dl t}
=\intr\limits_t^{t+\Dl t}\ec^{M\rk{x}\dif x}-\rkb{\Iu{r}+\int_t^{t+\Dl t}M\rk{x}\dif x}.
\]
From the inequality \eqref{E5.23-0829} and taking into account \eqref{E5.37-0912} with the constant $\rho$, we get
\[
\frac{1}{\Dl t}\norm{R_{\Dl t}}
\leq\frac{1}{2}\ekb{\frac{1}{\Dl t}\int_t^{t+\Dl t}\normb{M\rk{x}}\dif x}^2\ec^{\int_t^{t+\Dl t}\normb{M\rk{x}}\dif x}\Dl t.
\]
From this it follows
\[
\lim_{\Dl t\to0}\frac{1}{\Dl t}R_{\Dl t}
=0.
\]
Returning to \eqref{E5.38-0912} we obtain now that \tae{}\ on $\ab$ it holds
\[\begin{split}
    W_+'\rk{t}
    &=\lim_{\Dl t\to+0}\frac{W\rk{t+\Dl t}-W\rk{t}}{\Dl t}\\
    &=\lim_{\Dl t\to+0}W\rk{t}\rkb{\frac{1}{\Dl t}\int_t^{t+\Dl t}M\rk{x}\dif x+\frac{1}{\Dl t}R_{\Dl t}}
    =W\rk{t}M\rk{t}.
\end{split}\]
Since $W\rk{t}$ is absolutely continuous then $W'\rk{t}=W_+'\rk{t}$ \tae{}\ on $\ab$.
This means the inequality \eqref{E5.27-0829} is satisfied \tae{}\ on $\ab$.
\eproof

Hence, the multiplicative integral
\beql{E5.39-0912}
W\rk{t}=\intr\limits_a^t\ec^{M\rk{x}\dif x},\qquad t\in\ab,
\eeq
with Lebesgue integrable operator-valued function $M\rk{t}$ on the interval $\ab$ is a solution of the Cauchy problem
\beql{E5.40-0912}
\begin{cases}
    \frac{\dif W}{\dif t}\rk{t}=W\rk{t}M\rk{t},&t\in\ab\\
    W\rk{a}=\Iu{r}
\end{cases}.
\eeq
As it is known (see Naimark, Linear differential operators, \cch{I}.) this Cauchy problem has a unique solution in the class of absolutely continuous functions.

By integration we get that the multiplicative integral \eqref{E5.39-0912} is the unique solution of the integral equation
\beql{E5.41-0912}
W\rk{t}=\Iu{r}+\int_a^tW\rk{x}M\rk{x}\dif x,\qquad t\in\ab.
\eeq

\subsection{Example}\label{subsec5.4-1018}
In the space $L_2\ek{0,\ell}$ we consider the operator (see \zitaa{MR0322542}{\cch{I}})
\beql{E5.42-0912}
Af\rk{x}
=\iu\int_x^\ell K\rk{x,t}f\rk{t}\dif t,
\eeq
where
\begin{multline}\label{E5.43-0912}
    K\rk{x,t}=\ko{\xi_1\rk{x}}\xi_1\rk{t}+\dotsb+\ko{\xi_p\rk{x}}\xi_p\rk{t}\\
    -\ko{\xi_{p+1}\rk{x}}\xi_{p+1}\rk{t}-\dotsb-\ko{\xi_{p+q}\rk{x}}\xi_{p+q}\rk{t},\\\xi_j\rk{x}\in L_2\ek{0,\ell},\;j=1,2,\dotsc,p+q.
\end{multline}
We set
\[
\xi\rk{x}=\Mat{\xi_1\rk{x}\\\xi_2\rk{x}\\\vdots\\\xi_r\rk{x}},
\qquad J=\Mat{\Iu{p}&0\\0&-\Iu{q}},
\qquad r=p+q.
\]
Thus, \eqref{E5.43-0912} and \eqref{E5.41-0912} can be rewritten in the form
\begin{gather}
    K\rk{x,t}=\xi^\ad\rk{x}J\xi\rk{t},\notag\\
    Af\rk{x}=\iu\int_x^\ell\xi^\ad\rk{x}J\xi\rk{t}f\rk{t}\dif t.\label{E5.44-0912}
\end{gather}
Since
\beql{E5.45-0912}
A^\ad f\rk{x}
=-\iu\int_0^x\xi^\ad\rk{x}J\xi\rk{t}f\rk{t}\dif t,
\eeq
then
\beql{E5.46-0912}
\frac{A-A^\ad}{\iu}f\rk{x}
=\int_0^\ell\xi^\ad\rk{x}J\xi\rk{t}f\rk{t}\dif t.
\eeq

Let $\cH\defeq L_2\ek{0,\ell}$ and $\cG\defeq\C^r$.
We define the operator $\Phi\in\ek{\cH,\C^r}$ by setting
\beql{E5.47-0912}
\Phi f
\defeq\int_0^\ell\xi\rk{t}f\rk{t}\dif t.
\eeq
Then
\beql{E5.48-0912}
\Phi^\ad g\rk{x}=\xi^\ad\rk{x}g,\qquad g\in\C^r,
\eeq
and \eqref{E5.46-0912} can be rewritten in the form
\[
\frac{A-A^\ad}{\iu}
=\Phi^\ad J\Phi.
\]

Thus the quintuple
\beql{E5.49-0912}
\al
\defeq\ocol{\cH}{\C^r}{A}{\Phi}{J}
\eeq
is an operator colligation.

Let $S\rk{x,z}$, $x\in\ek{0,\ell}$, $z\neq0$, be a matrix function which is a solution of the Cauchy problem
\beql{E5.50-0912}
\begin{cases}
    \frac{\dif S\rk{x,z}}{\dif x}=\frac{\iu}{z}S\rk{x,z}\xi\rk{x}\xi^\ad\rk{x}J,&x\in\ek{0,\ell},\\
    S\rk{0,z}\equiv\Iu{r}.
\end{cases}
\eeq
Then
\beql{E5.51-0912}
S\rk{x,z}
=\Iu{r}+\frac{\iu}{z}\int_0^xS\rk{t,z}\xi\rk{t}\xi^\ad\rk{t}\dif tJ
\eeq
and (see \rsubsec{subsubsec5.1.3-0912})
\beql{E5.52-0912}
S\rk{x,z}
=\intr\limits_0^x\ec^{\frac{\iu}{z}\xi\rk{t}\xi^\ad\rk{t}\dif tJ}.
\eeq
We set
\beql{E5.53-0912}
\ze\rk{x,z}
\defeq-\frac{1}{z}S\rk{x,z}\xi\rk{x}.
\eeq
For $g\in\cG$ from \eqref{E5.45-0912}, \eqref{E5.48-0912}, \eqref{E5.51-0912}, and \eqref{E5.53-0912} we find
\[\begin{split}
    \rk{A^\ad-\ko z\IH}\ze^\ad\rk{x,z}g
    &=-\iu\xi^\ad\rk{x}J\int_0^x\xi\rk{t}\zeta^\ad\rk{t,z}\dif tg-\ko z\zeta^\ad\rk{x,z}g\\
    &=\xi^\ad\rk{x}\rkb{\frac{\iu}{z}J\int_0^x\xi\rk{t}\xi^\ad\rk{t}S^\ad\rk{t,z}\dif t+S^\ad\rk{x,z}}g\\
    &=\xi^\ad\rk{x}g
    =\Phi^\ad g\rk{x}.
\end{split}\]

From this by taking into account \eqref{E5.48-0912}, \eqref{E5.53-0912}, and \eqref{E5.51-0912} we get
\[\begin{split}
    \ipab{\Phi\rk{A-z\IH}^\inv\Phi^\ad Jg_1}{g_2}
    &=\ipab{\xi^\ad\rk{x}Jg_1}{\rk{A^\ad-\ko{z}\IH}^\inv\xi^\ad\rk{x}g_2}\\
    &=\ipab{\xi^\ad\rk{x}Jg_1}{\zeta^\ad\rk{x,z}g_2}\\
    &=\int_0^\ell\ipab{\zeta\rk{x,z}\xi^\ad\rk{x}Jg_1}{g_2}\dif x\\
    &=\ipab{-\frac{1}{z}\int_0^\ell S\rk{x,z}\xi\rk{x}\xi^\ad\rk{x}\dif xJg_1}{g_2}\\
    &=\ipab{\iu\rkb{S\rk{\ell,z}-\Iu{m}}g_1}{g_2}_\cG.
\end{split}\]
This means
\[
S_\al\rk{z}
=S\rk{\ell,z}
\]
and, in view of \eqref{E5.48-0912}, then
\beql{E5.54-0912}
S_\al\rk{z}
=\intr\limits_0^\ell\ec^{\frac{\iu}{z}\xi\rk{t}\xi^\ad\rk{t}\dif tJ}.
\eeq

If $p=1$, $q=0$, $\xi\rk{x}\equiv1$ on $\ek{0,\ell}$, then the operator $A$ is the Leibniz--Newton integral operator (see \rsubsec{subsec3.5-0912}) and in this case \eqref{E5.54-0912}
\[
S_\al\rk{z}
=\ec^{\frac{\iu}{z}\ell}.
\]

In the case $q=0$ (\tie{}\ $J=\Iu{r}$) as it was shown in \zitaa{MR0322542}{\cch{I}} the colligation \eqref{E5.49-0912} is simple if and only if the vector $\xi\rk{x}$ is \tae{}\ different from zero.

\section{Multiplicative representation of characteristic function}\label{sec6-0917}

\subsection{Potapov's Factorization Theorem for characteristic functions of class $\Om_J\rk{\C^r} $}
In his known paper \cite{MR0076882} V.~P.~Potapov investigated the multiplicative structure of matrix functions which are meromorphic and \tJ{contractive} in the open unit disk.
The following special case of characteristic functions of class $\Om_J\rk{\C^r} $ is important for us.
In this case V.~P.~Potapov's Theorem has the following form (see \cite{MR0076882,MR0100793}).

\bthml{T6.1-0917}
Each function $S\rk{z}$ of the class $\Om_J\rk{\C^r} $ admits a representation in the form
\beql{E6.1-0917}
S\rk{z}
=\prodr\limits_{j=1}^\ome\rkb{\Iu{r}+\iu\frac{\eta_k\eta_k^\ad}{z-\lambda_k}J}\intr\limits_0^\ell\ec^{\iu\frac{\dif E(t)}{z-a\rk{t}}J},
\eeq
where $\ome\leq\infty$, $\set{\eta_k}_{k=1}^\ome$ are column vectors from $\C^r$, $a(t)$ and $E(t)$ are non-decreasing on $[0,\ell]$ scalar and matrix functions, respectively, where
\[
\eta_k^\ad J\eta_k=2\im \lambda_k,\; k=1,2,3,\dotsc;\;\sum_{k=1}^\ome\eta_k^\ad\eta_k<+\infty,\; \tr E(t)=t,\, t\in[0,\ell].
\]
Conversely, every function of the form \eqref{E6.1-0917} belongs to $\Om_J\rk{\C^r} $.
\ethm

In \cite{MR0100793} the multiplicative representation \eqref{E6.1-0917} was obtained with methods of operator theory.
Below in \rsubsecsd{subsec6.2-0917}{subsec6.3-0917} (see \rthm{T6.6A-0925}) we indicate basic features of this approach.

\subsection{Approximation theorems}\label{subsec6.2-0917}
Let $S\rk{z}$ belong to the class $\Om_J\rk{\C^r} $ and let
\[
\al
\defeq\ocol{\cH}{\C^r}{A}{\Phi}{J}
\]
be a simple operator colligation such that
\[
S_\al\rk{z}
=S\rk{z}.
\]

\bthmnl{\cite{MR0100793}}{T6.2-0917}
Let
\[
\set{0}
=\cH_0
\subset\cH_1
\subset\cH_2
\subset\dotsb
\subset\cH_n
\subset\dotsb
\]
be an increasing sequence of subspaces of the space $\cH$, let $P_n$ be the orthoprojector from $\cH$ onto $\cH_n$, $n=1,2,3,\dotsc$ and let 
\[
S_n\rk{z}
\defeq\proj{S_\al\rk{z}}{\cH_n}.
\]
If the sequence $\set{P_n}_{n=1}^\infty$ strongly converges to $\IH$, then the sequence $S_n\rk{z}$ converges uniformly to $S_\al\rk{z}$ in some neighborhood of every non-real regular point of $S\rk{z}$.
\ethm

\bcornl{\cite{MR0100793}}{C6.3-0917}
Let the assumptions of the preceding theorem be satisfied and let the subspaces $\cH_n$, $n=1,2,3,\dotsc$ be invariant with respect to the operator $A$.
Then
\beql{E6.2-0917}
S\rk{z}
=\prodr\limits_{k=1}^\infty\proj{S_\al\rk{z}}{\cH_k\ominus\cH_{k-1}},
\eeq
where the infinite product \eqref{E6.2-0917} converges uniformly to $S\rk{z}$ in some neighborhood of any non-real regular point of $S\rk{z}$.
\ecor
\bproof
In the considered case, in view of \eqref{E2.56-0815}--\eqref{E2.57-0801}, we have
\[
\proj{S_\al\rk{z}}{\cH_n}
=\prodr\limits_{k=1}^n\proj{S_\al\rk{z}}{\cH_k\ominus\cH_{k-1}},
\]
and the assumption follows from the preceding theorem.
\eproof

\subsection{Multiplicative representation of characteristic functions of class $\Om_J\rk{\C^r} $}\label{subsec6.3-0917}
\subsubsection{The case of real spectrum}\label{subsubsec6.3.1-0917}
Let $S\rk{z}$ belong to the class $\Om_J\rk{\C^r} $ and let
\beql{E6.3-0917}
\al
\defeq\ocol{\cH}{\C^r}{A}{\Phi}{J}
\eeq
be a simple operator colligation such that
\beql{E6.4-0917}
S_\al\rk{z}
=S\rk{z}.
\eeq
We suppose that the whole spectrum of the operator $A$ is located at the real axis.
Let $\set{h_j}_{j=1}^q$ be an arbitrary orthonormal basis of the channel subspace $\ran{\Phi^\ad}$ and let $\set{h_j}_{j=q+1}^\infty$ be an arbitrary orthonormal basis of the orthogonal complement $\cH\ominus\ran{\Phi^\ad}$.
Let further $n>q$ and $\cH_n$ be the subspace generated by the vectors $\set{h_j}_{j=1}^n$ and
\[
S_n\rk{z}
=\proj{S_\al\rk{z}}{\cH_n}.
\]
Then, in view of \rthm{T4.25-1001}, it holds the representation \eqref{E4.50-1006}, namely
\beql{E6.5-0917}
S_\al\rk{z}
=\prodr\limits_{k=1}^n\rkb{\Iu{r}+\frac{\iu}{z-z_k^{\rk{n}}}\eta_k^{\rk{n}}{\eta_k^{\rk{n}}}^\ad J},
\eeq
where, taking into account $\ran{\Phi^\ad}\subset\cH_n$, we have
\begin{gather}
    {\eta_k^{\rk{n}}}^\ad J\eta_k^{\rk{n}}=2\im z_k^{\rk{n}},\label{E6.6-0917}\\
    \sum_{k=1}^n\eta_k^{\rk{n}}{\eta_k^{\rk{n}}}^\ad=\Phi\Phi^\ad,\label{E6.7-0917}\\
    2\sum_{k=1}^n\abs{\im z_k^{\rk{n}}}\leq\tr\Phi\Phi^\ad,\label{E6.8-0917}
\end{gather}
where $P_n$ is the orthoprojection from $\cH$ onto $\cH_n$.

We will suppose that the product \eqref{E6.4-0917} is formed in such way that
\beql{E6.9-0917}
a_1^{\rk{n}}
\leq a_2^{\rk{n}}
\leq \dotsb
\leq a_n^{\rk{n}},
\qquad a_k^{\rk{n}}=\re z_k^{\rk{n}},\;k=1,2,\dotsc,n.
\eeq
The positivity of such arrangement of the factors follows from (see \rsubsec{subsec4.8-0917}) the way of deriving representation \eqref{E6.5-0917} and does not depend on the choice of ordering the eigenvalues $\set{z_k^{\rk{n}}}_{k=1}^n$.

In view of \eqref{E6.4-0917} and \rthm{T6.2-0917}, we have
\beql{E6.10-0917}
S\rk{z}=\lim S_n\rk{z},
\qquad\im z\neq0.
\eeq

We consider the products
\begin{align}
    S_{\mathrm{Re},n}\rk{z}&\defeq\prodr\limits_{k=1}^n\rkb{\Iu{r}+\frac{\iu}{z-a_k^{\rk{n}}}\eta_k^{\rk{n}}{\eta_k^{\rk{n}}}^\ad J},\label{E6.11-0917}\\
    S_{\mathrm{Re},n}^{\rk{\mathrm{exp}}}\rk{z}&\defeq\prodr\limits_{k=1}^n\ec^{\frac{\iu}{z-a_k^{\rk{n}}}\eta_k^{\rk{n}}{\eta_k^{\rk{n}}}^\ad J}.\label{E6.12-0917}
\end{align}
Using \rlemss{L5.2-0725}{L5.4-0725} and the formulas \eqref{E6.6-0917}--\eqref{E6.8-0917}, as in the proof of \rthmss{T5.5-0805}{T5.9-0829} we obtain
\begin{gather*}
    \lim_{n\to\infty}\rkb{S_n\rk{z}-S_{\mathrm{Re},n}\rk{z}}=0,\qquad\im z\neq0,\\
    \lim_{n\to\infty}\rkb{S_{\mathrm{Re},n}\rk{z}-S_{\mathrm{Re},n}^{\rk{\mathrm{exp}}}\rk{z}}=0,\qquad\im z\neq0.
\end{gather*}
From this, taking into account \eqref{E6.10-0917}, it follows
\beql{E6.13-0917}
S\rk{z}=\lim_{n\to\infty}S_{\mathrm{Re},n}^{\rk{\mathrm{exp}}}\rk{z},\qquad\im z\neq0.
\eeq

We represent each of the products \eqref{E6.12-0917} in form of a multiplicative integral.
For this, taking into account \eqref{E6.7-0917}, we consider the interval $\ek{0,\ell}$, where
\beql{E6.14-0917}
\ell
=\sum_{k=1}^n\norm{\eta_k^{\rk{n}}}^2
=\sum_{k=1}^n{\eta_k^{\rk{n}}}^\ad\eta_k^{\rk{n}}
=\sum_{k=1}^n\tr\eta_k^{\rk{n}}{\eta_k^{\rk{n}}}^\ad
=\tr\Phi\Phi^\ad.
\eeq
We decompose the interval $\ek{0,\ell}$ in parts by the points
\[
t_0=0,\quad
t_k=\sum_{s=1}^k{\eta_s^{\rk{n}}}^\ad\eta_s^{\rk{n}},\qquad
k=1,2,\dotsc,n.
\]
We define a non-decreasing function $a_n\rk{t}$ on the interval $\ek{0,\ell}$ by setting
\beql{E6.15-0917}
a_n\rk{t}
=
\begin{cases}
    a_k^{\rk{n}},&\text{ if }t_{k-1}\leq t<t_k,\;k=1,2,\dotsc,n-1,\\
    a_n^{\rk{n}},&\text{ if }t_{n-1}\leq t\leq t_n,
\end{cases}
\eeq
and a non-decreasing operator function $E_n\rk{t}$ on the interval $\ek{0,\ell}$ by setting
\beql{E6.16-0917}
E_n\rk{t}
=
\begin{cases}
    \frac{t}{t_1}\eta_1^{\rk{n}}{\eta_1^{\rk{n}}}^\ad,&\text{ if }t_{0}\leq t\leq t_1,\\
    \sum_{s=1}^{k-1}\eta_s^{\rk{n}}{\eta_s^{\rk{n}}}^\ad+\frac{t-t_{k-1}}{t_k-t_{k-1}}\eta_k^{\rk{n}}{\eta_k^{\rk{n}}}^\ad,&\text{ if }t_{n-1}\leq t\leq t_n.
\end{cases}
\eeq
Then obviously
\beql{E6.17-0917}
S_{\mathrm{Re},n}^{\rk{\mathrm{exp}}}\rk{z}
=\intr\limits_0^\ell\ec^{\frac{\iu}{z-a_n\rk{t}}\dif E_n\rk{t}J}.
\eeq

From \eqref{E6.9-0917}, \eqref{E6.15-0917} (\tresp{}, \eqref{E6.7-0917}, \eqref{E6.16-0917}) it follow
\beql{E6.18-0925}
\absb{a_n\rk{t}}\leq\norm{A},
\quad E_n\rk{t}\leq\Phi\Phi^\ad,
\qquad t\in\ek{0,\ell},\;n=1,2,3,\dotsc
\eeq
Thus, one can chose subsequences $a_{n_j}\rk{t}$ and $E_{n_j}\rk{t}$ which converge on $\ek{0,\ell}$ to a non-decreasing function $a\rk{t}$ and a non-decreasing matrix function $E\rk{t}$, respectively.
Moreover, we have
\begin{gather}
    \normb{E_n\rk{t}}\leq\tr E_n\rk{t}=t,\label{E6.19-0925}\\
    \normb{E_n\rk{t''}-E_n\rk{t'}}\leq\tr E_n\rk{t''}-\tr E_n\rk{t'}=t''-t',\quad0\leq t'\leq t''\leq\ell.\label{E6.20-0925}
\end{gather}

Taking into account \eqref{E6.13-0917}, \eqref{E6.17-0917}, and \rthm{T5.9-0829}, we obtain
\beql{E6.21-0925}
S\rk{z}=\lim_{j\to\infty}\intr\limits_0^\ell\ec^{\frac{\iu}{z-a_{n_j}\rk{t}}\dif E_{n_j}\rk{t}J}=\intr\limits_0^\ell\ec^{\frac{\iu}{z-a\rk{t}}\dif E\rk{t}J},\qquad\im z\neq0.
\eeq
We denote  by $\cM$ the set of left and right limit points of the function $a\rk{t}$.
In \zitaa{MR0100793}{Appendix} it was shown that the set of singular points of the multiplicative integral \eqref{E6.21-0925} coincides with the set $\cM$.
Hence, the set $\cM$ coincides also with the set of singular points of the function $S\rk{z}$.

In view of \eqref{E6.19-0925} and \eqref{E6.20-0925}, we have
\beql{E6.22-E925}
\tr E\rk{t}=t,
\quad\normb{E\rk{t''}-E\rk{t'}}\leq t''-t',
\qquad t,t',t''\in\ek{0,\ell},\;t'<t''.
\eeq
This means, that there exists a Lebesgue integrable operator function $M\rk{t}$, $t\in\ek{0,\ell}$, such that
\[
E\rk{t}=\int_0^tM\rk{x}\dif x,\qquad t\in\ek{0,\ell}.
\]
Obviously,
\[
M\rk{t}\geq0,\quad\tr M\rk{t}=1,\qquad\text{\tae{}\ on }\ek{0,\ell}.
\]
Let $\rank M\rk{t}\leq p$ \tae{}\ on $\ek{0,\ell}$.
Then there exists a set of positive measure such that in all of its points the rank is equal to $p$.
From \eqref{E6.7-0917} and \eqref{E6.16-0917} it follows
\[
0\leq E\rk{t}\leq E\rk{\ell}=\Phi\Phi^\ad,\qquad t\in\ek{0,\ell}.
\]
From this it follows
\[
q=\rank\Phi\geq p.
\]
Hence, there exists a measurable $\xi\rk{t}$, $t\in\ek{0,\ell}$, with values in $\ek{\C^p,\C^r}$ such that
\[
M\rk{t}=\xi\rk{t}\xi^\ad\rk{t},\qquad\text{\tae{}\ on }\ek{0,\ell}.
\]
Thus,
\begin{gather}
    E\rk{t}=\int_0^t\xi\rk{x}\xi^\ad\rk{x}\dif x,\qquad t\in\ek{0,\ell},\label{E6.23-1001}\\
    \tr\xi\rk{t}\xi^\ad\rk{t}=1,\qquad\text{\tae{}\ on }\ek{0,\ell}.\notag
\end{gather}
From \eqref{E6.21-0925} and \eqref{E6.23-1001} it follows (see \rsubsec{subsubsec5.1.3-0912})
\beql{E6.24-1001}
S\rk{z}=\intr\limits_0^\ell\ec^{\frac{\iu}{z-a\rk{t}}\xi\rk{t}\xi^\ad\rk{t}J\dif t}.
\eeq

\bthmnl{\cite{MR0100793}}{T6.4-0925}
Let $S\rk{z}\in\Om_J\rk{\C^r} $ and let $\al$  be an operator colligation \eqref{E6.3-0917} such that $S_\al\rk{z}=S\rk{z}$.
Suppose that the spectrum of $A$ lies on the real axis.
Then the function $S\rk{z}$ admits the representation \eqref{E6.24-1001}, 
where $\ell=\tr\Phi\Phi^\ad$, $a\rk{t}$ is a scalar non-decreasing function on $\ek{0,\ell}$, 
$\xi\rk{t}$ is a measurable $r\times p$\nobreakdash-matrix function on $\ek{0,\ell}$ with $p\leq q=\rank\Phi\leq r$.
Moreover, there exists a set of positive measure on $\ek{0,\ell}$ such that $\rank\xi\rk{t}=p$ on its points and the conditions
\begin{gather}
    \tr\xi\rk{t}\xi^\ad\rk{t}=1\qquad\text{\tae{}\ on }\ek{0,\ell},\label{E6.25-0925}\\
    \int_0^\ell\xi\rk{t}\xi^\ad\rk{t}\dif t=\Phi\Phi^\ad\label{E6.26-0927}
\end{gather}
hold.
The representation \eqref{E6.24-1001} holds in all points $z\in\C\setminus\cM$, where $\cM$ is the set of singular points of the multiplicative integral \eqref{E6.24-1001}, which 
consists of the left and right limit points of the function $a\rk{t}$.
\ethm

\breml{R6.5-0925}
From \rthmss{T6.4-0925}{T4.20-0721} it follows that the spectrum of the operator $A$ (the fundamental operator of the colligation \eqref{E6.3-0917}) coincides with the set $\cM$.
\erem

\subsubsection{The case that the invariant subspaces corresponding to the complex eigenvalues generate the whole space}
Let $S\rk{z}$ belong to the class $\Om_J\rk{\C^r} $ and let
\beql{E6.26-0925}
\al
\defeq\ocol{\cH}{\C^r}{A}{\Phi}{J}
\eeq
be a simple operator colligation such that
\beql{E6.27-0925}
S_\al\rk{z}=S\rk{z}.
\eeq
We suppose that the space $\cH$ is the closure of the linear hull of the invariant subspaces $\cH_{w_k}$, $k=1,2,3,\dots,\ome$, $\ome\leq\infty$, formed by the non-real points $w_k$ of the spectrum of the operator $A$ (see \rthm{T4.16-0715}).
In the following we suppose that $\ome=\infty$.

Let
\beql{E6.28-0925}
\cH_n=\cH_{w_1}\dotplus\cH_{w_2}\dotplus\dotsb\dotplus\cH_{w_n},\qquad n=1,2,\dotsc
\eeq
In view of \eqref{E6.27-0925} and \rcor{C6.3-0917}, then
\beql{E6.29-0925}
S\rk{z}=\prod_{k=1}^\infty\proj{S_\al\rk{z}}{\cH_k\ominus\cH_{k-1}},\qquad\cH_0=\set{0},
\eeq
where the infinite product \eqref{E6.29-0925} converges uniformly to $S\rk{z}$ in the neighborhood of each regular point of $S\rk{z}$.
Taking into account that each of the subspaces $\cH_k\ominus\cH_{k-1}$, $k=1,2,3,\dots$, is finite dimensional and that the operator $A_k=\rstr{A}{\cH_k\ominus\cH_{k-1}}$ has only the eigenvalues $z_k$, applying \rthm{T4.25-1001}, we find that
\begin{multline*}
    \proj{S_\al\rk{z}}{\cH_k\ominus\cH_{k-1}}=\prodr\limits_{j=1}^{q_k}\rkb{\Iu{r}+\frac{\iu}{z-z_k}\eta_j^{\rk{k}}{\eta_j^{\rk{k}}}^\ad J},\\
    \eta_j^{\rk{k}}J{\eta_j^{\rk{k}}}^\ad=2\im z_k,\quad j=1,2,\dotsc,s_k;\;k=1,2,3,\dotsc,
\end{multline*}
where $q_k=\dim\rk{\cH_k\ominus\cH_{k-1}}$.
Thus, taking into account \eqref{E6.29-0925}, we obtain
\[
S\rk{z}=\prodr\limits_{k=1}^\infty\rkb{\Iu{r}+\frac{\iu}{z-\lambda_k}\eta_k\eta_k^\ad J},
\]
where
\begin{gather}
    \lambda_1=\lambda_2=\dotsb=\lambda_{s_1}=z_1,\;\lambda_{s_1+1}=\lambda_{s_1+2}=\dotsb=\lambda_{s_1+s_2}=z_2,\;\dotsc\notag\\
    \eta_k=\eta_k^{\rk{1}},\;k=1,2,\dotsc,s_1;\qquad\eta_{s_1+k}=\eta_k^{\rk{2}},\; k=1,2,\dotsc,s_2;\qquad\dotsc,\notag\\
    \eta_kJ{\eta_k}^\ad=2\im \lambda_k,\quad k=1,2,3,\dotsc,\label{E6.30-0925}
\end{gather}
where as it follows from \rthm{T4.25-1001},
\beql{E6.31-0925}
\sum_{k=1}^\infty\eta_k\eta_k^\ad=\Phi\Phi^\ad.
\eeq
Thus, we get the following result.

\bthmnl{\cite{MR0100793}}{T6.6-0925}
Let $S\rk{z}\in\Om_J\rk{\C^r} $ and let $\al$ be a simple operator colligation \eqref{E6.26-0925}, where the space $\cH$ is the closure of the linear hull of the invariant subspaces which correspond to the non-real points of the spectrum of the operator $A$.
Then the function $S\rk{z}$ admits the representation
\beql{E6.32-0925}
S\rk{z}=\prodr\limits_{k=1}^\infty\rkb{\Iu{r}+\frac{\iu}{z-\lambda_k}\eta_k\eta_k^\ad J},
\eeq
where $\set{\lambda_k}_{k=1}^\infty$ is the sequence of eigenvalues of the operator $A$ counted with their multiplicity, $\set{\eta_k}_{k=1}^\infty$ is a sequence of column vectors from the space $\C^r$, and equalities \eqref{E6.30-0925} and \eqref{E6.31-0925} are satisfied.
The representation \eqref{E6.32-0925} holds in all points $z\in\C\setminus\Lambda$, where $\Lambda$ is the closure of the set $\set{\lambda_k}_{k=1}^\infty$.
\ethm

\subsubsection{The general case}
Let $S\rk{z}$ belong to the class $\Om_J\rk{\C^r} $ and let
\beql{E6.33-0925}
\al\defeq\ocol{\cH}{\C^r}{A}{\Phi}{J}
\eeq
be a simple operator colligation such that
\beql{E6.34-0925}
S\rk{z}=S_\al\rk{z}.
\eeq

We denote by $\cH_1$ the closure of the linear hull of the invariant subspaces which correspond to the non-real points of the spectrum of the operator $A$.
Let $\cH_2=\cH\ominus\cH_1$.
With respect to the decomposition $\cH=\cH_1\oplus\cH_2$ the operator $A$ has the block representation
\[
A=\begin{pmatrix}A_1&A_{12}\\0&A_2\end{pmatrix},
\]
where $A_1\in\ek{\cH_1}$, $A_2\in\ek{\cH_2}$, $A_{12}\in\ek{\cH_2,\cH_1}$.
In view of \rthm{T3.7}, we have
\beql{E6.34A-0925}
\al=\al_1\al_2,
\eeq
where
\begin{gather*}
    \al_j=\rk{\cH_j,\C^r;A_j,\Phi_j,J},\qquad j=1,2,\\
    \Phi_j=\rstr{\Phi}{\cH_j},\qquad j=1,2,
\end{gather*}
where, taking into account \eqref{E6.34A-0925}, we obtain
\[
S\rk{z}=S_{\al_1}\rk{z}S_{\al_2}\rk{z}.
\]
Applying now the \rthmss{T6.6-0925}{T6.4-0925} to $S_{\al_1}\rk{z}$ and $S_{\al_2}\rk{z}$, \tresp{}, we obtain the following result.

\bthmnl{\cite{MR0100793}}{T6.6A-0925}
Let $S\rk{z}\in\Om_J\rk{\C^r} $ and let $\al$ be an operator colligation of the form \eqref{E6.33-0925}.
Then the function $S\rk{z}$ admits the representation
\beql{E6.35-0925}
S\rk{z}
=\prodr\limits_{k=1}^\infty\rkb{\Iu{r}+\frac{\iu}{z-\lambda_k}\eta_k\eta_k^\ad J}\intr\limits_0^\ell\ec^{\frac{\iu}{z-a\rk{t}}\xi\rk{t}\xi^\ad\rk{t}J\dif t},
\eeq
where the sequences $\set{\lambda_k}_{k=1}^\infty$, $\set{\eta_k}_{k=1}^\infty$ and the functions $a\rk{t}$ and $\xi\rk{t}$ satisfy the conditions of \rthmss{T6.6-0925}{T6.4-0925}, respectively.
\ethm

\breml{R6.7-0925}
In view of \eqref{E6.34A-0925}, \eqref{E2.29} and \eqref{E6.30-0925}, \eqref{E6.26-0927}, we have
\beql{E6.36-0925}
\sum_{k=1}^\infty\eta_k\eta_k^\ad+\int_0^\ell\xi\rk{t}\xi^\ad\rk{t}\dif t
=\Phi_1\Phi_1^\ad+\Phi_2\Phi_2^\ad
=\Phi\Phi^\ad.
\eeq
\erem

\section{Triangular model of a bounded linear operator which is simple with respect to a finite dimensional subspace}\label{S1551}
\subsection{The case of real spectrum}
Let $\cH$ be a Hilbert space and let the operator $A\in\ek{\cH}$ be simple with respect to the finite dimensional subspace $\cE\subset\cH$ (see \rdefn{D4.23-1015}).
We suppose that $\sigma\rk{A}\subset\R$.
We embed the operator $A$ in a simple colligation (see \rthm{T2.2} and \rrem{R2.3-0929})
\beql{E7.1-0928}
\al
\defeq\ocol{\cH}{\C^r}{A}{\Phi}{J},
\eeq
such that $\cE=\ran{\Phi^\ad}$.
From \rthm{T6.4-0925} it follows that the characteristic function $S_\al\rk{z}$ of the colligation $\al$ admits the representation \eqref{E6.24-1001}:
\beql{E7.2-0928}
S_\al\rk{z}
=\intr\limits_0^\ell\ec^{\frac{\iu}{z-a\rk{t}}\xi\rk{t}\xi^\ad\rk{t}J\dif t}.
\eeq
Here $a\rk{t}$ is a scalar non-decreasing function on $\ek{0,\ell}$, $\xi\rk{t}$ is a measurable $r\times p$\nobreakdash-matrix function on $\ek{0,\ell}$, where $p\leq q=\rank\Phi\leq r$, such that there exists a set of positive measure on $\ek{0,\ell}$ such that $\rank\xi\rk{t}=p$ on its points and the conditions (see \eqref{E6.25-0925} and \eqref{E6.26-0927})
\begin{gather}
    \tr\xi\rk{t}\xi^\ad\rk{t}=1\qquad\text{\tae{}\ on }\ek{0,\ell},\label{E7.3-1015}\\
    \int_0^\ell\xi\rk{t}\xi^\ad\rk{t}\dif t=\Phi\Phi^\ad\label{E7.4-1015}
\end{gather}
hold.

We consider now the space
\beql{E7.5-1021}
\hat{\cH}_\cm
\defeq L^2\rkb{\ek{0,\ell},\C^r},
\eeq
the elements of which are defined on $\ek{0,\ell}$ by the vector functions
\beql{E7.3A-0928}
h\rk{x}\defeq\col\setb{h_1\rk{x},h_2\rk{x},\dotsc,h_p\rk{x}},
\quad h_k\rk{x}\in L^2\ek{0,\ell},\;k=1,2,\dots,p,
\eeq
where the number $p$ is defined by the conditions of \rthm{T6.4-0925}.
The scalar product of the elements $h,g\in\hat{\cH}_\cm$ is defined in the following way
\beql{E7.3B-0928}
\ipa{h}{f}
\defeq\int_0^\ell\ipab{h\rk{x}}{f\rk{x}}_\C\dif x
=\int_0^\ell\sum_{k=1}^ph_k\rk{x}\ko{f_k\rk{x}}\dif x.
\eeq

We define in $\hat{\cH}_\cm$ the operator $\hat{A}_\cm$ by setting
\beql{E7.4-0928}
\rk{\hat{A}_\cm h}\rk{x}
\defeq a\rk{x}h\rk{x}+\iu\int_x^\ell\xi^\ad\rk{x}J\xi\rk{t}h\rk{t}\dif t.
\eeq
From the boundedness of the function $a\rk{x}$ on $\ek{0,\ell}$ and \eqref{E7.4-1015} it follows the boundedness of the operator $\hat{A}_\cm$.

It can be easily seen that
\beql{E7.9-0505}
\rk{\hat{A}_\cm^\ad h}\rk{x}
\defeq a\rk{x}h\rk{x}-\iu\int_0^x\xi^\ad\rk{x}J\xi\rk{t}h\rk{t}\dif t. 
\eeq
Thus,
\beql{E7.5-0928}
\rkb{\frac{\hat{A}_\cm-\hat{A}_\cm^\ad}{\iu}h}\rk{x}
=\int_0^\ell\xi^\ad\rk{x}J\xi\rk{t}h\rk{t}\dif t.
\eeq

We define the mapping $\hat{\Phi}_\cm\in\ek{\hat{\cH}_\cm,\C^r}$ by setting
\beql{E7.6-0928}
\hat{\Phi}_\cm h\defeq \int_0^\ell\xi\rk{t}h\rk{t}\dif t,
\qquad h\in\hat{\cH}_\cm.
\eeq
Then
\[
\rk{\hat{\Phi}_\cm^\ad g}\rk{x}=\xi^\ad\rk{x}g,\qquad g\in\C^r,\;x\in\ek{0,\ell},
\]
and \eqref{E7.5-0928} can be rewritten in the form
\beql{E7.11-1024}
\frac{1}{\iu}\rk{\hat{A}_\cm-\hat{A}_\cm^\ad}
=\hat{\Phi}_\cm^\ad J\hat{\Phi}_\cm.
\eeq
Thus, the quintuple
\beql{E7.7-0928}
\hat{\al}_\cm
\defeq\rk{\hat{\cH}_\cm,\C^r;\hat{A}_\cm,\hat{\Phi}_\cm,J}
\eeq
is an operator colligation.

Let the function $S\rk{x,z}$, $x\in\ek{0,\ell}$, $z\in\rho\rk{A}$, the values of which are linear operators in $\C^r$, be a solution of the Cauchy problem
\beql{E7.11-1018}
\begin{cases}
    \frac{\dif S\rk{x,z}}{\dif x}=\frac{\iu}{z-a\rk{x}}S\rk{x,z}\xi\rk{x}\xi^\ad\rk{x}J,&x\in\ek{0,\ell},\\
    S\rk{0,z}\equiv\Iu{r}.
\end{cases}
\eeq
Then
\beql{E7.12-1018}
S\rk{x,z}=\Iu{r}+\int_0^x\frac{\iu}{z-a\rk{t}}S\rk{t,z}\xi\rk{t}\xi^\ad\rk{t}\dif tJ
\eeq
and (see \rsubsec{subsubsec5.1.3-0912})
\[
S\rk{x,z}=\intr\limits_0^x\ec^{\frac{\iu}{z-a\rk{t}}\xi\rk{t}\xi^\ad\rk{t}J}.
\]
The Cauchy problem \eqref{E7.11-1018} and the integral equation \eqref{E7.12-1018} are analogous to the Cauchy problem \eqref{E5.50-0912} and the  integral equation \eqref{E5.51-0912}, respectively.
Now setting (see \eqref{E5.53-0912})
\beql{E7.13-1018}
\ze\rk{x,z}
\defeq-\frac{1}{z-a\rk{x}}S\rk{x,z}\xi\rk{x}
\eeq
and repeating the consideration of \rsubsec{subsec5.4-1018} we get
\beql{E7.16-1031}
S_{\hat{\al}_\cm}\rk{z}
=S\rk{\ell,z}
=\intr\limits_0^\ell\ec^{\frac{\iu}{z-a\rk{t}}\xi\rk{t}\xi^\ad\rk{t}J\dif t}.
\eeq
Thus,
\[
S_{\hat{\al}_\cm}\rk{z}
=S_\al\rk{z}.
\]
From this and \rthm{T3.4} we obtain the following result.

\bthmnl{\cite{MR0062955,MR0100793}}{T7.1-0928}
The operator colligation $\al$ (see \eqref{E7.1-0928}) is unitarily equivalent to the principal part of the colligation $\hat{\al}_\cm$.
\ethm


\bcorl{C7.2-0928}
Let the operator $A\in\ek{\cH}$ be simple with respect to a finite dimensional subspace and suppose  that it possesses real spectrum.
Then the operator $A$ is unitarily equivalent to the completely non-unitary part of the model operator $\hat{A}_\cm$ of form \eqref{E7.4-0928}.
\ecor

\subsection{The case that the invariant subspaces which correspond to the complex eigenvalues generate the whole space}

Let $\cH$ be a Hilbert space and let $A\in\ek{\cH}$ be simple with respect to the finite dimensional subspace $\cE\subset\cH$ (see \rdefn{D4.23-1015}).
We suppose that the space $\cH$ is the closure of the linear hull of the invariant subspaces formed by the non-real points of the spectrum of the operator $A$ (see \rthm{T4.16-0715}).
For definiteness we assume that there are infinitely many such points.
We embed the operator $A$ in a simple colligation (see \rthm{T2.2} and \rrem{R2.3-0929})
\beql{E7.14-1018}
\al
\defeq\ocol{\cH}{\C^r}{A}{\Phi}{J},
\eeq
such that $\cE=\ran{\Phi^\ad}$.
From \rthm{T6.6-0925} it follows that the characteristic function $S_\al\rk{z}$ of the colligation $\al$ admits the representation \eqref{E6.32-0925}

\beql{E7.15-1018}
S_\al\rk{z}=\prodr\limits_{k=1}^\infty\rkb{\Iu{r}+\frac{\iu}{z-\lambda_k}\eta_k\eta_k^\ad J}.
\eeq
Here $\set{\lambda_k}_{k=1}^\infty$ is the sequence of non-real eigenvalues of the operator $A$ counted with respect to their multiplicity and $\set{\eta_k}_{k=1}^\infty$ is a sequence of column vectors from the space $\C^r$ such that the equations
\begin{gather}
    \eta_k J\eta_k^\ad=2\im\lambda_k,\qquad k=1,2,\dotsc,\label{E7.16-1018}\\
    \sum_{k=1}^\infty\eta_k\eta_k^\ad=\Phi\Phi^\ad\label{E7.17-1018}
\end{gather}
hold.

We consider now the space
\beql{E7.17-1021}
\tilde{\cH}_\cm
\defeq\ell^2,
\eeq
the elements of which have the form
\beql{E7.21-1024}
h=\col\rk{h_1,h_2,\dotsc,h_n,\dotsc},\qquad h_n\in\C,\;\sum_{k=1}^\infty\abs{h_k}^2<+\infty,
\eeq
and the scalar product of the elements $h,f\in\ell^2$ is defined in the usual way
\[
\ipa{h}{f}
\defeq\sum_{k=1}^\infty h_k\ko{f_k}.
\]

We define the operator $\tilde{A}_\cm$ in $\ell^2$ by setting
\beql{E7.18-1018}
\rk{\tilde{A}_\cm h}_k
=\lambda_k h_k+\iu\sum_{\ell=k+1}^\infty\eta_k^\ad J\eta_\ell h_\ell,
\qquad k=1,2,3,\dotsc
\eeq
From the boundedness of the sequence $\set{\lambda_k}_{k=1}^\infty$ and \eqref{E7.14-1018} it follows the boundedness of the operator $\tilde{A}_\cm$.

It can be easily seen that
\beql{E7.19-1018}
\rk{\tilde{A}_\cm^\ad h}_k
=\ko{\lambda_k}h_k-\iu\sum_{\ell=1}^{k-1}\eta_k^\ad J\eta_\ell h_\ell,
\qquad k=1,2,3,\dotsc
\eeq
Thus taking into account \eqref{E7.13-1018}, we get
\beql{E7.20-1018}
\rkb{\frac{\tilde{A}_\cm-\tilde{A}_\cm^\ad}{\iu}h}_k
=\sum_{\ell=1}^\infty\eta_k^\ad J\eta_\ell h_\ell,
\qquad k=1,2,3,\dotsc
\eeq

We define the mapping $\tilde{\Phi}_m\in\ek{\ell^2,\C^r}$ by setting
\beql{E7.21-1018}
\tilde{\Phi}_\cm h
\defeq\sum_{\ell=1}^\infty\eta_\ell h_\ell.
\eeq
Then
\beql{E7.22-1018}
\rk{\tilde{\Phi}_\cm^\ad g}_k
=\eta_k^\ad g,
\qquad g\in\C^r,\;k=1,2,3,\dotsc
\eeq
and \eqref{E7.16-1018} can be rewritten in the form
\beql{E7.27-1024}
\frac{1}{\iu}\rk{\tilde{A}_\cm-\tilde{A}_\cm^\ad}
=\tilde{\Phi}_\cm^\ad J\tilde{\Phi}_\cm.
\eeq
Thus, the quintuple
\beql{E7.29-1031}
\tilde{\al}_\cm
\defeq\ocol{\tilde{\cH}_\cm}{\C^r}{\tilde{A}_\cm}{\tilde{\Phi}_\cm}{J}
\eeq
is an operator colligation.

We consider the sequence of matrix functions
\beql{E7.23-1018}
\begin{cases}
    S\rk{0,z}\equiv\Iu{r},\\
    S\rk{k,z}=\prodr\limits_{\ell=1}^k\rk{\Iu{r}+\frac{\iu}{z-\lambda_\ell}\eta_\ell\eta_\ell^\ad J},&k=1,2,3,\dotsc
\end{cases}
\eeq
The sequence $\set{S\rk{k,z}}_{k=1}^\infty$ satisfies the system of difference equations
\beql{E7.24-1018}
\begin{cases}
    S\rk{k+1,z}-S\rk{k,z}=\frac{\iu}{z-\lambda_{k+1}}S\rk{k,z}\eta_{k+1}\eta_{k+1}^\ad J,&k=0,1,2,\dotsc,\\
    S\rk{0,z}\equiv\Iu{r}.
\end{cases}
\eeq
Hence,
\beql{E7.25-1018}
S\rk{k,z}
=\Iu{r}+\sum_{\ell=1}^k\frac{\iu}{z-\lambda_\ell}S\rk{\ell-1,z}\eta_\ell\eta_\ell^\ad J,
\qquad k=1,2,\dotsc
\eeq
In view of \eqref{E7.23-1018}, we get
\beql{E7.26-1018}
S_\al\rk{z}
=S\rk{\infty,z}
=\lim_{k\to\infty}S\rk{k,z}.
\eeq
We note that the system of difference equations \eqref{E7.23-1018} and the identities \eqref{E7.24-1018} are discrete analogues of the Cauchy problem \eqref{E5.50-0912} (see also \eqref{E7.11-1018}) and the integral equation \eqref{E5.51-0912} (see also \eqref{E7.12-1018}), respectively.

Repeating the considerations of \rsubsec{subsec5.4-1018}, we set (see \eqref{E5.53-0912} and also \eqref{E7.13-1018})
\beql{E7.27-1018}
\zeta_k\rk{z}
\defeq-\frac{1}{z-\lambda_k}S\rk{k-1,z}\eta_k,
\qquad k=1,2,3,\dotsc
\eeq
Then the vector
\[
\zeta_g\rk{z}
\defeq\col\rkb{\ze_1^\ad\rk{z}g,\ze_2^\ad\rk{z}g,\dotsc,\ze_k^\ad\rk{z}g,\dotsc},
\qquad g\in\C^r,
\]
belongs to $\ell^2$.

From \eqref{E7.19-1018}, \eqref{E7.22-1018}, \eqref{E7.25-1018}, and \eqref{E7.27-1018} we find
\[\begin{split}
    \rkb{\rk{\tilde{A}^\ad_\cm-\ko{z}\Iu{\tilde{\cH}_\cm}}\ze_g\rk{z}}_k
    &=\ko{\lambda_k}\ze_k^\ad\rk{z}g-\iu\sum_{\ell=1}^{k-1}\eta_k^\ad J\eta_\ell\ze_\ell^\ad\rk{z}g-\ko{z}\ze_k^\ad\rk{z}g\\
    &=\eta_k^\ad\rkb{S^\ad\rk{k-1,z}+\sum_{\ell=1}^{k-1}\frac{\iu}{\ko{z}-\ko{\lambda_\ell}}J\eta_\ell\eta_\ell^\ad S^\ad\rk{\ell-1,z}}g\\
    &=\eta_k^\ad g
    =\rk{\tilde{\Phi}_\cm^\ad g}_k,\qquad k=1,2,3,\dotsc
\end{split}\]
From this by taking into account \eqref{E7.22-1018}, \eqref{E7.25-1018}, and \eqref{E7.27-1018} we infer for $g_1,g_2\in\C^r$ then
\[\begin{split}
    &\ipab{\tilde{\Phi}_\cm\rk{\tilde{A}_\cm-z\Iu{\tilde{\cH}_\cm}}^\inv\tilde{\Phi}_\cm^\ad Jg_1}{g_2}_{\C^r}
    =\ipab{\tilde{\Phi}_\cm^\ad Jg_1}{\rk{\tilde{A}_\cm^\ad-\ko{z}\Iu{\tilde{\cH}_\cm}}^\inv\tilde{\Phi}_\cm^\ad g_2}_{\ell^2}\\
    &=\ipab{\tilde{\Phi}_\cm^\ad Jg_1}{\ze_{g_2}\rk{z}}_{\ell^2}
    =\sum_{\ell=1}^\infty g_2^\ad\ze_\ell\rk{z}\eta_\ell^\ad Jg_1\\
    &=\ipab{-\sum_{\ell=1}^\infty\frac{1}{z-\lambda_\ell}S\rk{\ell-1,z}\eta_\ell\eta_\ell^\ad Jg_1}{g_2}_{\C^r}
    =\ipab{\iu\rkb{S\rk{\infty,z}-\Iu{r}}g_1}{g_2}_{\C^r}.
\end{split}\]
Thus
\[
S_{\tilde{\al}_\cm}\rk{z}
=S_\al\rk{z}.
\]
Thus
\beql{E7.35-1031}
S_{\tilde{\al}_\cm}\rk{z}
=\prodr\limits_{k=1}^\infty\rkb{\Iu{r}+\frac{\iu}{z-\lambda_k}\eta_k\eta_k^\ad J}.
\eeq
From \rthm{T7.7-1031} (see \rsubsec{subsec7.4-0505}) it follows that the colligation $\tilde{\al}_\cm$ is simple.
From this and \rthm{T3.4} we obtain the following result.

\bthmnl{\cite{MR0062955,MR0100793}}{T7.3-1024}
The operator colligation $\al$ (see \eqref{E7.1-0928}) is unitarily equivalent to the operator colligation $\tilde{\al}_\cm$.
\ethm

\bcorl{C7.4}
Let the operator $A\in\ek{\cH}$ be simple with respect to a finite dimensional subspace and suppose that the space $\cH$ is the closure of the linear hull of the invariant subspaces which correspond to the non-real points of the spectrum of the operator $A$.
Then the operator $A$ is unitarily equivalent to the model operator $\tilde{A}_\cm$ of form \eqref{E7.18-1018}.
\ecor

\subsection{The general case}
Let $\cH$ be a Hilbert space and let $A\in\ek{\cH}$ be simple with respect to the finite dimensional subspace $\cE\subset\cH$ (see \rdefn{D4.23-1015}).
We embed the operator $A$ in a simple colligation (see \rthm{T2.2} and \rrem{R2.3-0929})
\beql{E7.32-1024}
\al
\defeq\ocol{\cH}{\C^r}{A}{\Phi}{J},
\eeq
such that $\cE=\ran{\Phi^\ad}$.
The characteristic function $S_\al\rk{z}$ belongs to the class $\Om_J\rk{\C^r} $ and from 
\rthm{T6.6A-0925} it follows that $S_\al\rk{z}$ admits the representation \eqref{E6.35-0925}, namely
\beql{E7.33-1024}
S_\al\rk{z}
=\prodr_{k=1}^\infty\rkb{\Iu{r}+\frac{1}{z-\lambda_k}\eta_k\eta_k^\ad J}\intr_0^\ell\ec^{\frac{1}{z-a\rk{t}}\xi\rk{t}\xi^\ad\rk{t}J\dif t}.
\eeq
Here $\set{\lambda_k}_{k=1}^\infty$ is the sequence of non-real eigenvalues of the operator 
$A$ counted with respect to their multiplicity (we suppose that there are infinitely many of such eigenvalues), $\set{\eta_k}_{k=1}^\infty$ is a sequence of column vectors from the space $\C^r$ such that the equations (see \eqref{E6.30-0925})
\beql{E7.35-1024}
\eta_k J\eta_k^\ad
=2\im\lambda_k,
\qquad k=1,2,3,\dotsc
\eeq
hold.
Further, $a\rk{t}$ is a scalar non-decreasing function on $\ek{0,\ell}$, $\xi\rk{t}$ is a measurable $r\times p$\nobreakdash-matrix function on $\ek{0,\ell}$, where $p\leq r$, such that there exists a set of positive measure on $\ek{0,\ell}$ such that $\rank\xi\rk{t}=p$ on its points and the conditions (see \eqref{E6.25-0925} and \eqref{E6.36-0925})
\begin{gather}
    \tr\xi\rk{t}\xi^\ad\rk{t}=1,\qquad\text{\tae{}\ on }\ek{0,\ell},\label{E7.36-1024}\\
    \sum_{k=1}^\infty\eta_k\eta_k^\ad+\int_0^\ell\xi\rk{t}\xi^\ad\rk{t}\dif t
    =\Phi\Phi^\ad\label{E7.37-1024}
\end{gather}
are satisfied.

We consider now the space
\beql{E7.38-1024}
\cH_\cm
\defeq\tilde{\cH}_\cm\oplus\hat{\cH}_\cm
\eeq
where $\tilde{\cH}_\cm\defeq\ell^2$, $\hat{\cH}_\cm\defeq L^2\rk{\ek{0,\ell},\C^r}$ (see \eqref{E7.17-1021} and \eqref{E7.5-1021}), the elements $h_\cm$ of which have the form $h_\cm\defeq\rk{h,h\rk{x}}$ where the vector $h$ and the vector function $h\rk{x}$ have the form \eqref{E7.21-1024} and \eqref{E7.3A-0928}, respectively.

In the space $\cH_\cm$ we consider the operator $A_\cm h_\cm=g_\cm$ by setting
\beql{E7.39-1024}
\begin{cases}
    g_k\defeq\lambda_kh_k+\iu\sum_{\ell=k+1}^\infty\eta_k^\ad J\eta_\ell h_\ell+\iu\int_0^\ell\eta_k^\ad J\xi\rk{t}h\rk{t}\dif t,\vspace{6pt}\\
    g\rk{t}\defeq\al\rk{x}h\rk{x}+\iu\int_x^\ell\xi^\ad\rk{x}J\xi\rk{t}h\rk{t}\dif t.
\end{cases}
\eeq
We note that the operator $\Ga_\cm\in\ek{\hat{\cH}_\cm,\tilde{\cH}_\cm}$, of the form
\[
\rk{g}_k
=\iu\int_0^\ell\eta_k^\ad J\xi\rk{t}h\rk{t}\dif t,
\qquad k=1,2,3,\dots,
\]
where $h\rk{t}\in\hat{\cH}_\cm$ and $g\in\tilde{\cH}_\cm$, can be written in the form
\beql{E7.40-1024}
\Ga_\cm
=\iu\tilde{\Phi}_\cm^\ad J\hat{\Phi}_\cm,
\eeq
where the operators $\hat{\Phi}_\cm\in\ek{\hat{\cH}_\cm,\C^r}$ and $\tilde{\Phi}_\cm\in\ek{\tilde{\cH}_\cm,\C^r}$ have the form \eqref{E7.6-0928} and \eqref{E7.21-1018}, respectively.
Thus, from \eqref{E7.39-1024}, \eqref{E7.18-1018}, \eqref{E7.4-0928}, and \eqref{E7.40-1024} it follows that with respect to the decomposition \eqref{E7.38-1024} of the space $\cH_\cm$ the operator $A_\cm$ has the block form
\beql{E7.41-1024}
A_\cm
=\Mat{\tilde{A}_\cm&\Ga_\cm\\0&\hat{A}_\cm}
=\Mat{\tilde{A}_\cm&\iu\tilde{\Phi}_\cm^\ad J\hat{\Phi}_\cm\\0&\hat{A}_\cm}.
\eeq
From this, taking into account \eqref{E7.11-1024} and \eqref{E7.27-1024}, we get
\beql{E7.42-1024}\begin{split}
    \frac{1}{\iu}\rk{A_\cm-A_\cm^\ad}
    &=\Mat{\frac{1}{\iu}\rk{\tilde{A}_\cm-\tilde{A}_\cm^\ad}&\tilde{\Phi}_\cm^\ad J\hat{\Phi}_\cm\\\hat{\Phi}_\cm^\ad J\tilde{\Phi}_\cm&\frac{1}{\iu}\rk{\hat{A}_\cm-\hat{A}_\cm^\ad}}\\
    &=\Mat{\tilde{\Phi}_\cm^\ad J\tilde{\Phi}_\cm&\tilde{\Phi}_\cm^\ad J\hat{\Phi}_\cm\\\hat{\Phi}_\cm^\ad J\tilde{\Phi}_\cm&\hat{\Phi}_\cm^\ad J\hat{\Phi}_\cm}
    =\Phi_\cm^\ad J\Phi_\cm
\end{split}\eeq
where the operator $\Phi_\cm\in\ek{\cH_\cm,\C^r}$ with respect to the decomposition \eqref{E7.38-1024} has the block form
\beql{E7.46-1031}
\Phi_\cm
\defeq\rk{\tilde{\Phi}_\cm,\hat{\Phi}}.
\eeq
From \eqref{E7.42-1024} it follows that the quintuple
\beql{E7.43-1024}
\al_\cm
\defeq\ocol{\cH_\cm}{\C^r}{A_\cm}{\Phi_\cm}{J}
\eeq
is an operator colligation.

From \eqref{E7.41-1024} and \eqref{E7.46-1031} it follows that the operator colligation $\al_\cm$ can be represented as product of the colligations $\tilde{\al}_\cm$ and $\hat{\al}_\cm$ of the form \eqref{E7.29-1031} and \eqref{E7.7-0928}, respectively:
\beql{E7.48-1031}
\al_\cm
=\tilde{\al}_\cm\hat{\al}_\cm.
\eeq
From this, taking into account \rthm{T2.7} and the form of the characteristic functions $S_{\tilde{\al}_\cm}\rk{z}$, $S_{\hat{\al}_\cm}\rk{z}$, and $S_\al\rk{z}$ (see \eqref{E7.35-1031}, \eqref{E7.3A-0928}, and \eqref{E7.33-1024}, respectively), we get
\[\begin{split}
    S_{\al_\cm}\rk{z}
    &=S_{\tilde\al_\cm}\rk{z}S_{\hat\al_\cm}\rk{z}\\
    &=\prodr_{k=1}^\infty\rkb{\Iu{r}+\frac{1}{z-\lambda_k}\eta_k\eta_k^\ad J}\intr_0^\ell\ec^{\frac{1}{z-a\rk{t}}\xi\rk{t}\xi^\ad\rk{t}J\dif t}
    =S_\al\rk{z}.
\end{split}\]
From this and \rthm{T3.4} we obtain the following result.

\bthmnl{\cite{MR0062955,MR0100793}}{T7.5-1031}
The operator colligation $\al$ (see \eqref{E7.35-1031}) is unitarily equivalent to the principal part of the colligation $\al_\cm$ (see \eqref{E7.43-1024}).
\ethm

\bcorl{C7.6-1031}
Let the operator $A\in\ek{\cH}$ be simple with respect to a finite dimensional subspace.
Then the operator $A$ is unitarily equivalent to the completely non-unitary part of the model operator $A_\cm$ of the form \eqref{E7.39-1024}.
\ecor

\subsection{Redundant part of the triangular model}\label{subsec7.4-0505}
An important role for the description of the redundant part of the model colligation $\al_\cm$ of the form \eqref{E7.43-1024} (see \rdefn{D2.15-1031}) plays the block representation \eqref{E7.41-1024} of the model operator $A_\cm$ and the factorization generated by the representation \eqref{E7.48-1031}.

\bthmnl{\cite{MR0100793}}{T7.7-1031}
The additional part of the model colligation $\al_\cm$ (see \eqref{E7.43-1024}) coincides with the additional component of the colligation $\hat\al_\cm$ of form \eqref{E7.7-0928}.
\ethm
\bproof
The block representation \eqref{E7.41-1024} of the model operator $A_\cm$ is generated by the decomposition \eqref{E7.38-1024} of the model space $\cH_\cm$ as the orthogonal sum of the subspaces $\tilde\cH_\cm=\ell^2$ and $\hat\cH_\cm=L^2\rk{\ek{0,\ell},\C^r}$.
We denote by $e_k$ the vector of the subspace $\tilde\cH_\cm$, the $k$\nobreakdash-th coordinate of which is $1$ whereas all remaining coordinates are $0$.
The vectors $\set{e_k}_{k=1}^\infty$ form an orthonormal basis in $\tilde\cH_\cm$ and with respect to this basis the operator $\tilde{A}_\cm\defeq\rstr{A_\cm}{\tilde\cH_\cm}$ has the triangular representation
\begin{align*}
    \tilde{A}_\cm e_1&=\lambda_1e_1,\\
    \tilde{A}_\cm e_2&=a_{12}e_1+\lambda_2e_2,\\
    \tilde{A}_\cm e_3&=a_{13}e_1+a_{23}e_2+\lambda_3e_3,\\
    &\vdots
\end{align*}
where $a_{j\ell}=\iu\eta_j^\ad J\eta_\ell$ for $j<\ell$.

Let (see \rdefn{D2.15-1031})
\beql{E7.49-1031}
\cH_\cm
=\cH_{\al_\cm}\oplus\cH_{\al_\cm}^{\rk{0}}
\eeq
and let $P_0$ be the orthoprojection from $\cH_\cm$ onto the subspace $\cH_{\al_\cm}^{\rk{0}}$, which is an inner subspace of the redundant part of the colligation $\al_\cm$.

With respect to the decomposition \eqref{E7.49-1031} the operator $A_\cm$ is written as the orthogonal sum
\beql{E7.50-1031}
A_\cm
=A_{\al_\cm}\oplus A_{\al_\cm}^{\rk{0}},
\eeq
where $A_{\al_\cm}\in\ek{\cH_{\al_\cm}}$, $A_{\al_\cm}^{\rk{0}}\in\ek{\cH_{\al_\cm}^{\rk{0}}}$ and $A_{\al_\cm}^{\rk{0}}$ is a self-adjoint operator.

From \eqref{E7.50-1031} it follows $A_{\al_\cm}^{\rk{0}}P_0=A_\cm P_0=P_0A_\cm$ and
\[
A_{\al_\cm}^{\rk{0}}P_0e_1
=P_0A_\cm e_1
=\lambda_1e_1.
\]
Since $A_{\al_\cm}^{\rk{0}}$ is a self-adjoint operator and $\lambda_1\neq\ko{\lambda_1}$, it follows $P_0e_1=0$.
Analogously, we find
\[
A_{\al_\cm}^{\rk{0}}P_0e_2
=P_0A_\cm e_2
=P_0\rk{a_{12}e_1+\lambda_2e_2}
=\lambda_2P_0e_2
\]
and, thus, $P_0e_2=0$.
Continuing this procedure, we get $P_0h=0$ for each $h\in\tilde\cH_\cm$.
This means $\cH_{\al_\cm}^{\rk{0}}\subset\hat\cH_\cm$.
From this it follows
\beql{E7.52-1031}
\cH_{\al_\cm}^{\rk{0}}
\subset\cH_{\hat\al_\cm}^{\rk{0}},
\eeq
where $\cH_{\hat\al_\cm}^{\rk{0}}$ is the inner subspace of the redundant part of the colligation $\hat\al_\cm$.

On the other side, it is obvious that
\[
\rstr{\hat\Phi_\cm^\ad}{\cH_{\hat\al_\cm}^{\rk{0}}}
=0.
\]
From this and \eqref{E7.41-1024} we get
\[
\cH_{\hat\al_\cm}^{\rk{0}}
\subset\cH_{\al_\cm}^{\rk{0}}.
\]
Together with \eqref{E7.52-1031} this implies
\[
\cH_{\al_\cm}^{\rk{0}}
=\cH_{\hat\al_\cm}^{\rk{0}}.\qedhere
\]
\eproof

\subsection{Two appendices on dissipative operators}\label{subsec7.5-0503}
\subsubsection{Dissipative colligations}

\bdefnnl{\zitaa{MR0322542}{\cch{I}}}{D7.8-0519}
The colligation
\beql{E7.52-0503}
    \al
    \defeq\ocol{\cH}{\cG}{A}{\Phi}{J}
\eeq
is called \emph{dissipative}, if $J=\IG$.
\edefn

The fundamental operator of a dissipative colligation is dissipative because for an arbitrary vector $h\in\cH$ it holds
\[
    \ipab{\frac{A-A^\ad}{\iu}h}{h}
    =\ipa{\Phi^\ad\Phi h}{h}
    =\norm{\Phi h}^2
    \geq0.
\]
The converse assertion is not true.
Using the method presented in \rthm{T2.2}, one can easily construct a nondissipative colligation with dissipative fundamental operator.
However, setting in the proof of this theorem $\cG\defeq\clo{\ran{\im A}}$, we come to the following assertion.

\blemnl{\zitaa{MR0322542}{\cch{I}}}{L7.9-0503}
Each dissipative operator can be embedded in a dissipative colligation.
\elem

\bthmnl{\zitaa{MR0322542}{\cch{I}}}{T7.10-0503}
The channel subspace of a dissipative colligation $\al\defeq\ocol{\cH}{\cG}{A}{\Phi}{\IG}$ coincides with the subspace of non-Hermiteness $\clo{\ran{\im A}}$.
\ethm
\bproof
If $h\perp\ran{\im A}$, then
\[
    \norm{\Phi h}^2
    =\ip{\Phi^\ad\Phi h}{h}
    =\ip{\im A h}{h}
    =0,
\]
and, consequently, $h\perp\ran{\Phi^\ad}$.
Hence $\clo{\ran{\Phi^\ad}}\subseteq\ran{\im A}$.
This leads together with \eqref{E10.1} to the identity $\clo{\ran{\Phi^\ad}}=\clo{\ran{\im A}}$.
\eproof

\bcornl{\zitaa{MR0322542}{\cch{I}}}{C7.11-0503}
Necessary and sufficient for the simplicity of a dissipative colligation is that its fundamental operator is completely non-selfadjoint.
\ecor

From this and \rcor{C2.20-0519} it follows:

\bcornl{\zitaa{MR0322542}{\cch{I}}}{C7.12-0503}
Let $P'$ be the orthoprojection onto the invariant subspace $\cH'$ of the completely non-selfadjoint dissipative operator $A$.
Then the operator $\rstr{P'A}{\cH'}$ is completely non-selfadjoint.
\ecor

We note that in \zitaa{MR0322542}{\cch{I}} there is given a description of the class of characteristic functions of dissipative colligations.

\subsubsection{Volterra operators with one-dimensional imaginary part}
A completely continuous operator is called \emph{Volterra operator}, if its spectrum consists only of the point zero.
Let in the space $\cH$ be given a completely non-selfadjoint Volterra operator $A$
with a one-dimensional imaginary part:
\beql{E7.53-0503} 
\frac{1}{\iu}\rk{A-A^\ad}
    =\ell\ipa{h}{h_0}h_0,
    \qquad h\in\cH,\,\norm{h_0}=1,\;\ell\neq0.
\eeq
An example of such operator is the integral operator \eqref{E3.29-0801} from \rsubsec{subsec3.5-0912}.
If $\ell>0$, then the operator $A$ is dissipative.
Otherwise, the operator $-A$ is dissipative.
In this section we assume $\ell>0$.

It holds the following important assertion.

\bthmnl{\cite{MR0062955}}{T7.13-0503}
Each completely non-selfadjoint dissipative Volterra operator $A$ with one-dimensional imaginary part \eqref{E7.53-0503} is unitarily equivalent to the integration operator
\beql{E7.54-0503}
    \rk{\cI f}\rk{x}
    =\iu\int_x^\ell f\rk{t}\dif t
\eeq
in $L_2(0,\ell)$.
\ethm
\bproof
Let $\tau\in\C$ and $\tau=1$.
We define the operator $\Phi\in\ek{\cH,\C}$ by setting
\beql{E7.55-0503}
    \Phi h
    =\sqrt\ell\ipa{h}{h_0}\tau,
    \qquad h\in\cH.
\eeq
Then
\beql{E7.56-0503}
    \Phi^\ad \tau
    =\sqrt\ell h_0.
\eeq
From \eqref{E7.53-0503} and \eqref{E7.54-0503}, \eqref{E7.55-0503} it follows
\[
    \Phi^\ad\Phi h
    =\ell\ipa{h}{h_0}h_0
    =\im Ah,
    \qquad h\in\cH.
\]
Hence, the quintuple
\beql{E7.57-0503}
    \al
    \defeq\ocol{\cH}{\C}{A}{\Phi}{\IC}
\eeq
is a dissipative colligation.
From the complet non-selfadjointy of the operator $A$ it follows the simplicity of the colligation $\al$.

In the considered case from \eqref{E7.55-0503} and \eqref{E7.56-0503} we get
\beql{E7.58-0503}
    \Phi\Phi^\ad
    =\ip{\Phi\Phi^\ad \tau}{\tau}
    =\norm{\Phi^\ad \tau}^2
    =\ell.
\eeq

The function $S_\al\rk{z}$ belongs to the class $\Om_{\IC}\rk{\C}$ and, as it follows from \rthm{T6.6A-0925}, it admits the representation \eqref{E6.35-0925} with $r=1$.
We note that in view of \rcor{C7.12-0503} zero can not be an eigenvalue of the operator $A$.
Moreover, from the dissipativity and the Volterra property of the operator $A$ it follows, that in the representation \eqref{E6.35-0925} the discrete product is missing 
and for each $t\in[0,\ell]$:
$\al\rk{t}=0, \ \xi\rk{t}=1$.
So, the representation \eqref{E6.35-0925} for the function $S_\al\rk{z}$ has the form\ednote{$\ec^{\iu\frac{\ell}{z}}$?}
\beql{E7.58A-0503}
    S_\al\rk{z}
    =\ec^{\iu\frac{\ell}{z}}.
\eeq
From \rthm{T7.5-1031} it follows that in the given case the model space $\cH_\cm$ (see \eqref{E7.38-1024}) is the space $L^2[0,\ell]$, whereas the model operator $A_\cm$ (see \eqref{E7.39-1024}) is the integration operator $\cI$.
The assertion of the theorem follows now from \rcor{C7.6-1031} and \rthm{T7.7-1031}.
Hence, the model colligation $\al_\cm$ (see \eqref{E7.43-1024}) coincides in the given case with the colligation $\hat{\al}$ (see \eqref{E3.34-0429}).
\eproof

We note that the unitarily equivalence of the colligations $\al$ and $\hat{\al}$ follows immediately from \eqref{E7.58A-0503}, \eqref{E3.36-0505}, the simplicity of the colligations $\al$, $\hat{\al}$, and \rthm{T3.4}.

\bcornl{\cite{MR0062955}}{C7.14-0519}
Let the fundamental operator $A$ of the operator colligation $\al$ be a completely non-selfadjoint dissipative Volterra operator with one-dimensional imaginary part of the form \eqref{E7.53-0503}.
Then
\[
    S_\al\rk{z}
    =\ec^{\iu\frac{\ell}{z}}.
\]
\ecor

The following assumption\ednote{result?} leads to an important class of unicellular operators.

\bthml{T7.14-0519}
Each completely non-selfadjoint Volterra operator with one-dimensional imaginary part is a unicellular operator.
\ethm

In view of the M.~S.~Liv\v{s}ic theorem~\ref{T7.13-0503} each such operator $A$ with one-dimensional imaginary part of the form \eqref{E7.53-0503} with $\ell>0$ is unitarily equivalent to the operator $\cI$ of integration (see \eqref{E7.54-0503}).
If $\ell<0$, then this property has the operator $-A$.

For this reason, \rthm{T7.14-0519} is equivalent to the following statement:

\bthmnl{\cite{MR0094717,MR0090760,MR140893}}{T7.15-0519}
The integration operator $\cI$ (see \eqref{E7.54-0503}) is a unicellular operator.
The set of its invariant subspaces coincides with the set of subspaces $\{\cL_\si\}_{0\leq\si\leq\ell}$, where $\cL_\si$ is the set of all functions $f\rk{t}\in L_2(0,\ell)$, which vanish almost everywhere in $[\si,\ell]$.
\ethm
\bproof
Let $L_1\subset L_2(0,\ell)$ be some invariant subspace of the operator $\cI$.
We consider the colligation $\hat{\al}$ (see \eqref{E3.34-0429})
\[
    \hat{\al}
    =\ocol{L_2\rk{0,\ell}}{\C}{\cI}{\Phi}{\Iu{\C}},
\]
where the operator $\Phi\in\ek{L_2\rk{0,\ell},\C}$ is defined by \eqref{E3.34-0514} and let the colligation 
\[
    \al_1
    =\ocol{L_1}{\C}{\cI_1}{\Phi_1}{\Iu{\C}}
\]
be the projection of the colligation $\hat{\al}$ onto the subspace $L_1$, \tie\ (see \rsubsec{subsec2.3.3-0519})
\[
    \cI_1\defeq\rstr{\cI}{L_1},
    \qquad\Phi_1\defeq\rstr{\Phi}{L_1}.
\]
Obviously, the operator $\cI_1$ has the Volterra property and is dissipative.
Since the operator $\cI$ is completely non-selfadjoint then in view of \rcor{C7.12-0503}, the operator $\cI_1$ is completely non-selfadjoint, too.
But then in view of \rcor{C7.14-0519}
\[
    S_{\al_1}\rk{z}
    =\ec^{\iu\frac{\ell_1}{z}},
\]
where (see \eqref{E7.58-0503})
\[
    \ell_1
    =\norm{\Phi_1^\ad \tau}^2
    =\norm{P_{L_1}\Phi^\ad \tau}^2
    \leq\norm{\Phi^\ad \tau}^2
    =\ell.
\]
Here $P_{L_1}$ is the orthoprojection in $L_2\rk{0,\ell}$ onto $L_1$, $\tau\in\C$ and $\tau=1$.

Let $\hat{\al}_\si$ be the projection of the colligation $\hat{\al}$ onto the invariant subspace $\cL_\si$ with respect to the operator $\cI$, $0\leq\si\leq\ell$:
\[
    \hat{\al}_\si
    \defeq\ocol{\cL_\si}{\C}{\cI_\si}{\Phi_\si}{\Iu{\C}}.
\]
Obviously
\[
    S_{\hat{\al}_\si}\rk{z}
    =\ec^{\iu\frac{\si}{z}}.
\]
Hence,
\[
S_{\al_1}\rk{z}
=S_{\hat{\al}_{\ell_1}}\rk{z},
\]
and the functions $S_{\al_1}\rk{z}$ and $S_{\hat{\al}_{\ell_1}}\rk{z}$ are regular right divisors of the function $S_{\hat{\al}}\rk{z}$.
In view of \rthm{T4.14-0711} it holds $L_1=\cL_{\ell_1}$ and
this means $\al_1=\hat{\al}_{\ell_1}$.
\eproof

\bcorl{C7.17-0525}
Necessary and sufficient for the denseness of the linear hull of the sequence of functions
\[
    \cI^nf(x)\qquad(f(x)\in L_2(0,\ell),\;\cI f(x)=\iu\int_x^\ell f(t)\dif t,\;n=0,1,2\dotsc)
\]
in $L_2(0,\ell)$ is that for arbitrary $\epsilon>0$ the set of all points $x$ of the interval $[\ell-\epsilon,\ell]$ for which $f(x)\neq0$ has a positive measure.
\ecor

\subsubsection{Completeness criterion for the finite-dimensional invariant subspaces of a dissipative operator with finite-dimensional imaginary part}\label{subsubsec7.5.3-0525}

Let $A\in\ek{\cH}$ be a dissipative operator with finite-dimensional imaginary part.
We embed the operator $A$ in a simple colligation (see \rthm{T2.2}, \rrem{R2.3-0929}, \rlem{L7.9-0503} and \rthm{T7.10-0503})
\[
    \al
    \defeq\ocol{\cH}{\C^r}{A}{\Phi}{\Iu{r}},
\]
where $\Iu{r}\defeq\Iu{\C^r}$ and $r\defeq\dim\clo{\ran{\im A}}$.
The function $S_\al(z)$ belongs to the class $\Omega_{\Iu{r}}\rk{\C^r}$ and, as it follows from \rthm{T6.6A-0925}, it admits the representation \eqref{E6.35-0925}.

Taking into account \eqref{E6.25-0925}, from identity \eqref{E6.36-0925} we find
\[
    \sum_{k=1}^\infty\tr\eta_k\eta_k^\ad+\ell
    =\tr\Phi\Phi^\ad.
\]
From this, the dissipativity of the operator $A$ and identity \eqref{E6.30-0925} it follows
\[\begin{split}
    \sum_{k=1}^\infty\im\la_k
    &=\frac{1}{2}\sum_{k=1}^\infty\eta_k\eta_k^\ad
    =\frac{1}{2}\sum_{k=1}^\infty\tr\eta_k^\ad\eta_k\\
    &=\frac{1}{2}\tr\Phi\Phi^\ad-\frac{\ell}{2}
    =\frac{1}{2}\tr\Phi^\ad\Phi-\frac{\ell}{2}
    =\tr\frac{A-A^\ad}{2\iu}-\frac{\ell}{2},
\end{split}\]
where $\{\la_k\}_{k=1}^\infty$ are the non-real eigenvalues of the operator $A$.
Hence,
\[
    \tr\im A
    =\sum_{k=1}^\infty\im\la_k+\frac{\ell}{2},
\]
\tie\
\beql{E7.61-0525}
    \sum_{k=1}^\infty\im\la_k
    \leq\tr\im A.
\eeq
If there is equality in this inequality then we have $\ell=0$ and in the triangular model of the operator $A$ the continuous part $\hat{A}_\cm$ is missing, \tie\ $\cH_\cm=\hat{\cH}_\cm$.

Consequently, we have the following statement.

\bthmnl{\cite{MR0062955}}{T7.18-0525}
Let $\cH$ be a Hilbert space and let $A\in\ek{\cH}$ be a dissipative operator with finite-dimensional imaginary part.
Then it holds the inequality \eqref{E7.61-0525}.
The system of finite-dimensional invariant subspaces of the operator $A$ is complete in the space $\cH$ if there is equality in \eqref{E7.61-0525}.
\ethm

We note that \rthm{T7.18-0525} is true for all completely non-selfadjoint dissipative operators with nuclear imaginary part, \tie, such operators $A$, for which the condition $\tr\im A<+\infty$ is satisfied (\cite{MR0062955}).

%
%
%
%


\subsection*{Acknowledgment}
The first author would like to express special thanks to Professor Tatjana Eisner for her generous support.
He also thanks the Max Planck Institute for Human Cognitive and Brain Sciences, in particular 
Professor Christian Doeller, Dr Christina Schroeder, and Dr Sebastian Ziegaus.

The first author was also partially supported by the Volkswagen Foundation grant within the frameworks of the international project ``From Modeling and Analysis to Approximation''.

\end{document}